\documentclass[12pt]{amsart}
\usepackage{graphicx}
\usepackage{amssymb}
\usepackage{amsmath}
\usepackage{amsfonts}
\usepackage[new]{old-arrows}
\usepackage{geometry}
\usepackage{comment}
\usepackage{hyperref}
\usepackage{cleveref}
\usepackage{thmtools} 
\usepackage{thm-restate}
\usepackage{cancel}
\usepackage{enumitem}
\usepackage{fullpage}

\usepackage[T1]{fontenc}
\usepackage[utf8]{inputenc}

\newtheorem{thm}{Theorem}[section]

\newtheorem{question}[thm]{Question}

\newtheorem{example}[thm]{Example}

\newtheorem{remark}[thm]{Remark}
\newtheorem{lemma}[thm]{Lemma}

\newtheorem{prop}[thm]{Proposition}
\newtheorem{definition}[thm]{Definition}

\newenvironment{inequality}
  {\crefalias{equation}{inequality}\begin{equation}}
  {\end{equation}\ignorespacesafterend}
\crefname{inequality}{inequality}{inequalities}
\Crefname{inequality}{Inequality}{Inequalities}

\newcommand{\Z}{\mathbb{Z}}
\newcommand{\N}{\mathbb{N}}
\newcommand{\C}{\mathbb{C}}
\newcommand{\Q}{\mathbb{Q}}

\newcommand{\Wh}{\mathrm{SC}}
\newcommand{\Out}{\mathrm{Out}}
\newcommand{\Mod}{\mathrm{Mod}}
\newcommand{\Hom}{\mathrm{H}}
\newcommand{\W}{\mathrm{HSC}}
\newcommand{\Csep}{\mathcal{C}^{\mathrm{sep}}}
\newcommand{\rose}{\mathbf{R}}
\newcommand{\free}{\mathbf{F}}
\newcommand{\sep}{^{\mathrm{sep}}}
\newcommand{\ffc}{\mathrm{FF}}
\newcommand{\Aut}{\mathrm{Aut}}
\newcommand{\lab}{\mathrm{lab}}
\newcommand{\pathwh}{\mathcal{P}\mathrm{SC}}
\newcommand{\Cay}{\mathrm{Cay}}

\newcommand{\seplength}[1]{||#1||_{\mathrm{sep}}}
\newcommand{\dsep}[2]{d_{\mathrm{sep}}\Big(#1, #2\Big)}
\newcommand{\pant}{^{\mathrm{pant}}}
\newcommand{\nf}{^{\mathrm{nf}}}
\newcommand{\scc}{^{\mathrm{scc}}}
\newcommand{\prim}{^{\mathrm{prim}}}
\newcommand{\Inn}{\mathrm{Inn}}
\newcommand{\fbc}{\mathrm{FB}}

\newcommand{\curves}{\mathcal{C}}
\newcommand{\area}{\mathrm{area}}
\newcommand{\words}{\mathcal{W}}

\newcommand{\inv}[1]{#1^{-1}}

\title[Separable homology of graphs and the separability complex]{Separable homology of graphs and \\ the separability complex}
\author{Becky Eastham}

\begin{document}

\newpage
\maketitle

\begin{abstract}
   We introduce the \emph{separability complex}, a one-complex associated to a finite regular cover of the rose and show that it is connected if and only if the fundamental group of the associated cover is generated by its intersection with the set of elements in proper free factors of $\free_n$. The separability complex admits an action of $\Out(\free_n)$ by isometries if the associated cover corresponds to a characteristic subgroup of $\free_n$. We prove that the separability complex of the rose has infinite diameter and is nonhyperbolic, implying it is not quasi-isometric to the free splitting complex or the free factor complex. As a consequence, we obtain that the Cayley graph of $\free_n$ with generating set consisting of all primitive elements of $\free_n$ is nonhyperbolic.  
\end{abstract}  

\section{Introduction}

Let $\Sigma=\Sigma_{g, n}$ be an oriented surface of finite type, $G$ an abelian group, and $\mathcal{C}$ a collection of closed curves on $\Sigma$. For a curve $\gamma\subset\Sigma$, we denote by $[\gamma]$ the associated element of $\Hom_1(\Sigma; G)$. For a cover $\pi: \widetilde{\Sigma}\rightarrow \Sigma$, let $\mathrm{H}_1^{\mathcal{C}}(\widetilde{\Sigma}; G)$ be the subspace of $\mathrm{H}_1(\widetilde{\Sigma}; G)$ spanned by the following set:
\[\{[\gamma]\in \mathrm{H}_1(\widetilde{\Sigma}; G) \mid \gamma \text{ is a closed curve in }\widetilde\Sigma, \pi(\gamma)\in\curves\}.\] 

There has been significant recent interest (see \cite{Koberda_2016}, \cite{farb2016finite}, \cite{malestein2019simple}, \cite{lee2020graph}, \cite{boggi2023generating}, \cite{klukowski2023simple}) in questions of the following form: for a collection of curves $\mathcal{C}\subset \pi_1(\Sigma)$, does $\Hom_1^{\mathcal{C}}(\widetilde{\Sigma}; G)= \Hom_1(\widetilde{\Sigma}; G)$ for all finite regular covers $\widetilde{\Sigma}$ of $\Sigma?$ Since the fundamental group of a punctured surface is free, there are analogous questions regarding the homology of finite regular covers of graphs. The primary focus of this paper is the following closely related question.

\begin{question}\label{question:main question}
    Let $\free_n$ be the free group on $n$ generators, $N$ a finite-index normal subgroup of $\free_n$, and $\Csep$ the set of elements of $\free_n$ contained in proper free factors of $\free_n$.  Is $N$ always generated by $\Csep\cap N$?
\end{question}

\noindent Throughout the paper, we identify $\pi_1(\rose_n)$ with $\free_n$. Since finite regular covers of the rose correspond to finite-index normal subgroups of $\free_n$, the question above is equivalent to the following question: is $\pi_1(\Gamma)$ generated by $\pi_1(\Gamma)\cap\Csep$ for every finite regular cover $\Gamma\twoheadrightarrow\rose_n$? 

A positive answer to \Cref{question:main question} would imply that the integral homology of a finite regular cover of $\rose_n$ is generated by $\Csep,$ as $\Hom_1(\Gamma; \Z)$ is the abelianization of $\pi_1(\Gamma)$. It would also imply that $\Hom_1\sep(\widetilde{\Sigma}; \Z) = \Hom_1(\widetilde{\Sigma}; \Z)$ for $\widetilde{\Sigma}$ a finite regular cover of a punctured surface $\Sigma$. Importantly, \Cref{question:main question} is related to the congruence subgroup problem for mapping class groups. For a statement and history of the congruence subgroup problem, see Chapter 4 of \cite{farb2006problems}. Boggi \cite{boggi2006profinite, Kent2016} has reduced the congruence subgroup problem for the mapping class group of $\Sigma_{g, n}$ to simple--connectivity of a certain \textit{procongruence curve complex} associated to $\Sigma_{g, n}$. Using her work in \cite{kent2009trees}, Kent can show that, if $\pi_1(\widetilde \Sigma_{0, 5})$ is generated by the nonfilling elements of $\pi_1(\Sigma_{0, 5})$, then the procongruence curve complex of $\Sigma_{2,0}$ is simply--connected, giving a conditional solution to the congruence subgroup problem in genus two \cite{Kent2023} \footnote{This has been announced by Boggi \cite{Boggi-announced}, but the complete proof has not appeared.}. 
The set $\Csep$ properly contains the set of nonfilling curves on a punctured surface, so a positive answer to \Cref{question:main question} does not imply the congruence subgroup problem for the mapping class group of the closed surface of genus two. However, it would be a partial result toward this goal.  

We provide a reformulation of \Cref{question:main question} in terms of the connectivity of spaces $\Wh(\Gamma)$ naturally associated to finite regular covers $\Gamma \rightarrow \rose_n$. 

\begin{restatable}{prop}{equivalentcondition}
\label{thm: equiv-connectivity}
    Given a finite-index normal subgroup $N$ of $\free_n$, let $\Gamma$ be the finite regular cover of $\rose_n$ corresponding to $N.$ Then $\Wh(\Gamma)$ is connected if and only if  $N$ is generated by $N\cap\Csep$.
\end{restatable}

\noindent Thus, if $\Wh(\Gamma)$ is connected for all finite regular covers $\Gamma,$ the answer to \Cref{question:main question} is \emph{yes}. Note that we are identifying the fundamental group of $\Gamma$ with its image inside $\free_n$. If $\Gamma$ represents a characteristic subgroup of $\free_n$, the complex $\Wh(\Gamma)$ admits a natural action of $\Out(\free_n)$ by isometries (see \Cref{theorem:outFn-action-char}). Each $\Wh(\Gamma)$ is contained in $\Wh(\rose_n)$, which is connected and admits an isometric action of $\Out(\free_n)$ (see \Cref{theorem:isometric-action}). Much of this paper discusses the topological and coarse geometric properties of $\Wh(\Gamma)$.

\begin{restatable}{thm}{infinitediameter}\label{thm:infinite-diameter}
Let $\Gamma$ be a finite regular cover of $\rose_n$. The diameter of every component of $\Wh(\Gamma)$ is infinite.
\end{restatable}

\noindent Let $\curves\prim$ be the set of all primitive elements of $\free_n$. Since $\Wh(\rose_n)$ is a quotient of $\Cay(\free_n, \Csep)$, and $\Cay(\free_n, \Csep)$ is quasi-isometric to $\Cay(\free_n, \curves\prim)$, we obtain the following corollary.

\begin{restatable}{cor}{cayinfdiameter}\label{cor:Cay-inf-diameter}
    The graphs $\Cay(\free_n, \Csep)$ and $\Cay(\free_n, \curves\prim)$ have infinite diameter.
\end{restatable}

Note that, after this work was completed, Putman informed the author that \Cref{cor:Cay-inf-diameter} was shown in a 2003 paper \cite{bardakov2003palindromic} of Bardakov, Shpilrain, and Tolstych.  However, the proof of \Cref{cor:Cay-inf-diameter} provided here is different from the proof in \cite{bardakov2003palindromic}. We also show that $\Wh(\rose_n)$ is not quasi-isometric to many of the other spaces that admit an isometric action of $\Out(\free_n)$, including the free factor complex and the free splitting complex, because $\Wh(\rose_n)$ is nonhyperbolic.  

\begin{restatable}{thm}{nonhyperbolicity}
\label{theorem:nonhyperbolicity}
$\Wh(\rose_n)$ is nonhyperbolic.
\end{restatable}

\noindent In addition, we obtain the following corollary. 

\begin{restatable}{cor}{caynonhyperbolicity}\label{cor:Cay-prim-nonhyperbolicity}
    The Cayley graphs $\Cay(\free_n, \Csep)$ and $\Cay(\free_n, \curves\prim)$ are nonhyperbolic.
\end{restatable}

We prove \Cref{theorem:nonhyperbolicity} and \Cref{cor:Cay-prim-nonhyperbolicity} by finding explicit infinite sets of arbitrarily fat geodesic triangles in $\Wh(\rose_n)$ and $\Cay(\free_n, \Csep)$. Our techniques are elementary and directly analyze the structure of the graphs $\Wh(\rose_n)$ and $\Cay(\free_n, \Csep).$ 

The recent preprint \cite{Chu25infinite} recovers \Cref{cor:Cay-prim-nonhyperbolicity} and expands the result to other Cayley graphs of groups with conjugation-invariant generating sets.  The proof of \Cref{cor:Cay-prim-nonhyperbolicity} in this paper and the proof in \cite{Chu25infinite} are quite different: the proof in \cite{Chu25infinite} uses machinery developed in \cite{brandenbursky2016cancellation} and \cite{brandenbursky2019aut} related to $\Aut$-invariant norms on certain groups. 

There are other known complexes admitting isometric actions of $\Out(\free_n)$ which are nonhyperbolic, including Culler-Vogtmann's Outer space (introduced in \cite{culler1986moduli}) and the edge splitting complex (proven to be nonhyperbolic by Sabalka and Savchuk in \cite{sabalka2014geometry}). The former is an analog of Teichm\"uller space for free groups, and the latter was a proposed analog of the curve complex until it was shown to be nonhyperbolic. Interestingly, though the author was made aware of Sabalka and Savchuk's paper after this work was completed, the techniques used in \cite{sabalka2014geometry} to prove that the edge splitting complex is nonhyperbolic are similar to those used here. Both proofs utilize \emph{Van Kampen diagrams} (see \Cref{sec: Van Kampen diagrams}) and a theorem of Stallings which gives an algorithm to determine whether an element of $\free_n$ is separable.

In \cite{farb2016finite}, Farb and Hensel consider a finite graph $X$ and the set of primitive elements of $\free_n$. They ask if $\Hom_1^{\mathrm{prim}}(\Gamma; \C) = \Hom_1(\Gamma; \C)$ for every finite regular cover $\pi: \Gamma\rightarrow X$. They prove that, if the deck group of $\pi$ is abelian or two-step nilpotent, the answer is \emph{yes}. However, in general, the answer is \emph{no}. Malestein and Putman state in Theorem C of \cite{malestein2019simple} that if $n\geq 2$ and $\mathcal{C}$ is contained in the union of finitely many $\Aut(\free_n)$-orbits, there are finite-index normal subgroups $R\triangleleft \free_n$ such that $\Hom_1^{\mathcal{C}}(R; \Q)\neq \Hom_1(R; \Q).$ Since the set of all primitive elements of $\free_n$ is itself an $\Aut(\free_n)$-orbit, this implies that the primitive homology of a regular cover $\Gamma$ of $\rose_n$ is not necessarily finite-index in the full homology of $\Gamma$. Later, using GAP \footnote{Groups, Algorithms, Programming -
a System for Computational Discrete Algebra}, Lee, Rosenblum-Sellers, Safin, and Tenie \cite{lee2020graph} found covers $\Gamma$ of relatively low degree compared to those of \cite{malestein2019simple} which satisfy $\Hom_1\prim(\Gamma; \C)\neq\Hom_1(\Gamma; \C).$ The deck groups of these covers are rank-$3$ $p$-groups for $p=2$ of order at most 512.

Let $\mathcal{C}\scc$ be the set of simple closed curves on a surface $\Sigma$ with genus $g$ and $n$ punctures (possibly zero). For $\Sigma$ with large enough complexity, i.e. $(g, n)\not\in\{(0, 0), (0, 1), (0, 2), (1, 0)\}$, Koberda and Santharoubane \cite{Koberda_2016} give examples of finite covers $\widetilde{\Sigma}\rightarrow\Sigma$ satisfying \[\Hom_1\scc(\widetilde{\Sigma}; \Z)\neq\Hom_1(\widetilde{\Sigma}; \Z)\] using TQFT representations. Furthermore, Theorem C of \cite{malestein2019simple} implies that, if $\Sigma$ is a punctured surface of genus at least two, there are finite regular covers $\widetilde{\Sigma}$ of $\Sigma$ satisfying $\Hom_1^{\mathrm{scc}}(\widetilde{\Sigma}; \Q)\neq \Hom_1(\widetilde{\Sigma}; \Q)$.  This is because $\mathcal C\scc$ is contained in finitely many $\Aut(\free_{2g+n-1})$-orbits. The same is true for the set of curves on $\Sigma$ with at most $k$ self-intersections.  Recently, Klukowski (\cite{klukowski2023simple}) showed that for any surface $\Sigma=\Sigma_{g, n}$ satisfying $g\geq 2$ and $n\geq 0$, there is a finite 
cover $\widetilde{\Sigma}$ of $\Sigma$ such that $\Hom_1\scc(\widetilde{\Sigma}; \Q)\neq \Hom_1(\widetilde{\Sigma}; \Q).$

In 2008, while working on the congruence subgroup problem for $\Mod(\Sigma_g)$, Kent asked the following question, which remains open despite the closely related work described above.

\begin{question}\cite{Kent2012} \label{question:Kent}
    Let $\Sigma$ be a compact oriented surface, $\mathcal{C}^{\mathrm{nf}}$ the set of all nonfilling curves on $\Sigma$, and $\pi:\widetilde{\Sigma}\rightarrow \Sigma$ a finite regular cover.  Is $\Hom_1\nf(\widetilde{\Sigma}; \Z) = \Hom_1(\widetilde{\Sigma}; \Z)?$
\end{question}

\noindent Using Boggi's original approach to the congruence subgroup problem (see \cite{boggi2006profinite}) and the ideas of Kent mentioned previously, a positive answer to this question would also give a proof of the congruence subgroup problem in genus two. Recently, Boggi, Putman, and Salter \cite{boggi2023generating} made significant progress on \Cref{question:Kent}: they showed that there is a set of nonfilling curves on $\Sigma$--namely, the set $\mathcal{C}\pant$ of curves contained in pairs of pants--which generates the rational homology of every finite branched cover of a closed oriented surface.  This does not fully answer \Cref{question:Kent}, but it does imply that $\Hom_1\nf(\widetilde{\Sigma}; \Z)$ is finite-index in $\Hom_1(\widetilde{\Sigma}; \Z)$ for any finite cover $\widetilde{\Sigma}$ of a closed oriented surface $\Sigma$.  

We say an element of $\free_n$ is \emph{separable} if it is contained in a proper free factor of $\free_n$. Note that the set of nonfilling curves on a punctured surface $\Sigma$ is contained in the set of separable curves on $\Sigma$. Therefore, if $\Hom_1\sep(\Gamma; \Z)\neq\Hom_1(\Gamma; \Z),$ $\Hom_1\nf(\widetilde{\Sigma}; \Z)\neq\Hom_1(\widetilde{\Sigma}; \Z).$ A weaker version of \Cref{question:main question} is below.

\begin{question}\label{question: homology-question}
    Is $\Hom_1(\Gamma; \Z)$ generated by $\Csep$ for every finite regular cover $\Gamma\rightarrow \rose_n?$
\end{question}

\noindent We give the following reformulation of this question.

\begin{restatable}{prop}{connectedhomology}
\label{connectedhomology}
    Given a finite-index regular cover $\Gamma$ of $\rose_n$, there is a naturally defined space $\W_1(\Gamma)$ which is connected if and only if $\Hom_1\sep(\Gamma; \Z) = \Hom_1(\Gamma; \Z).$
\end{restatable}

\noindent Finally, using the result of Boggi, Putman, and Salter mentioned above, we obtain the following.

\begin{restatable}{prop}{finitelyManyComponents}\label{thm:W1-has-finitely-many-components}
    Let $\Gamma\rightarrow\rose_n$ be a finite regular cover.  Then $\W_1(\Gamma)$ has finitely many components.
\end{restatable}

\subsection*{Outline} 
In \Cref{sec: background}, we describe the connection between graphs and free groups (\Cref{sec: graphs and free groups}), state some important facts about Van Kampen diagrams (\Cref{sec: Van Kampen diagrams}), and recall an algorithm of Whitehead which can determine whether $w\in\free_n$ is separable (\Cref{sec: Whitehead's algorithm}). In \Cref{sec:whiteheadspace}, we define the separability complex (\Cref{def:whiteheadspace}) and prove \Cref{thm: equiv-connectivity}.  In \Cref{sec: inf-diameter} we prove \Cref{thm:infinite-diameter} and \Cref{cor:Cay-inf-diameter}.  We dedicate \Cref{sec: nonhyperbolicity} to the proof of \Cref{theorem:nonhyperbolicity} and \Cref{cor:Cay-prim-nonhyperbolicity}. \Cref{sec: homology version} contains a proof of \Cref{connectedhomology} and \Cref{thm:W1-has-finitely-many-components}. Finally, \Cref{sec: appendix} contains the proofs of certain lemmas and propositions for small values of $n$.

\subsection*{Acknowledgments}
The author would like to thank her advisors, Autumn Kent and Marissa Loving, for their invaluable help on this project. She would also like to thank Caglar Uyanik for many helpful conversations. She is grateful to Alex Wright for his comments on \Cref{cor:Cay-prim-nonhyperbolicity}. In particular, Wright pointed out that $\Cay(\free_n, \Csep)$ and $\Cay(\free_n, \curves\prim)$ are quasi-isometric. Alex Wright also pointed her to \cite{sabalka2014geometry}, wherein the authors prove that the edge splitting complex is nonhyperbolic. She also thanks Andrew Putman for pointing her to \cite{bardakov2003palindromic}, in which the authors prove that the primitive width of the free group is infinite. She thanks 
Eleanor Rhoads for helpful conversations. She also thanks Adam Klukowski, Jean Pierre Mutanguha, Nick Salter, and Alex Wright for helpful comments on a draft of this paper. She thanks Edgar Bering for helpful discussions around \Cref{question:which_subroup_stabilizes_edges}. 

This work is supported in part by NSF Grants DMS-2202718, RTG DMS-2230900,  and DMS-1904130.

\section{Background}\label{sec: background} 
Here we cover some prerequisite information for the rest of the paper.  There is information about graphs and free groups in \Cref{sec: graphs and free groups}.  In \Cref{sec: Van Kampen diagrams}, we discuss Van Kampen diagrams and annular diagrams, which are essential tools in \Cref{sec:whiteheadspace} and \Cref{sec: nonhyperbolicity}.  In \Cref{sec: Whitehead's algorithm}, we discuss an algorithm of Whitehead which decides whether $w\in\free_n$ is in a proper free factor of $\free_n.$

\subsection{Graphs and Free Groups} \label{sec: graphs and free groups}
\begin{definition}
    A \emph{directed graph} is a 1-complex given by a tuple $(V, E, \iota, \tau)$ where $V$ is a set of 0-cells called \emph{vertices}, $E$ is a set of 1-cells called \emph{edges},  and $\iota: E\rightarrow V$ and $\tau: E\rightarrow V$ are maps that send $e\in E$ to the endpoints of $e$ in $V$.  For all $e\in E$, $\iota(e)$ will be called the \emph{initial vertex} of $e$ and $\tau(e)$ will be called the \emph{terminal vertex} of $e$.  By assigning initial and terminal vertices of $e$, we assign an orientation to $e$.  We represent the orientation of $e$ by an arrow on $e$ pointing toward $\tau(e)$.  The \emph{valence} of a vertex $v$ is the number of components of $N(v)\setminus \{v\}$, where $N(v)$ is a small neighborhood of $v$.
\end{definition}
\noindent Note that, under our definition of graph, we allow multiple edges with the same initial and terminal vertices. Often we will want to label the edges or vertices with elements of some set. We will denote the label of an edge $e$ or vertex $v$ by $\lab(e)$ and $\lab(v)$, respectively. 

\begin{definition}
    A \emph{rose} is a graph with one vertex and one or more edges.  We denote the rose with $n$ edges by $\rose_n$.  
\end{definition}

Since our definition of the Cayley graph is slightly different than the definition that some authors use, we will define it below. Here and throughout the paper we identify vertices in $\Cay(G, S)$ with the corresponding group elements of $G.$

\begin{definition}\label{def:Cayley_graph}
    The \emph{Cayley graph} $\Cay(G, S)$ of a group $G$ with set $S$ of elements of $G$ is the directed graph whose vertices are in bijection with elements of $G$ and whose edges are labeled by elements of $S$.  There is an edge labeled $s\in S$ from the vertex $v\in\Cay(G, S)$ to the vertex $w\in\Cay(G, S)$ if $vs=w$ in $G$. 
\end{definition}

This definition differs from some other authors' definitions in that we do not require that $S$ be a generating set of $G$, so $\Cay(G, S)$ may not be connected. The graph $\Cay(G, S)$ is connected if and only if $S$ is a generating set for $G$.  

\begin{prop}
    The graph $\Cay(G, S)$ is connected if and only if $S$ generates $G$.
\end{prop}

\begin{proof}
    There is a path from $1$ to $g\in G$ if and only if $g$ is a product of elements of $S$.
\end{proof}

\begin{definition}
    Let $S\subset \free\langle X\rangle$ be a set of elements of the free group on the set $X$.  The \emph{symmetric closure} of $S$, denoted $S^*$, is the set of cyclic permutations of elements of $S$ and their inverses.  For example, if $S=\{abc, b\inv{a}\}$, then $S^*=\{abc, bca, cab, b\inv{a}, \inv{a}b, \inv{(abc)}, \newline \inv{(bca)}, \inv{(cab)}, a\inv{b}, \inv{b}a\}.$ In particular, $S^*$ is finite if $S$ is finite.
\end{definition}

Throughout the paper, we let $\free_n$ be the free group on $n$ generators \[X_n:=\{x_1, x_2, \ldots, x_n\}.\] A \emph{word} on $X_n^*$ is an element of the set \[ \mathcal{W}_n = \{w\in \displaystyle\prod_{i =1}^m X_n^*\mid m<\infty\}.\] Elements of $\free_n$ are equivalence classes of words under \emph{free reduction}.  A word is \emph{freely reduced} if it contains no subword of the form $y\inv{y}$ where $y\in X_n^*$. The freely reduced word $w=y_1y_2\ldots y_m$ where $y_i\in X_n^{*}$ is \emph{cyclically reduced} if $y_1\neq \inv{y_m}$. A \emph{cyclic subword} of a word $w$ is a subword of one of the cyclic permutations of $w$. The length of $w=y_1y_2\ldots y_m$ where $y_i\in X_n^*$ is denoted by $|w|$: in this case, $|w|=m$.  When we refer to the length of $v\in\free_n$, we mean the length of the freely reduced word representing $v$. For an automorphism $\varphi: \free_n\rightarrow\free_n$, there is a natural map $\Tilde{\varphi}: \mathcal{W}_n\rightarrow \mathcal{W}_n$ which extends $\varphi$: for $w=y_1\ldots y_k\in \mathcal{W}_n$ where $y_i\in X_n^*$, let $\Tilde{\varphi}(w)=\varphi(y_1)\varphi(y_2)\ldots\varphi(y_k).$  For example, if $\varphi$ is the map on $\free_2\langle a, b\rangle$ sending $a$ to $ba\inv{b}$ and $b$ to $b$, $\Tilde{\varphi}(ab)=ba\inv{b}b.$

Since $\pi_1(\rose_n) \cong \free_n$, we will assume throughout that the edges of $\rose_n$ are labeled by elements of $X_n$. A cover of $\rose_n$ will inherit this labeling, so that a graph $\Gamma$ is a cover of $\rose_n$ if and only if for every vertex $v$ in $\Gamma$, there is a neighborhood $U$ of $v$ consisting of $v$ and two edge segments incident to $v$ for each $x_i\in X_n$, oriented in opposite directions.  We identify $\pi_1(\rose_n)$ with $\free_n$ and, for $v$ a vertex of $\Gamma$,  we identify $\pi_1(\Gamma, v)$ with the corresponding subgroup of $\free_n$.

Every finite-index subgroup of $\free_n$ is represented by a finite cover of $\rose_n$.  Suppose $G<\free_n$ is a finitely generated subgroup of $\free_n$ that is not finite-index.  Then there is a finite graph called a \emph{core graph} such that every element of $G$ is represented by a closed path in $G$ with no backtracking. In order to define a core graph we need to introduce the notion of folding. 

\begin{definition}
    A \emph{Stallings folding}, or simply \emph{folding}, of a labeled, directed graph is an identification of edges $e_1, e_2$ that satisfy the following:
    \begin{enumerate}
        \item $\iota(e_1)=\iota(e_2)$ or $\tau(e_1)=\tau(e_2);$
        \item $\lab(e_1) = \lab(e_2).$
    \end{enumerate}

    \noindent If a labeled, directed graph has no pairs of edges satisfying the conditions above, we say the graph is \emph{completely folded}.
\end{definition}

See \Cref{fig: folding} for an example of a folding. Note that foldings induce surjections  on the corresponding fundamental groups: see \cite[Corollary 4.4]{stallings1983topology}.  

\begin{figure}
    \centering
    \includegraphics[scale=.7]{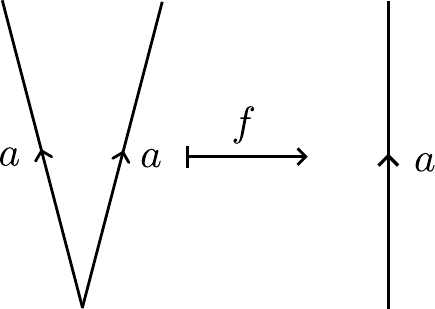}
    \caption{A Folding}
    \label{fig: folding}
\end{figure}

\begin{definition}
A \emph{core graph} is a finite completely folded graph with at most one valence-one vertex.
\end{definition}

To obtain a core graph representing a finitely-generated subgroup $G=\langle w_1, w_2, \ldots, w_k\rangle$ of $\free_n$, take the wedge sum of the $k$ subdivided cycles labeled by the generators of $G$, then completely fold. The wedge point is the basepoint of the graph. See \cite[Section 5.4]{stallings1983topology} for a proof of this fact. An example of obtaining a core graph via complete folding is shown in \Cref{fig: core graph} for $G=\langle ab\inv{a}, [a, b]\rangle$, where  $[a, b]=ab\inv{a}\inv{b}$. We will use this notation for commutators throughout the paper.

    \begin{figure}[ht]
        \centering
        \includegraphics[scale=.6]{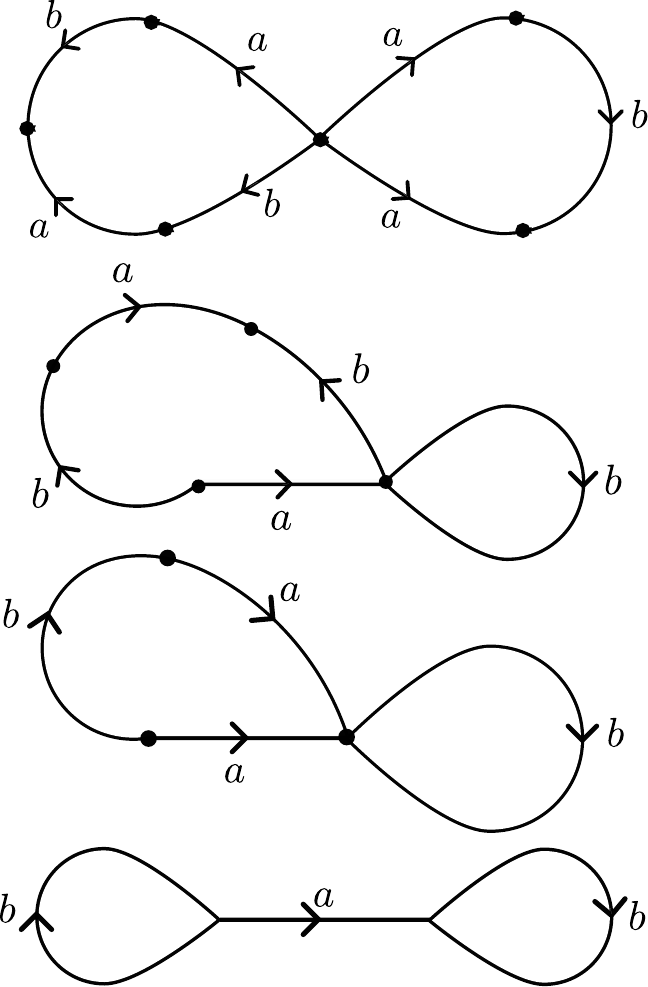}
        \caption{Folding to obtain a core graph for $G=\langle a, b\mid ab\inv{a}, [a, b]\rangle$}
        \label{fig: core graph}
    \end{figure} 

We will introduce the separability complex of the rose $\Wh(\rose_n)$ in \Cref{sec:whiteheadspace} and prove that it admits an isometric action of $\Out(\free_n)$. Thus, we include here a brief discussion of $\Out(\free_n)$ and some of the spaces on which it acts. $\Aut(\free_n)$ is the group of automorphisms of the free group. A theorem of Nielsen (see \cite[Theorem~3.2]{magnus2004combinatorial}) gives a nice set of generators for $\Aut(\free_n).$ 

\begin{thm}[Nielsen] \label{theorem: gens-for-Aut(Fn)}
Let $X_n=\{x_1, \ldots, x_n\}$ be a basis for $\free_n$.
$\mathrm{Aut}(\free_n)$ is finitely generated by the following  \emph{elementary Nielsen transformations}.
\begin{enumerate}
    \item Let $\sigma$ be a permutation of $\{1, 2, \ldots, n\}$ and $\epsilon_i\in\{\pm 1\}$ for $1\leq i\leq n$. A function $N_{\sigma}: \free_n\rightarrow\free_n$ satisfying
    \[
        N_{\sigma}(x_i)= x_{\sigma(i)}^{\epsilon_i}
    \] for all $i$ is an elementary Nielsen transformation.
    \item Let $i$ and $j$ satisfy $1\leq i, j\leq n$ and $i\neq j$. A function $N:\free_n\rightarrow\free_n$ fixing all $x_k$ for $k\neq i$ and sending $x_i$ to an element of $\{x_ix_j^{\epsilon}, x_j^{\epsilon}x_i, x_j^{\epsilon}x_ix_j^{-\epsilon}\}$ where $\epsilon\in\{\pm 1\}$ is also an elementary Nielsen transformation.
\end{enumerate}  
\end{thm}

\begin{remark}\label{remark:one-letter-primitive}
A straightforward corollary of \Cref{theorem: gens-for-Aut(Fn)} is that if the multiset of letters of $w\in\free_n$ contains only one element of $\{x, \inv{x}\}$ for some $x\in X_n$, then $w$ is primitive. To see this, observe that there is a composition of Nielsen automorphisms taking $x$ to $xu,$ where $u\in \free\langle X_n\setminus\{x\}\rangle.$ 

Alternatively, to see that $w$ is primitive, just observe that $xu\cup(X_n\setminus\{x\})$ generates the free group since it generates $x$. Since a basis for $F_n$ consists of $n$ primitive elements, $xu$ is primitive.
\end{remark}

For $n\geq 2$, $\Out(\free_n)$ acts on the \emph{free factor complex} $\ffc_n$ (first introduced by Hatcher and Vogtmann in \cite{hatcher1998complex}) and its 1-skeleton, the \emph{free factor graph}.  A theorem of Bestvina and Feign \cite{bestvina2014hyperbolicity} states that for $n\geq 2$, $\ffc_n$ is Gromov-hyperbolic (see also \cite[Theorem 6.1]{kapovich2014hyperbolicity} and \cite{hilion2017hyperbolicity}).  In addition, $\Out(\free_n)$ acts on the \emph{free bases complex}, denoted $\fbc_n$; by \cite[Proposition 4.3]{kapovich2014hyperbolicity}, $\fbc_n$ is quasi-isometric to $\ffc_n.$  $\Out(\free_n)$ also acts on the \emph{free splitting complex}, which admits a Lipschitz map to the free factor graph \cite[Remark 4.5]{kapovich2014hyperbolicity}.  The \emph{free splitting complex} was first shown to be hyperbolic by Handel and Mosher in \cite{handel2013free}.

\subsection{Van Kampen diagrams}\label{sec: Van Kampen diagrams}
Throughout the paper, we make use of Van Kampen diagrams, first introduced in 1933 by E. R. Van Kampen.  For a more complete discussion, see Chapter V of \cite{lyndon1977combinatorial}. Let $\mathcal{P} = \langle X\mid R\rangle$ be group presentation.  We will assume throughout the paper that elements of the relation set $R$ of a presentation $\mathcal{P}$ are cyclically reduced. 

\begin{definition}
    A \emph{disk diagram} $D$ is a finite, connected, simply-connected planar 2-complex with connected 1-skeleton and at most one vertex of degree one that is required to be on the boundary of $D$. The disks of $D$ are called \emph{regions}, and regions of disk diagrams have immersed boundaries. The \emph{area} of $D$ is the number of regions of $D$.  
\end{definition}
\noindent We note that we assume regions of disk diagrams have immersed boundaries throughout the paper because all disk diagrams in this paper will correspond to Van Kampen diagrams (\Cref{def:Van_Kampen_diagram}), and regions in Van Kampen diagrams can always be taken to have immersed boundaries (see \Cref{remark: disk gluing}).

\begin{definition}\label{def:Van_Kampen_diagram}
    Let $\mathcal{P} = \langle X\mid R\rangle$ be a presentation for a group, and recall we assume elements of $R$ are cyclically reduced.  A \emph{Van Kampen diagram} for $w\in\langle\langle R\rangle\rangle$ over $\mathcal{P}$ is a disk diagram $D$ with oriented edges labeled by elements of $X$. The 1-skeleton of $D$ is a folding of a wedge of finitely many completely folded cycles labeled by conjugates of elements of $R^*$. The basepoint of $D$ is the wedge point of the cycles. Boundaries of regions are attached along cycles labeled by elements of $R^*$ and are therefore immersed (see \Cref{remark: disk gluing}). The boundary of $D$ is labeled by $w$ beginning at the basepoint of $D$ and reading clockwise along the boundary of $D$ in the plane. A Van Kampen diagram $D$ for $w$ over $\mathcal{P}$ is \emph{minimal} if there is no Van Kampen diagram $D'$ for $w$ over $\mathcal{P}$ of smaller area. 
\end{definition}

\noindent Let $G$ be a quotient of $\free_n$, and let $\mathcal{P}=\langle X_n\mid R\rangle$ be a presentation for $G$.  An element $w\in\langle\langle R\rangle\rangle$ has a Van Kampen diagram over $\mathcal{P}$.  The converse is also true: if $w\not\in\langle\langle R\rangle\rangle$, then $w$ has no Van Kampen diagram over $\mathcal{P}$.

\begin{remark}\label{remark: disk gluing}
    Let $\mathcal{P}=\langle X\mid R\rangle$ be a presentation for a group. When constructing Van Kampen or annular diagrams (see \Cref{def:annular-diagram}) over $\mathcal{P}$, we adopt in the definition the convention that the boundaries of disks are glued along cycles labeled by elements of $R^*$ (as opposed to conjugates of elements of $R$). For example, if $D$ is a Van Kampen diagram for $w=s_1s_2\ldots s_k$ where $s_1=aba\inv{c}b^{-2}ab\inv{a}\inv{b}\inv{a}$, the boundary of the region corresponding to $s_1$ in $D$ will be glued along a cycle labeled by $\inv{c}b^{-2}ab$, not $s_1$. A similar convention is taken in the definition of annular diagrams over a presentation (\Cref{def:annular-diagram}).  This implies that the boundaries of regions of a diagram are immersed, as the 1-skeleta of diagrams are folded together in the sense of Stallings.  Gluing disks in this way simplifies certain arguments in \Cref{sec:whiteheadspace} and \Cref{sec: nonhyperbolicity}, which is why we take this convention.
\end{remark}

Let $D(w)$ be a Van Kampen diagram over a presentation $\mathcal{P}$ for a word $w$, and let $cw\inv{c}$ be any conjugate of $w$.  There is a diagram for $cw\inv{c}$ which differs from $D(w)$ only by a change of basepoint and, possibly, the addition or deletion of a dead-end path beginning at some point on the boundary of $D(w)$. 

\begin{definition}
A \emph{Van Kampen diagram} over $\langle X\mid R\rangle$ for a conjugacy class $[w]$ is the 2-complex resulting from taking a Van Kampen diagram for any $\hat{w}\in [w]$, forgetting the basepoint, and deleting all valence-one vertices and the edges incident to valence-one vertices.
\end{definition}

We will also utilize \emph{annular diagrams} in \Cref{sec: inf-diameter} and \Cref{sec: nonhyperbolicity}. The term \emph{annular diagram} can either refer to a $2$-complex described in \Cref{def:annular-diagram} below or a $2$-complex associated to a group presentation, similar to the distinction between \emph{disk diagram} and \emph{Van Kampen diagram}. Annular diagrams over a presentation $\mathcal{P} = \langle X\mid R\rangle$ are similar to Van Kampen diagrams in that they are planar 2-complexes and their regions have boundaries labeled by elements of $R^*$, but they are not simply connected.

\begin{definition}\label{def:annular-diagram}
    An \emph{annular diagram} $A$ is a finite connected planar $2$-complex with connected $1$-skeleton, fundamental group $\Z$, and at most one valence-$1$ vertex which is required to be on the boundary of $A$. The boundaries of disks (or \emph{regions}) of annular diagrams are immersed. The area of an annular diagram is the number of regions of the diagram. We say that an annular diagram has two boundaries (an ``inner boundary" and an ``outer boundary") even though the boundaries of an annular diagram may share edges or vertices (see \Cref{fig:annular-diagrams}). Note that whether one boundary is inside or outside the other depends on how the picture is drawn, not on a mathematically meaningful choice.

    Let $\mathcal{P}=\langle X\mid R\rangle$ be a group presentation, and assume that elements of $R$ are cyclically reduced. An \emph{annular diagram over $\mathcal{P}$} is a $2$-complex as described above with oriented edges labeled by elements of $X$. The boundaries of the regions of annular diagrams over $\mathcal{P}$ are labeled by elements of $R^*$. An annular diagram over $\mathcal{P}$ with boundaries labeled $v$ and $w$ exists only if $w$ is a product of a conjugate of $v$ with conjugates of elements of $R$. Such a diagram may not exist even if $w$ satisfies this condition, however: see \Cref{lemma:existence-annular}.
\end{definition}

\begin{remark}
    Typically we think about diagrams with the path metric given by assigning each edge a length of $1.$ Unlike some authors, we require that diagrams over $\langle X\mid R\rangle$ have edges labeled by generators as opposed to words in the generators. We do this because the proofs of \Cref{thm:infinite-diameter} and \Cref{theorem:nonhyperbolicity} relate the sum of the lengths of edges along the boundaries of diagrams to the lengths of the words labeling the boundaries of diagrams. We relax the condition that edges of diagrams $D$ be labeled by elements of the generating set $X$ in \Cref{def:nice} in a way that maintains the correspondence between the length of the boundary of a diagram $D$ and the length(s) of the freely reduced word(s) labeling the boundary of $D$.
\end{remark}

Much like with Van Kampen diagrams, we will typically forget the basepoint of annular diagrams and delete valence-$1$ vertices and edges incident to valence-$1$ vertices. Then we can refer to the boundaries as being labeled by cyclically reduced elements of conjugacy classes of $\free_n$ as opposed to elements of $\free_n$.

\begin{figure}
    \centering
    \includegraphics[scale=.35]{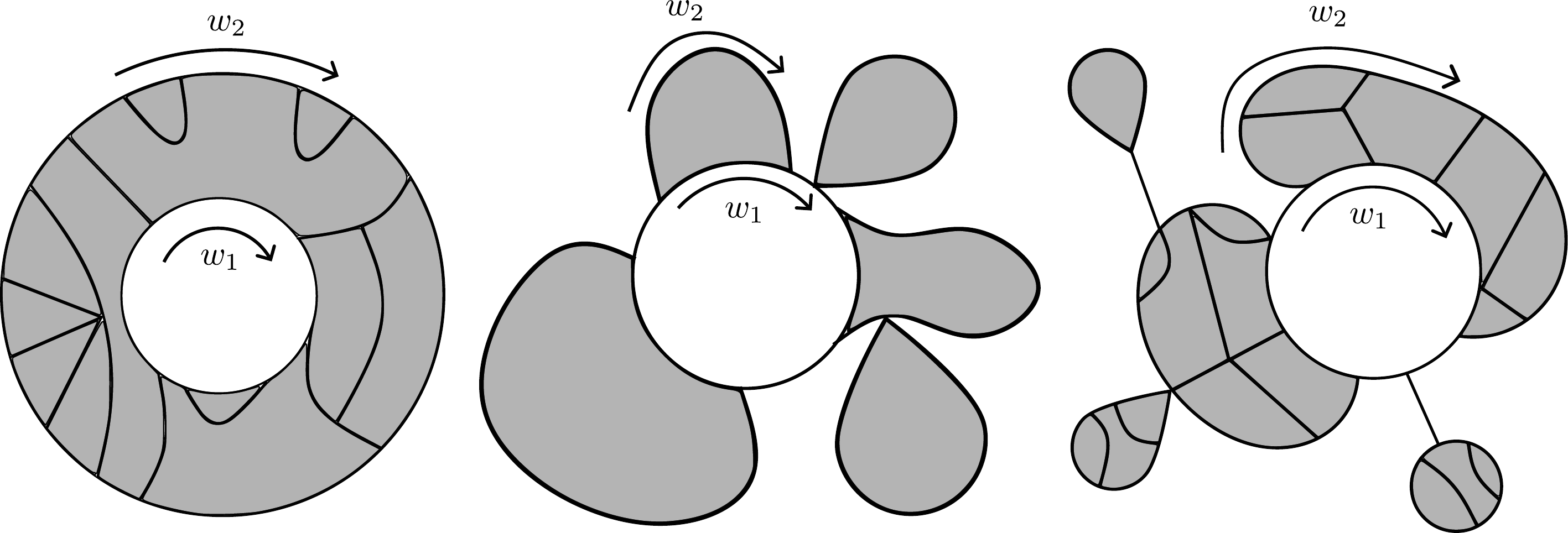}
    \caption{Three annular diagrams with inner boundary $w_1$ and outer boundary $w_2$.  Note that in the middle and right diagrams, the inner and outer boundaries overlap in certain edges and vertices.} 
    \label{fig:annular-diagrams}
\end{figure}

\subsection{Whitehead's algorithm for separable words} \label{sec: Whitehead's algorithm}
A \emph{proper free factor} of $\free_n$ is a nontrivial proper subgroup $F<\free_n$ such that $\free_n$ splits as a free product $\free_n = F*G$ for some nontrivial, proper $G<\free_n$.  We will need to be able to distinguish between words that are in a proper free factor of $\free_n$ and words that are not. 

\begin{definition}[\cite{stallings1999whitehead}] 
     A word in $\free_n$ is \emph{separable} if it belongs to a proper free factor of $\free_n$. Otherwise, $w$ is \emph{inseparable.}
\end{definition}

Some authors refer to these words as \emph{simple}; we use Stallings' terminology in order to avoid confusion with the definition of \emph{simple curves} on a surface. The proof of the following theorem of Stallings relies on ideas of Whitehead in \cite{whitehead1936certain}, hence the name ``Whitehead's algorithm for separable words."

\begin{thm}[\cite{stallings1999whitehead}]
    Given a word $w\in\free_n,$ there is an efficient algorithm to determine whether or not $w$ is separable.
\end{thm}

\noindent In fact, Algorithm 2.5 of \cite{stallings1999whitehead} determines whether finite sets of words $S\subset\free_n$ are separable (i.e., whether all elements of $S$ are conjugate into the factors of any splitting $F* G=\free_n$). We will only need the result for words in $\free_n$, so our definitions and algorithm only refer to separability of a single word. The algorithm to determine separability of $w$ constructs a graph called the \emph{Whitehead graph} of $w$, denoted $\Omega(w)$. Recall that, for a cyclically reduced word $w$, a cyclic subword of $w$ is a subword of one of the cyclic permutations of $w$.

\begin{definition}[Whitehead graph of a word]
Let $w\in\free_n$, where $\free_n$ is generated by $X_n$.  The Whitehead graph of $w$, denoted $\Omega(w)$, is a directed graph constructed as follows:
\begin{itemize}
    \item The vertex set of $\Omega(w)$ is $X_n^*$;
    \item There is one edge from $x$ to $y$ for each distinct copy of the cyclic subword $x\inv{y}$ in $w$; thus the edge set $E$ of $\Gamma$ is in bijection with the multiset of length two cyclic subwords of $w.$
\end{itemize}
\end{definition}
Equivalently, we can define the Whitehead graph as follows.  Let $M$ be the connected sum of $n$ copies of $\mathbb S^1\times \mathbb S^2$.  The fundamental group of $M$ is $\free_n$, and the generators of $\pi_1(M)$ are represented by a system of $n$ disjoint 2-spheres $\{S_1, S_2, \ldots, S_n\}\subset M$. Each sphere has a positive side $S_i^+$ and a negative side $S_i^-$.  The element $x_i\in\free_n$ is represented by a loop beginning at the basepoint of $M$ which travels to $S_i^-$, pierces $S_i$, and returns to the basepoint from $S_i^+$. Here we choose the system of spheres so that no sphere intersects the basepoint of $M$. The advantage of considering curves in $M$ instead of $\rose_n$ is that every $w\in\free_n$ can be represented by an embedding of $\mathbb S^1$ in $M,$ whereas in $\rose_n$ this is impossible.  Cutting $M$ along the sphere system results in a handlebody $\Hat{M}$ with $2n$ spheres on its boundary.  Let $w\in\free_n$, and represent $w$ as a geodesic curve $c\in M$.  The Whitehead graph of $w$ is the graph whose edges are the arcs of $c$ in $\Hat{M}$ and whose vertices are the $2n$ boundary spheres of $\Hat{M}$.  The orientations of the edges in the Whitehead graph of $w$ are given by the orientations of the arcs of $c$ in $\Hat{M}$. An example of the Whitehead graph for a specific word in $\free_n$ is shown in \Cref{fig: Whitehead graph}.

    \begin{figure}
        \centering
        \includegraphics[scale=.65]{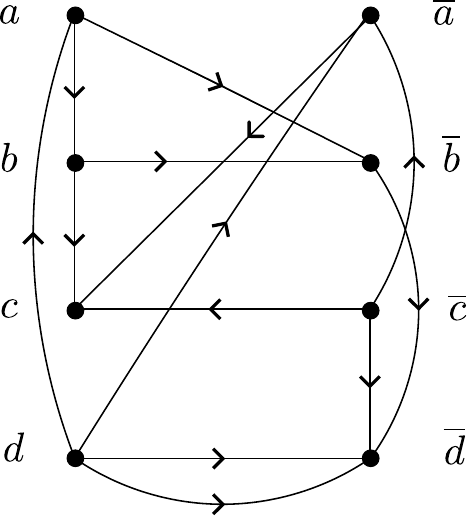}
        \caption{The Whitehead graph of $ab^2\inv{c}da\inv{b}d^3\inv{(c^2a)}\in \free_4$.}
        \label{fig: Whitehead graph}
    \end{figure}

\begin{definition}
    A vertex $v$ in a graph $\Gamma$ is a \emph{cut vertex} if, after removing $v$ and all edges incident to $v$, the resulting graph is disconnected.
\end{definition}

Whitehead's algorithm to decide separability of $w$ begins by constructing $\Omega(w)$.  If $\Omega(w)$ is connected with no cut vertices, $w$ is inseparable, and if $\Omega(w)$ is disconnected, $w$ is separable.  In the remaining case, $\Omega(w)$ is connected and has a cut vertex $x$.  In this case, we apply a particular automorphism to $w$ and reconstruct $\Omega(w)$.  The type of automorphisms used in the algorithm are called \emph{Whitehead automorphisms of the second kind.}

\begin{definition}
A \emph{Whitehead automorphism of the second kind} is an element $\phi\in \Aut(\free_n)$ satisfying the property that there is some fixed $x\in X_n^*$ called the \emph{multiplier} such that for $a\in X_n^*,$ $\phi(a)\in\{a, xa, a\inv{x}, xa\inv{x}\}$.
\end{definition}
\noindent Note that Whitehead automorphisms of the \emph{first} kind are the elements of $\Aut(\free_n)$ that permute the generators and their inverses.  The set of Whitehead automorphisms of the first and second kinds generates $\Aut(\free_n)$ as it contains the elementary Nielsen transformations.

Given a choice of $x\in X_n^*$ and a partition $Y\bigsqcup Z$ of $X_n^*$ into two nonempty subsets such that $x\in Y$ and $\inv{x}\in Z$, we define a Whitehead automorphism $\varphi_{x, Y, Z}$ as follows:

\[   \varphi_{x, Y, Z}(y) = 
     \begin{cases}
       y &\quad\text{if } x=y \text{ or } \{y, \inv{y}\} \subset Y; \\
        xy &\quad\text{if } y \in Y \text{ and } \inv{y} \in Z; \\
       y\inv{x} &\quad\text{if } y\in Z \text{ and }\inv{y}\in Y; \\
       xy\inv{x} &\quad\text{if } y, \inv{y}\in Z.\\ 
     \end{cases}
\]

\begin{thm}[\cite{stallings1999whitehead}]\label{theorem: Whitehead algorithm}
    The following process will determine whether $w\in\free_n$ is separable in finite time.
    \begin{enumerate}
        \item Construct $\Omega(w)$.
        \item If $\Omega(w)$ is disconnected, $w$ is separable. \cite[Proposition 2.2]{stallings1999whitehead}
        \item Else if $\Omega(w)$ is connected with no cut vertices, $w$ is inseparable. \cite[Theorem 2.4]{stallings1999whitehead}
        \item Else $\Omega(w)$ is connected and has a cut vertex $x\in X_n^*$.  Partition $X_n^*$ into two disjoint subsets $Y, Z$ as follows:  let $\Omega'$ be the graph resulting from the removal of $x$ and all edges incident to $x$.  Let $Z$ be the vertices in the component of $\Omega'$ containing $\inv{x}$, and let $Y = X_n^*\setminus Z$.  Return to Step 1 and repeat the process for $\varphi_{x, Y, Z}(w);$ eventually, this process terminates \cite[Proposition 2.3]{stallings1999whitehead}.
    \end{enumerate}
\end{thm}

\noindent A proof that this algorithm works can be found in \cite{stallings1999whitehead}.  

\begin{example}
    We will use Whitehead's algorithm to show $w = a\inv{b}cb\inv{a}\,\inv{c}a$ is separable in $\free\langle a, b, c\rangle.$  See \Cref{fig:Whitehead's Algorithm}. Observe that $a$ is a cut vertex. The partition of the vertices is given by $\{a, b\}\bigsqcup\{\inv{a}, \inv{b}, c, \inv{c}\}.$ The automorphism associated with this partition is below.
    \[\phi(x) = \begin{cases}
        a & x=a\\
        ab & x=b\\
        ac\inv{a} & x=c
    \end{cases}\]
    Apply $\phi$ to $w$ to obtain $\phi(w) = a\inv{b}cb\inv{c}.$  Observe that $w$ is primitive as it contains only one copy of $a$ or its inverse, so $w$ is separable.  
   
    \begin{figure}
        \centering
        \includegraphics[scale=.5]{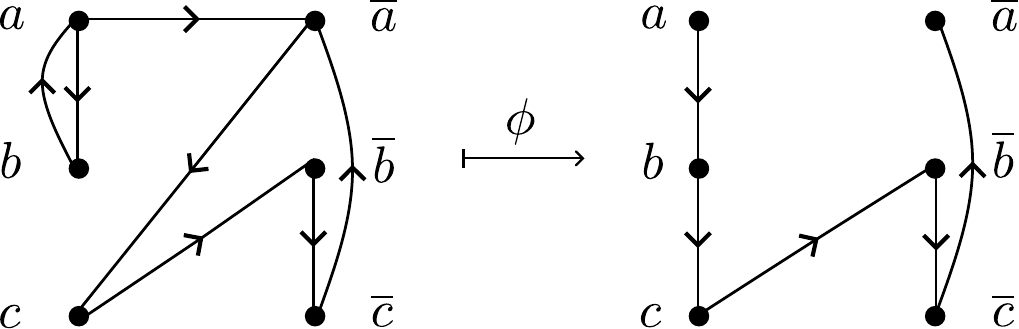}
        \caption{On the left is the Whitehead graph for $w=a\inv{b}cb\inv{a}\,\inv{c}a$; on the right, the Whitehead graph for $\phi(w)=a\inv{b}cb\inv{c}$.}
        \label{fig:Whitehead's Algorithm}
    \end{figure}

\end{example}

\begin{definition}
The \emph{profile} of a Whitehead graph $\Omega$ is an undirected finite graph $\sigma(\Omega)$ which is the image of $\Omega$ under the graph surjection $\sigma$ which takes all directed edges of $\Omega$ between vertices $x$ and $y$ to a single undirected edge between $x$ and $y$. 
\end{definition}

See \Cref{fig:Whitehead profile} for an example of a Whitehead profile. The advantage of considering the profile of a Whitehead graph is that there are finitely many profiles of Whitehead graphs of curves $\gamma$ in $\Gamma$, and one can determine from the Whitehead profile whether or not the Whitehead graph has a cut vertex.  Showing that the Whitehead profile of $w\in\free_n$ has no cut vertex is a quick way to ensure $w$ is inseparable.

\begin{figure}
    \centering
    \includegraphics[scale = .6]{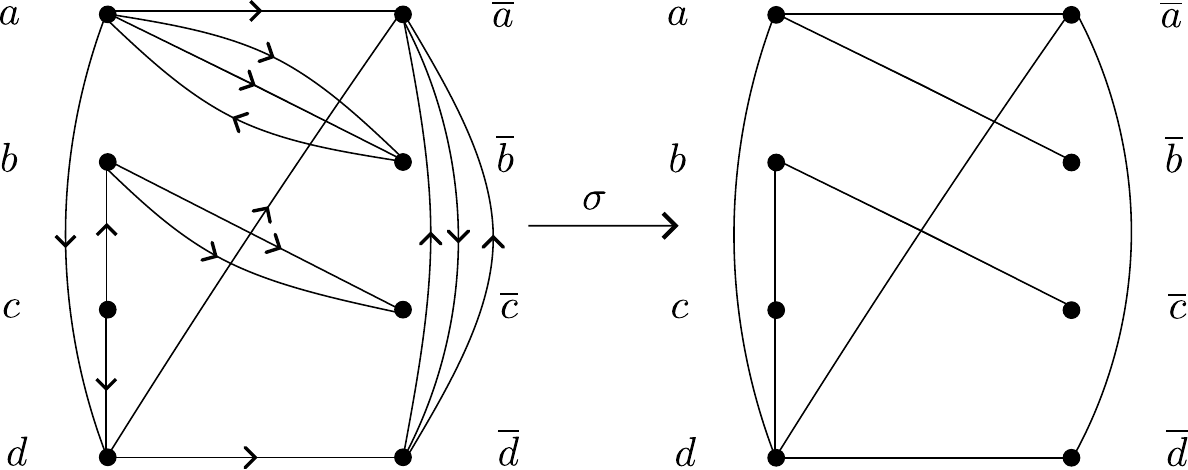}
    \caption{On the left, the Whitehead graph of $a^2bc\inv{d}a\inv{d}abc\inv{b}\inv{a}d^2$; on the right, its Whitehead profile}
    \label{fig:Whitehead profile}
\end{figure}

Another variation of the Whitehead graph that is sometimes useful is the \emph{labeled Whitehead graph}, defined below.

\begin{definition}\label{def:labeled-Wh-graph}
    The \emph{labeled Whitehead graph} of a cyclically reduced $w\in\free_n$ is the Whitehead graph $\Omega(w)$ together with a labeling of the edges of $\Omega(w)$ by $\Z\big/ m\Z$, where $m$ is the length of $w$. The edges of $\Omega(w)$ are labeled in the order that the corresponding length-two cyclic subwords of $w$ appear in $w$. Two labeled Whitehead graphs $\Omega$ and $\Psi$ are equivalent whenever $\Omega$ and $\Psi$ are isometric via an isometry that preserves vertex labels, and the edge labels of $\Omega$ and $\Psi$ differ only by a fixed element of $\Z\big/ m\Z$. See \Cref{fig:equiv labeled Whitehead graphs} for an example.  
\end{definition} \noindent Equivalence classes of labeled Whitehead graphs are in one-to-one correspondence with conjugacy classes of elements of the free group. 

\begin{figure}
    \centering
    \includegraphics[scale = .6]{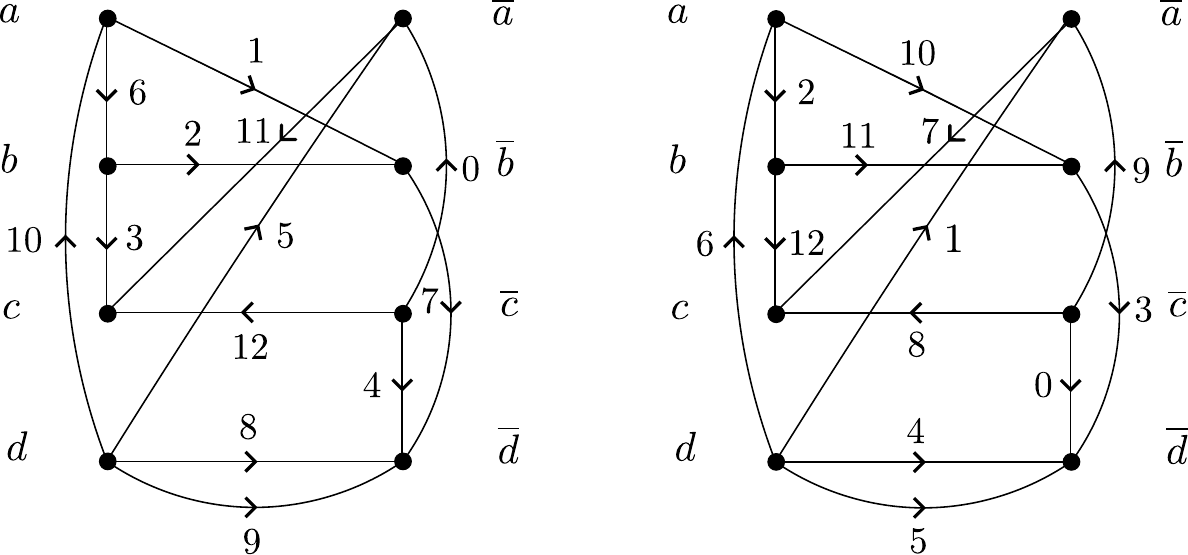}
    \caption{Equivalent labeled Whitehead graphs of $ab^2\inv{c}da\inv{b}d^3\inv{(c^2a)}$}
    \label{fig:equiv labeled Whitehead graphs}
\end{figure}

\section{The separability complex}
\label{sec:whiteheadspace}

The purpose of this section is to introduce the separability complex and describe some of its properties. The separability complex is a tool to study finite-index normal subgroups of $\free_n$ (equivalently, finite regular covers $\Gamma$ of the rose). Let $N$ be a finite-index normal subgroup of $\free_n$ and $\Gamma$ the cover of $\rose_n$ with fundamental group $N$. A key property of $\Wh(\Gamma)$ is that it is connected if and only if $\pi_1(\Gamma)$ is generated by $\Csep\cap\pi_1(\Gamma)$ (see \Cref{thm: equiv-connectivity}). Throughout the paper, $\Gamma$ is a finite-index regular cover of $\rose_n$ for $n\geq 2$, and $\Gamma\sep$ is the cover of $\rose_n$ representing the subgroup of $\pi_1(\Gamma)$ generated by elements of $\Csep\cap\pi_1(\Gamma)$. See \Cref{fig: piece of Wh space} for a small subgraph of the separability complex of a specific cover of $\rose_n$.  

\begin{definition}[separability complex] \label{def:whiteheadspace}
Let $\Gamma$ be a finite regular cover of $\mathbf{R}_n$ for some $n\geq 2$.  The separability complex of $\Gamma$, denoted $\Wh(\Gamma)$, is a locally infinite directed graph constructed as follows:

\begin{itemize}
    \item The vertex set $V$ of $\Wh(\Gamma)$ is the set of equivalence classes of elements of $\pi_1(\Gamma)$ under conjugation by $\free_n$;
    \item Each edge is labeled by the (free) conjugacy class of a nontrivial element of $\Csep\cap\pi_1(\Gamma)$.  There is an edge $[\alpha]$ from $[v]$ to $[w]$ whenever there are $\hat{\alpha}\in[\alpha],$ $\hat{w}\in [w],$ $\hat{v}\in[v]$ such that $\hat{w}=\hat{v}\hat{\alpha}$ in $\free_n$.
\end{itemize}
\end{definition} 
We will generally use Latin letters in square brackets to refer to vertices in $\Wh(\Gamma)$ and Greek letters in square brackets to denote edge labels in $\Wh(\Gamma)$. Note that we can take $\hat{w}$ to be any representative of $[w]$ by conjugating the equation $\hat{w}=\hat{v}\hat{\alpha}$.

\begin{remark}
    The complex $\Wh(\rose_n)$ is very similar to $\Cay(\free_n, \Csep)/\Inn(\free_n)$. There is an action of $\Inn(\free_n)$ on $\Cay(\free_n, \Csep)$ because the property of being separable is preserved by $\Aut(\free_n)$. Note that $\Cay(\free_n, \Csep)/\Inn(\free_n)$ surjects $\Wh(\Gamma),$ and this surjection preserves distances between vertices. To see this, note that the vertices in both complexes are labeled identically. The surjection takes the vertex labeled $[v]\in \Cay(\free_n, \Csep)/\Inn(\free_n)$ to the vertex labeled $[v]\in\Wh(\Gamma).$ Let $v$ and $w$ be adjacent vertices of $\Cay(\free_n, \Csep)$.  We identify $v$ and $w$ with the elements of $\free_n$ they represent. Suppose that $\alpha\in\Csep$ such that $v\alpha=w$. The quotient of $\Cay(\free_n, \Csep)$ by $\Inn(\free_n)$ identifies, for all $c\in\free_n$, edges labeled $c\alpha\inv{c}$ from $cv\inv{c}\in\Cay(\free_n, \Csep)$ to $cw\inv{c}\in\Cay(\free_n, \Csep)$.  Edges labeled $[\alpha]$ from $[v]$ to $[w]$ in $\Cay(\free_n, \Csep)/\Inn(\free_n)$ therefore correspond to sets of equations $\{cv\alpha\inv{c}=cw\inv{c}\mid c\in\free_n\}$. The surjection from $\Cay(\free_n, \Csep)/\Inn(\free_n)$ to $\Wh(\rose_n)$ takes every edge labeled $[\alpha]$ from $[v]$ to $[w]$ in $\Cay(\free_n, \Csep)/\Inn(\free_n)$ corresponding to multiplication of a representative of $[v]$ by a representative of $[\alpha]\subset\free_n$ to the single edge from $[v]$ to $[w]$ in $\Wh(\Gamma)$ labeled $[\alpha]$. Thus the large-scale geometry of $\Wh(\Gamma)$ is identical to that of $\Cay(\free_n, \Csep)/\Inn(\free_n).$
    
    However, the quotient from $\Cay(\free_n, \Csep)$ to $\Wh(\free_n)$ has nontrivial kernel.  An example of an equation $[v][\alpha]=[w]$ as described above with multiple solutions up to conjugacy is described in \Cref{example:Cay_quotient_nontrivial_kernel}.
\end{remark}

\begin{example}\label{example:Cay_quotient_nontrivial_kernel}
    Consider the equation on free conjugacy classes $[a][b]=[abab^{-1}a^{-1}baba^{-1}b^{-1}].$ The following two equations in the free group are solutions to this equation. Here $a^c=cac^{-1}$ and $[x, y]=[xyx^{-1}{y^-1}].$
    \begin{equation}
        a\cdot b^{[b, a]}=abab^{-1}a^{-1}baba^{-1}b^{-1}
    \end{equation}
    \begin{equation}
        a^{ab}\cdot b^{ba}=abab^{-1}a^{-1}baba^{-1}b^{-1}
    \end{equation}
    Note that there are two corresponding edges labeled $[b]$ from $[a]$ to $[abab^{-1}a^{-1}baba^{-1}b^{-1}]$ in $\Cay(\free_n, \Csep)/\Inn(\free_n)$ which are identified in $\Wh(\free_n)$.
\end{example}

\noindent A version of the following question was asked in a personal communication with the author.
\begin{question}\label{question:which_subroup_stabilizes_edges}\cite{Bering2026}
    Given an edge $e$ labeled $[\alpha]$ from $[v]$ to $[w]$ in $\Cay(\free_n, \Csep)/\Inn(\free_n)$, what is the subgroup of $\Out(\free_n)$ preserving the label and endpoints of $e$?
\end{question}

\begin{figure}
    \centering
    \includegraphics[scale=.6]{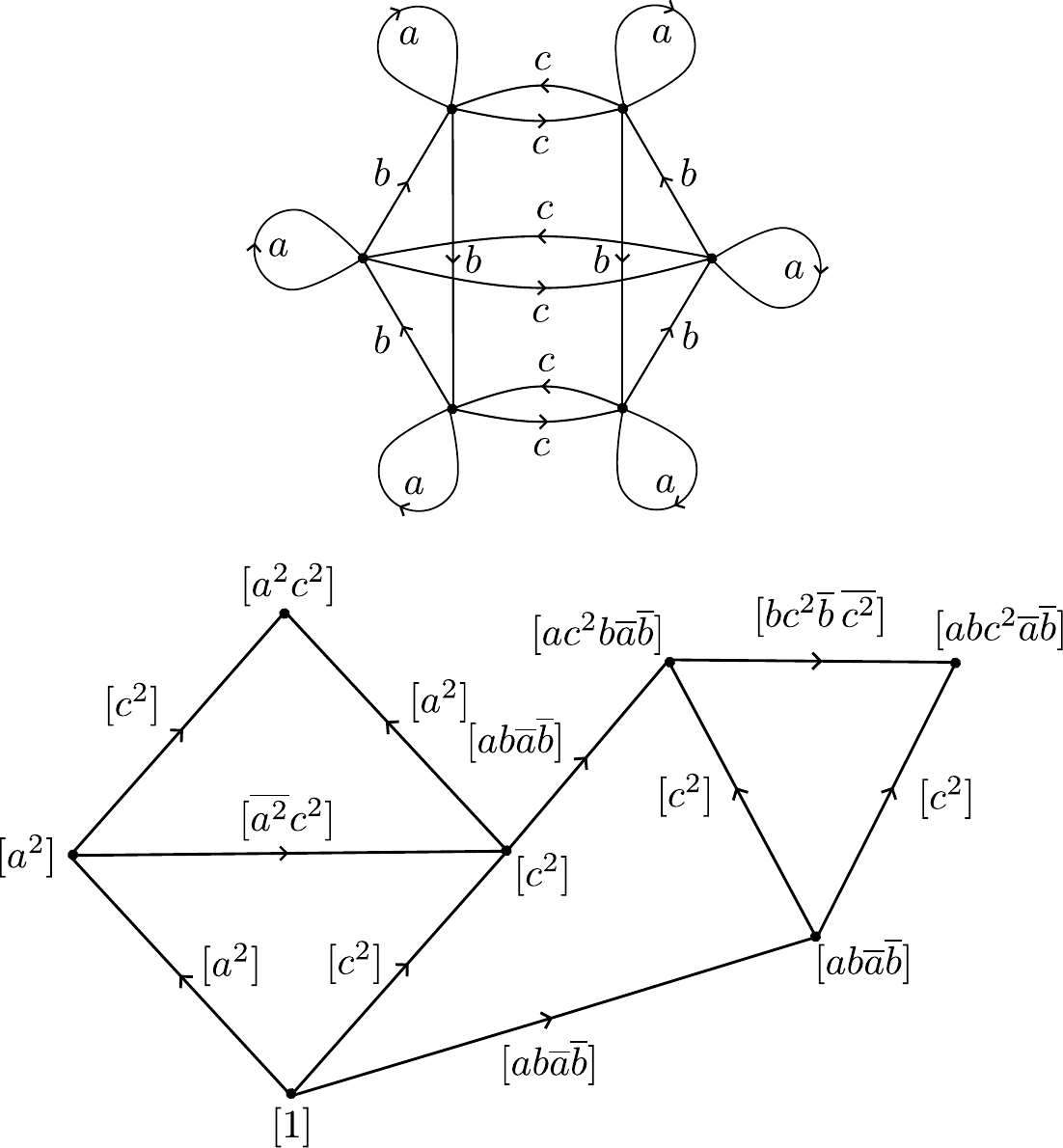}
    \caption{Above, a finite regular cover $\Gamma$ of $\rose_3$; below, a small portion of $\Wh(\Gamma)$ to illustrate edge relations.}
    \label{fig: piece of Wh space}
\end{figure}

Every vertex in $\Wh(\Gamma)$ is incident to at least one edge for each conjugacy class of nontrivial curves in $\Csep\cap\pi_1(\Gamma)$, so $\Wh(\Gamma)$ is locally infinite. We think of $\Wh(\Gamma)$ as a metric space with the induced path metric: every edge of $\Wh(\Gamma)$ has length one, and the distance between $[v]$ and $[w]\in\Wh(\Gamma)$ is the length of the shortest path between $[v]$ and $[w]$. If no path between $[v]$ and $[w]$ exists, the distance between them is infinite. The distance between $[v]$ and $[w]$ is denoted $\dsep{[v]}{[w]}$.  The distance between $[1]$ and $[w]\in\Wh(\Gamma)$ is called the \emph{separable length} of $[w]$ (or the separable length of $w$) and denoted by $\seplength{w}$.

Consider the set of paths in $\Wh(\Gamma)$, denoted $\pathwh(\Gamma)$.  Let $p([v], [w]) \in\pathwh(\Gamma)$ have edges labeled by $[\alpha_1], [\alpha_2], \ldots, [\alpha_{l(p)}]$, where $l(p)$ is the length of $p$.  Then $p$ corresponds to an expression for $\hat{w}\in[w]$ as a product of a representative $\hat{v}$ of $[v]$ and elements $\hat{\alpha_1}, \hat{\alpha_2}, \ldots \hat{\alpha_{l(p)}}$ of the separable conjugacy classes labeling the edges of $p$. For any path $p([1], [w])\in\pathwh(\Gamma)$, there is an associated Van Kampen diagram $D_p$ for $[w]$ over $\langle X_n\mid \Csep\cap\pi_1(\Gamma)\rangle.$ For the convenience of the reader, we summarize this construction below in \Cref{lemma:paths=diagrams}. This construction is also described in the proof of Theorem 1.1 in Chapter V of \cite{lyndon1977combinatorial} and pages 40-41 of \cite{epstein1992word}.

Recall that our definition of Van Kampen diagram (\Cref{def:Van_Kampen_diagram}) requires regions with immersed boundaries, and therefore the boundaries of regions of Van Kampen diagrams are assumed throughout the paper to be labeled by cyclically reduced words.
\begin{lemma}\label{lemma:paths=diagrams}
If $p([1], [w])\in\pathwh(\Gamma)$, there is a Van Kampen diagram $D_p$ for $[w]$ over $\langle X_n\mid \Csep\cap\pi_1(\Gamma)\rangle.$ The regions of $D_p$ have boundaries labeled by (cyclically reduced) representatives of some of the conjugacy classes labeling the edges of $p$. If $p$ is a path of minimal length, the regions of $D_p$ have boundaries labeled by all of the (cyclically reduced) representatives of the conjugacy classes labeling the edges of $p.$
\end{lemma}

\begin{proof}
Let the edges of $p([1], [w])$ be $[\alpha_1], [\alpha_2], \ldots, [\alpha_l]$ where $l$ is the length of $p([1], [w]).$  By the definition of $\Wh(\Gamma)$, there are $\hat{\alpha_i}\in [\alpha_i]$ so that $\hat{\alpha_1}\hat{\alpha_2}\ldots\hat{\alpha_l} =\hat{w}\in[w]$.  To construct a diagram for $[w]$, wedge disks with the same orientation and boundary labels $\hat{\alpha_1}, \hat{\alpha_2}, \ldots, \hat{\alpha_l}$ in clockwise order at a point to obtain a 2-complex $D$.  The complex $D$ has boundary which freely reduces to a conjugate of $w\in\free_n$ when read clockwise beginning at the wedge point. Then, to make the boundary of $D_p$ freely reduced, fold adjacent edges in the clockwise ordering on the boundary of $D$ in the sense of Stallings so that the boundary of the resulting complex is labeled by $w$ when read in clockwise order beginning at the basepoint. If, when folding, a disk subcomplex $T$ arises which has boundary label freely reducing to $1$, delete the interior of $T$ and fold the edges on the boundary of $T$ together. Note that this will happen only if the product of the relations on the interior of $T$ freely reduces to $1$, which implies that $[\alpha_1], [\alpha_2], \ldots, [\alpha_l]$ is not a path in of minimal length from $1$ to $[w]$ in $\Wh(\Gamma)$ as the relations labeling regions on the interior of $T$ could be eliminated without changing the product. Finally, delete all valence-one vertices to get a diagram for $[w].$  See \Cref{fig:paths-diagrams} for an example of the diagram corresponding to a specific path in the separability complex.
\end{proof} 

\begin{remark}\label{rmk:minimalpaths}
    A key takeaway of \Cref{lemma:paths=diagrams} is that whenever a path $p([1], [w])\in\Wh(\Gamma)$ is of minimal length in $\Wh(\Gamma)$, the length of $p$ is the area of the corresponding diagram.
\end{remark}

\begin{figure}
    \centering
    \includegraphics[scale=.55]{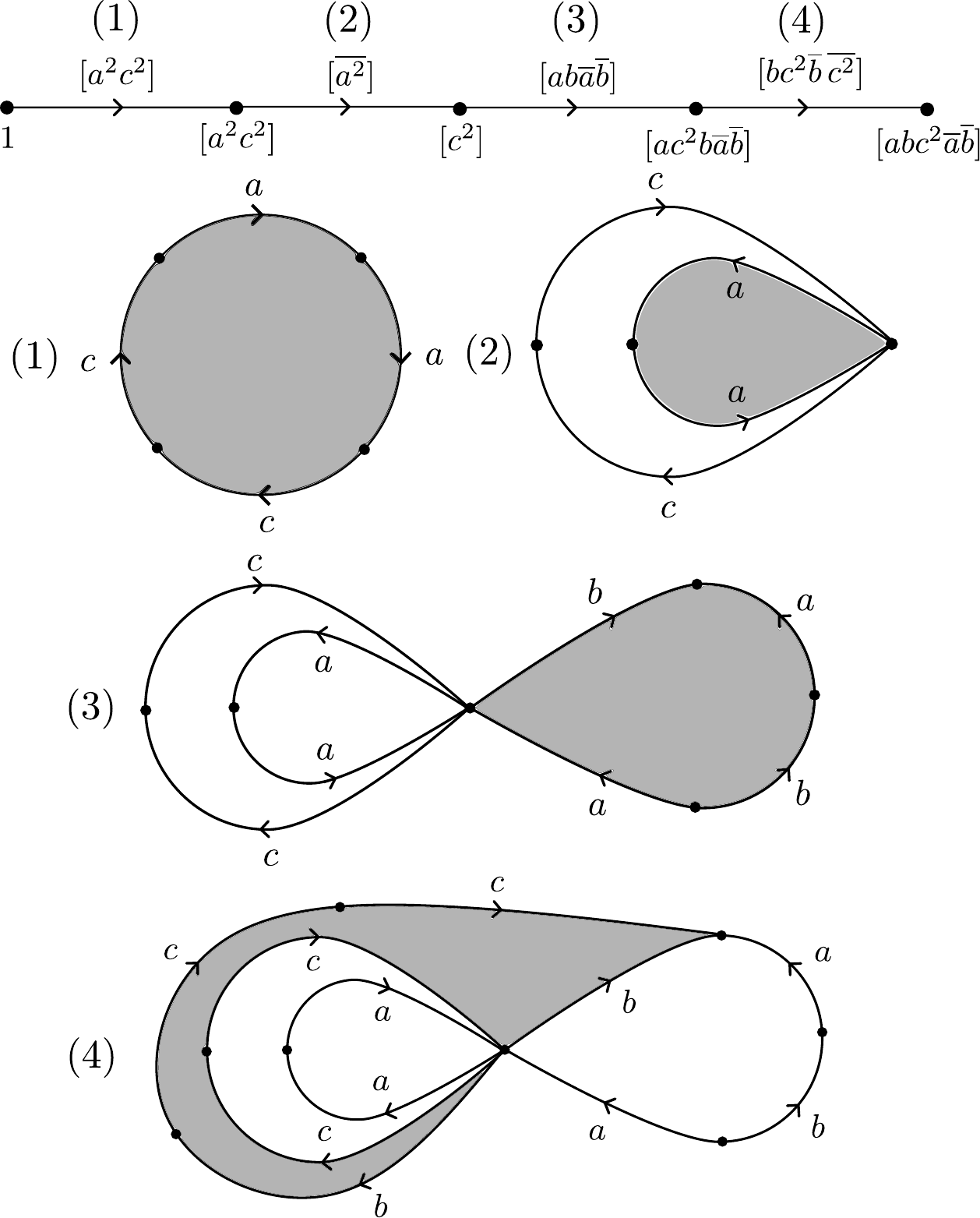}
    \caption{A path in $\Wh(\rose_3)$ from $1$ to $[abc^2\inv{(ba)}]$ and a corresponding sequence of Van Kampen diagrams.}
    \label{fig:paths-diagrams}
\end{figure}  

\equivalentcondition*

\begin{proof}
    It suffices to show that, given a finite regular cover $\Gamma$ of $\rose_n$, the complex $\Wh(\Gamma)$ is connected if and only if $\pi_1(\Gamma) = \pi_1(\Gamma\sep)$. Each path $p=p([1], [v])\in\pathwh(\Gamma)$ corresponds to an expression $\prod_{i=1}^l \alpha_i=\hat{v}\in[v]$, where $l$ is the length of $p$ and $\alpha_i\in\Csep\cap\pi_1(\Gamma)$ are representatives of the conjugacy classes labeling the edges of $p$.  If $\Wh(\Gamma)$ is connected, there is a path from $[1]$ to every vertex $[v]\in\Wh(\Gamma)$. Since $\Gamma$ is regular, the existence of an expression for $\hat{v}$ as a product of elements of $\Csep\cap\pi_1(\Gamma)$ implies that any $v'\in[v]$ is expressible as a product of elements of $\Csep\cap\pi_1(\Gamma)$. Hence $\pi_1(\Gamma)=\pi_1(\Gamma\sep)$. Similarly, supposing $\pi_1(\Gamma)=\pi_1(\Gamma\sep),$ there are expressions for all $g\in\pi_1(\Gamma)$ as  products of elements of $\Csep\cap\pi_1(\Gamma)$.  These expressions correspond to paths from $[1]$ to $[g]$ in $\Wh(\Gamma)$, and the union of these paths connects $\Wh(\Gamma)$.
\end{proof}

Consider $\Wh(\rose_n)$.  This space has a vertex for every conjugacy class of elements of $\free_n,$ and its edges are labeled by conjugacy classes of nontrivial elements of $\Csep$.  Since all primitive elements of $\free_n$ are in $\Csep$, the space $\Wh(\rose_n)$ is connected.  

\begin{prop}\label{theorem:isometric-action}
    $\Wh(\rose_n)$ admits an $\Out(\free_n)$-action by isometries. 
\end{prop}
\begin{proof}
    First we show $\Wh(\rose_n)$ admits an $\Aut(\free_n)$-action.  Let $\phi\in\Aut(\free_n)$.  The automorphism $\phi$ sends $[v]$ to $[\phi(v)]$. Suppose there is an edge $[\alpha]\in\Wh(\rose_n)$ from $[v]$ to $[w]$. Then there are elements $c, d\in\free_n$ satisfying $vc\alpha\inv{c}=dw\inv{d}$. If we apply $\phi$ to this equation, we get $\phi(v)\phi(c)\phi(\alpha)\inv{\phi(c)} = \phi(d)w\inv{\phi(d)}.$  So in $\Wh(\rose_n)$, the image of the edge $[\alpha]$ with endpoints $[v], [w]$ is an edge $[\phi(\alpha)]$ with endpoints $[\phi(v)], [\phi(w)]$.  Since $\phi\in\Aut(\free_n)$ is an automorphism, it permutes conjugacy classes of separable words, hence $\phi(\alpha)$ is separable.  Note that if $\phi$ is an inner automorphism, $\phi$ fixes $\Wh(\rose_n)$ pointwise, so we can upgrade the $\Aut(\free_n)$-action on $\Wh(\rose_n)$ to an $\Out(\free_n)$-action.  

    The action of $\Out(\free_n)$ preserves the metric on $\Wh(\rose_n)$: let $\phi\in\Out(\free_n)$ and \newline ${[v], [w]\in\Wh(\rose_n)}$. Let $p([v], [w])$ be any path between $[v]$ and $[w]$.  Since automorphisms of the free group preserve separability, $\seplength{v} = \seplength{\phi(v)}$, and $\mathrm{length}(p)=\mathrm{length}(\phi(p))$. The action of $\Out(\free_n)$ preserves $1$, so the action of $\Out(\free_n)$ on the separability complex of the rose is an action on a rooted graph.
\end{proof}
\begin{remark}
    Note that the only vertex of $\Wh(\rose_n)$ fixed by $\Out(\free_n)$ is $[1]$, so by deleting $[1]$ from $\Wh(\rose_n)$, we get a space with an isometric action of $\Out(\free_n)$ with no fixed points.  It is convenient for us to include $[1]$, but the proofs of all statements about $\Wh(\rose_n)$ could be altered to apply to $\Wh(\rose_n)\setminus \{[1]\}.$
\end{remark}

We get a similar result to \Cref{theorem:isometric-action} for all $\Wh(\Gamma)$ where $\Gamma$ represents a characteristic subgroup of $\free_n$. We do not get an $\Out(\free_n)$-action on the separability complexes of graphs representing subgroups which are not characteristic: by definition, subgroups that are not characteristic are not preserved under the action of $\Aut(\free_n)$ on $\free_n$. Since $\Gamma$ represents a normal subgroup of $\free_n$, $\pi_1(\Gamma)$ is fixed by $\Inn(\free_n)$, so there is an automorphism $\phi\in\Out(\free_n)$ and $g\in\pi_1(\Gamma)$ such that $\phi(g)\not\in\pi_1(\Gamma)$. 

There is a connected component of $\Wh(\Gamma)$ containing all nontrivial separable conjugacy classes of $\pi_1(\Gamma)$. If $\Gamma$ represents a characteristic subgroup, this ``separable component" is preserved under the action of $\Out(\free_n)$.

\begin{prop} \label{theorem:outFn-action-char}
Let $\Gamma$ represent a characteristic subgroup of $\free_n$. Then $\Wh(\Gamma)$ admits an isometric action of $\Out(\free_n)$. This action fixes the component of $\Wh(\Gamma)$ containing all elements of $\pi_1(\Gamma\sep)$.
\end{prop}

\begin{proof}
 For the first statement, note that since $\Gamma$ is characteristic, $\pi_1(\Gamma)$ is preserved under the action of $\Out(\free_n)$ on $\free_n$. The rest of the proof is the same as the proof of \Cref{theorem:isometric-action}.

 For the second statement, let $C$ be the component of $\Wh(\rose_n)$ containing the trivial vertex. Then $C$ contains vertices for all $[g]\subset\pi_1(\Gamma\sep)$ as there are edges from $[1]$ to the conjugacy classes of every separable element of $\pi_1(\Gamma)$. Let $[y]\in C$, so there is a path ${p=[\alpha_1], [\alpha_2], \ldots, [\alpha_m]}$ from the trivial vertex to $[y]$ where $m\in\N$ and $\alpha_i\in\Csep\cap\pi_1(\Gamma)$. For ${\phi\in\Out(\free_n)},$ $\phi(p)$ is a path with edges $[\phi(\alpha_1)], [\phi(\alpha_2)], \ldots, [\phi(\alpha_m)]$. Since $\phi$ is an automorphism and $\Gamma$ is characteristic, $[\phi(\alpha_i)]\in\Csep\cap\pi_1(\Gamma)$ for all $i,$ so $\phi(p)$ is a path from the trivial vertex to $[\phi(y)]$, implying ${[\phi(y)]\in C}$.  \qedhere
\end{proof}

Finally, regarding the topology of $\Wh(\rose_n)$, we have the following proposition.

\begin{prop}
   The flag complex of $\Wh(\rose_n)$ is infinite-dimensional for all $n\geq 2$.
\end{prop}

\begin{proof}
    Recall that we let $\free_n$ be generated by the set $X_n=\{x_1, x_2, \ldots, x_n\}$. Let \newline ${Y=\{x_1x_2^m\}_{m=1}^{\infty}.}$ Every element of $Y$ is primitive in $\free_n$ for all $n$, so every element of $Y$ is separable. Without loss of generality, let $k$ and $l$ be positive integers such that $k<l$. Observe that $\dsep{[ab^k]}{[ab^l]}=1$ for all $k, l$ as $ab^k\cdot b^{l-k}=ab^l.$ Thus, for all $d\geq 1$, $\Wh(\rose_n)$ contains the one-skeleton of a simplex of dimension $d$ on the vertices $\{[1], [ab], [ab^2], \ldots, [ab^d]\}.$
\end{proof}

\subsection{Diameter of the separability complex}\label{sec: inf-diameter}
In this section, we prove that the diameter of every component of $\Wh(\Gamma)$ is infinite. As a corollary, we show that the diameters of $\Cay(\free_n; \Csep)$ and $\Cay(\free_n, \curves\prim)$ are infinite. The key idea of the proof is that paths in the separability complex correspond to diagrams for conjugacy classes of $\free_n$. If $\Wh(\rose_n)$ had finite diameter, there would be a universal bound on the areas of minimal Van Kampen diagrams for every element of $\free_n$ over $\langle X_n\mid\Csep\rangle$. In the proof, we utilize \emph{positive words}.

\begin{definition}
    A word $w\in\words_n$ is \emph{positive} if all of the letters in $w$ are positive.  
\end{definition}

We begin by proving two lemmas.  First, we show that while a separable word can have inseparable subwords, these inseparable subwords are ``barely" inseparable: in particular, inseparable subwords $v$ of separable words have the property that the graph resulting from removing a certain edge of $\Omega(v)$ has a cut vertex (see \Cref{lemma: subwords-separable-words}).  In addition, we show that if $D$ is a Van Kampen diagram for a positive word, $D$ has no cut edges (see \Cref{lemma: positive-words-wedge-disks}). 

\begin{lemma}\label{lemma: subwords-separable-words}
    Let $y_i\in X_n$ and suppose $u=y_1y_2\ldots y_l$ is a positive word satisfying the property that removing the edge corresponding to $y_ly_1$ from the Whitehead graph of $u$ results in a graph with no cut vertices.  Let $t$ be any positive word and $w=ut$.  Then $w$ is inseparable.
\end{lemma}

\begin{proof}
    Let $t = z_{l+1}\ldots z_m$, where each $z_i$ is in $X_n$. The Whitehead graph of $w$ has all the edges in the Whitehead graph of $u$ except the edge corresponding to $y_ly_1$, plus the edges corresponding to $\{y_lz_{l+1}, z_jz_{j+1}, z_my_1 \mid l+1\leq j\leq m-1\}$.  Adding edges to a graph with no cut vertices results in a graph with no cut vertices, so $w$ is inseparable.
\end{proof}

\begin{definition}
    A $\emph{cut edge}$ of a cell complex $X$ is an edge $e\in X$ which satisfies the property that removing a single point on the interior of $e$ disconnects $X$.
\end{definition}

\begin{lemma}\label{lemma: positive-words-wedge-disks}
    Let $w$ be a positive word.  Then if $D$ is a Van Kampen diagram for $w$, $D$ has no cut edges.  
\end{lemma}

\begin{proof}
    If $D$ is a Van Kampen diagram for $w$ which contains a cut edge $e,$ then $w$ contains both $\lab(e)$ and $\lab(e)^{-1}$ as subwords, so $w$ is not positive.
\end{proof}

Let $\widehat{\words_n^{\infty}}$ be the set of freely reduced one-sided infinite words, i.e. the subset of $\prod_{i = 1}^{\infty} X_n^*$ containing exactly the elements with no subwords $x_i\inv{x_i}$ for any $i$. Our next lemma shows that there are positive infinite words $W\in\widehat{\words_n^{\infty}}$ with the property that all subwords of $W$ of sufficient length are inseparable and not a subword of any separable word.  Later, we will show that paths in $\Wh(\rose_n)$ corresponding to these infinite words have infinite diameter. 

\begin{lemma}\label{lemma: infinite-positive-filling-words}
    Let $k\in\N$ and $n\geq 2$.  If $k > 4n^2$, there are infinite words $W\in\widehat{\words_n^{\infty}}$ with the property that if $w_i$ is an initial subword of $W,$ and $u\subset w_i$ is a cyclic subword of $w_i$ of length at least $k$, then $u$ not a subword of any separable word.
\end{lemma}

\begin{proof}
    Recall that we take $X_n=\{x_i\mid 1\leq i\leq n\}$ to be the generating set for $\free_n$. The edges in the Whitehead graph of $w\in\free_n$ are in bijection with length-two cyclic subwords of $w$.  If $w$ is a positive word, then by construction of the Whitehead graph, the most complex Whitehead profile $w$ can have is the complete bipartite graph $B_n$, where the partition on the vertices of the profile is given by $\{x_i\in X_n\}\bigsqcup\{\inv{x_i}\mid x_i\in X_n\}$ (see \Cref{fig: Wh-profiles-positive}). Observe that for $n\geq 3$, $B_n$ has the property that removal of a single edge yields a graph with no cut vertices, so by \Cref{lemma: subwords-separable-words}, if $n\geq 3$, $B_n$ cannot be the Whitehead profile of any subword of a separable word.
    
    Let $n\geq 3$. In $\free_n$, there are $n^2$ positive words $x_ix_j$ of length two.  To construct $W$, let $k > 4n^2$ and ensure that each of the $x_ix_j$ are present in every subword of $W$ of length at least $\frac{k}{2}$ (see (1) of \Cref{ex: inf-pos-filling-word} for an example of such a word for $n=4$).  Let $w$ be an initial subword of $W$ with a cyclic subword $u$ of length at least $k$.  If $u$ is a subword of $w$, we are done as the Whitehead profile of $u$ is $B_n$.  If $u=ts$ where $t$ is a terminal subword of $w$ and $s$ is an initial subword of $w$, either $t$ or $s$ has length at least $\frac{k}{2}$, so we are done as the Whitehead profile of either $t$ or $s$ is $B_n$. 
      
    For $n=2$, removal of an edge from $B_2$ does yield a cut vertex, so we require that every length-$k$ subword of $W$ contains two copies of every positive length-two subword $x_ix_j$ (see (2) of \Cref{ex: inf-pos-filling-word} for an example). 
    This way, if $w\subset W$ is an initial subword, $u$ is a cyclic subword of $w$ of length at least $k$, and $e\in\Omega(u)$ is an edge, $\sigma(\Omega(u)\setminus\{e\})$ has no cut vertices. \end{proof}

\begin{figure}
        \centering
        \includegraphics[scale=.55]{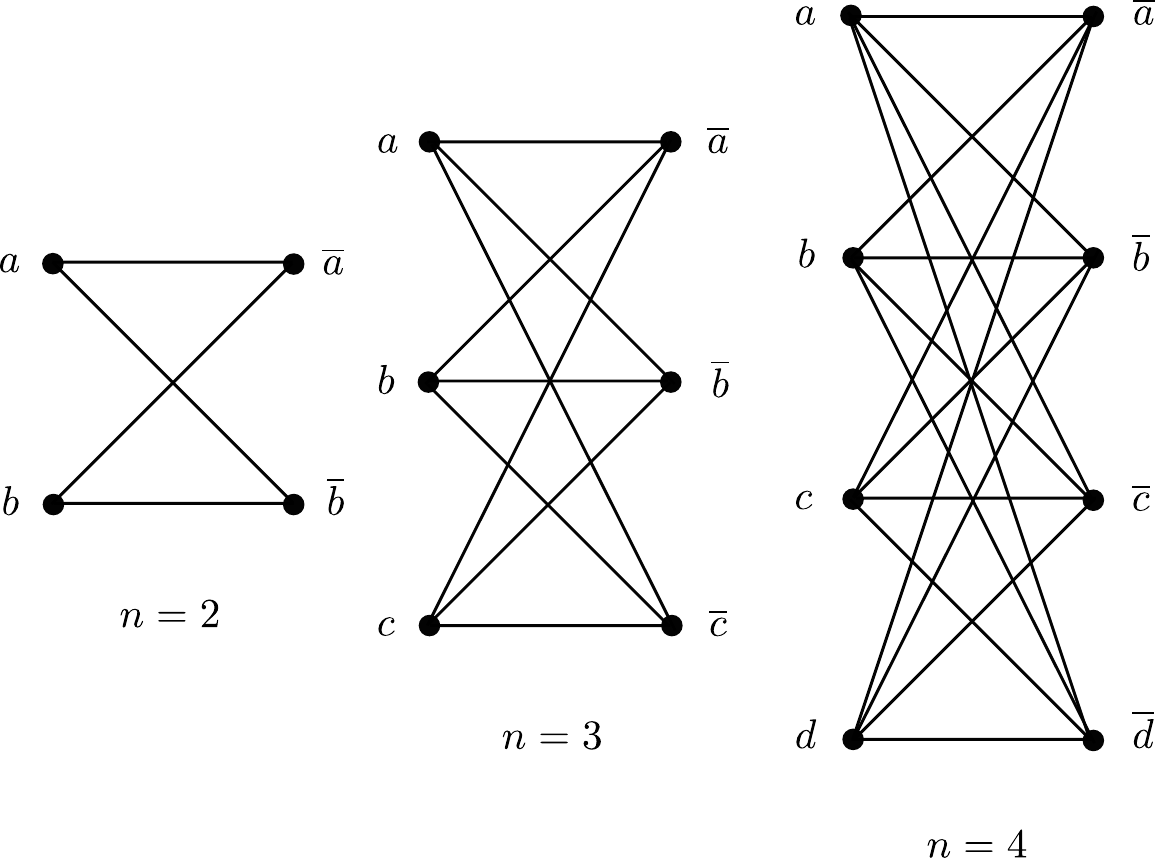}
        \caption{Maximal Whitehead profiles of positive words for small $n$}
        \label{fig: Wh-profiles-positive}
\end{figure}

\begin{example}\label{ex: inf-pos-filling-word}

    \begin{enumerate}
    \item Let $n=4$ and let $\{a, b, c, d\}$ be a basis for $\free_4$.  Consider the infinite word $W:=\prod_{i=1}^{\infty} w,$ where $w=a^2b^2c^2d^2acbdcadba$.  Then $W$ satisfies the conditions of \Cref{lemma: infinite-positive-filling-words}. 
    \item Let $n=2$.  Let $\{a, b\}$ be a basis for $\free_2$, $w=a^2b^2a$, and $W=\prod_{i=1}^\infty w$.  Then $W$ satisfies the conditions of \Cref{lemma: infinite-positive-filling-words}.
   \end{enumerate}
\end{example}

\begin{definition}
A \emph{boundary arc} of a diagram $D$ is a component of the intersection of the boundary of $D$ and the boundary of a single region $R\subset D$.  Boundary arcs are labeled by subwords $u\subset\lab(\partial D)$ such that  $u$ is also a subword of $\lab(\partial R)$.  See \Cref{fig:boundary arcs}. \begin{figure}
        \centering
        \includegraphics[scale=.5]{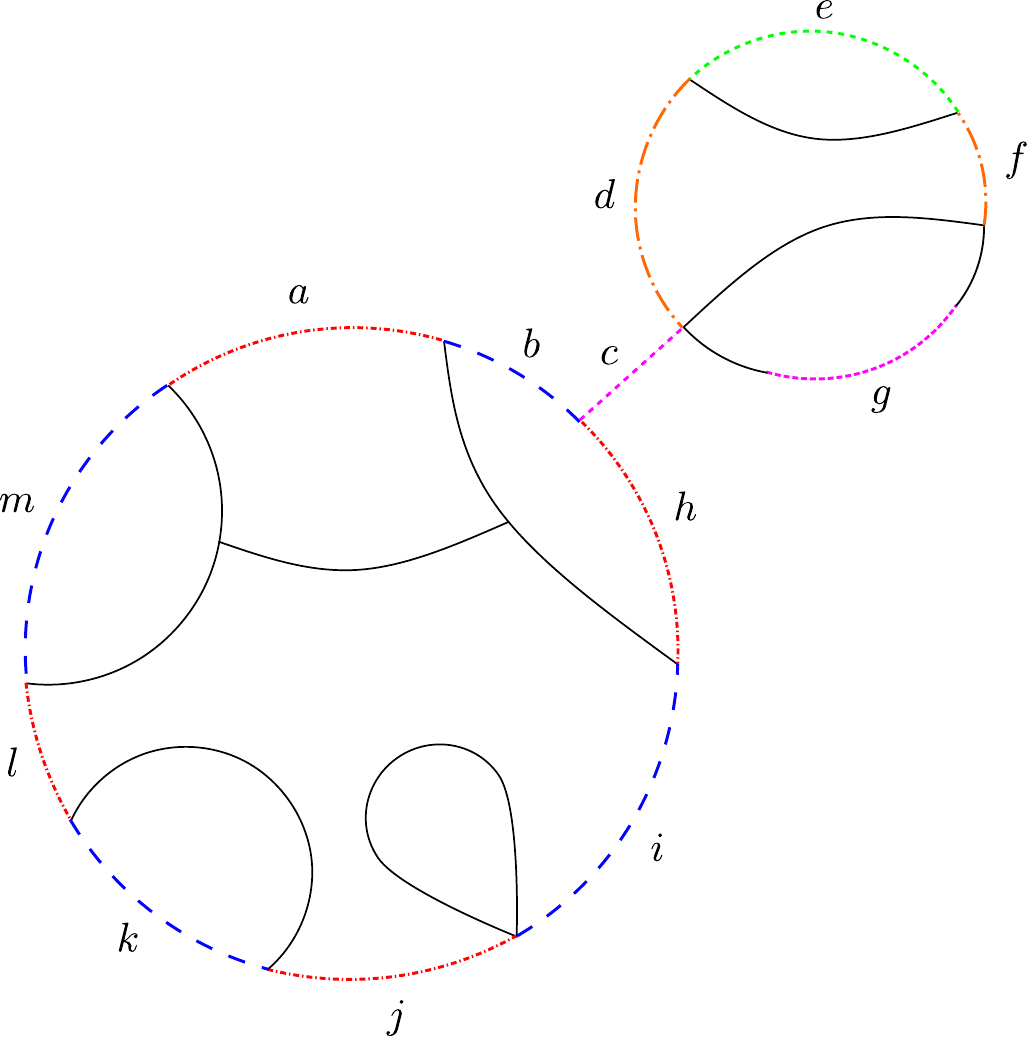}
        \caption{A diagram $D$.  Every labeled arc on the boundary of $D$ is a boundary arc except the pink arcs $c$ and $g$. The arc $c$ is not a boundary arc because it does not lie on the boundary of a region (see \Cref{remark: disk gluing}), and $g$ is not a boundary arc because it is not a connected component of the intersection of the boundary of a region with the boundary of $D$.}
        \label{fig:boundary arcs}
    \end{figure}
\end{definition}\noindent By \Cref{remark: disk gluing}, we glue regions along cycles labeled by cyclically reduced words, so cut edges of diagrams are \emph{not} boundary arcs. 

Let $\mathcal{P}=\langle X_n\mid R\rangle$ be a group presentation. Our next proposition gives an upper bound on the number of boundary arcs of a ``nice" Van Kampen diagram over $\mathcal{P}$ of area $a$ (see below for the definition of \emph{nice}). Let $(w_i)_i$ be an infinite sequence of positive words of increasing length, and suppose that there is an upper bound on the areas of minimal diagrams $D_i$ for $w_i$ over $\mathcal{P}$. Then the maximal length of a boundary arc of $D_i$ approaches infinity as $i\rightarrow\infty.$

In the rest of the paper, it will be helpful to consider diagrams with no valence-two vertices. Of course, Van Kampen diagrams normally have many valence-two vertices, but a Van Kampen diagram with valence-two vertices is topologically equivalent to a 2-complex without them. Suppose $v$ is a valence-two vertex in a Van Kampen diagram $D$ with incident edges $e_1$ and $e_2$ satisfying $\tau(e_1)=\iota(e_2)=v$.  We can eliminate $v$ by replacing the path of length two labeled by $e_1e_2$ with a single edge $e$ of length two labeled by $\lab(e_1)\lab(e_2).$

\begin{definition}\label{def:nice}
    A diagram is \emph{nice} if it has no valence-2 vertices. A \emph{nice diagram} is a  2-complex resulting from eliminating valence-two vertices from a disk diagram in the manner described above. 
\end{definition}

For disk diagrams $D$, it will be convenient to be able to refer to certain collections of regions of $D$. 

\begin{definition}  
Recall that a cut vertex or edge of a diagram $D$ is a vertex or edge of $D$ whose removal disconnects $D$. A \emph{continent} of a diagram $D$ is a connected component of \newline $D\setminus\{\text{cut vertices and edges of } $D$\}$. 
\end{definition}

\begin{prop}\label{prop:arcs of disk diagrams}
    Let $D$ be a disk diagram of area $a\geq 2$ with no vertex of degree one. Then the number of boundary arcs of $D$ is at most $2(a-1)$.
\end{prop}

\begin{proof}
     All diagrams in this proof are nice (see \Cref{def:nice}). If a disk diagram $D$ has separating edges, deleting these edges and identifying their endpoints results in a disk diagram with the same boundary arcs as $D$, so we may assume that $D$ has no separating edges. The proof is by induction on $a$.  For $a=2,$ $D$ is homotopy equivalent to one of the diagrams in \Cref{fig:diagrams-area-two}, all of which have at most two boundary arcs. Now let $a>2$ and suppose that every disk diagram with no degree-one vertex of area $b\geq2$ where $b<a$ has at most $2(b-1)$ boundary arcs. Let $D$ be a disk diagram of area $a$ without a degree-one vertex. The proof now splits into three cases: in the first case, $D$ has an interior edge $e$ on the boundary of distinct regions such that the closure of $e$ does not intersect $\partial D$; in the second case, $D$ has an interior edge $e$ on the boundary of distinct regions such that the closure of $e$ intersects $\partial D$; in the third case, $D$ is a wedge of disks.  We note that there are no other cases: an edge is either part of the boundary of zero, one, or two distinct regions.  An edge of a disk diagram not lying on the boundary of any region is a separating edge, which we assumed did not exist. The case where $D$ has an interior edge lying on the boundary of one region is impossible as $D$ has no vertices of degree one. Thus all edges of the nice diagram $D$ are external and lie on the boundary of one region, so $D$ is a wedge of disks.

     \begin{figure}
        \centering
        \includegraphics[scale=.7]{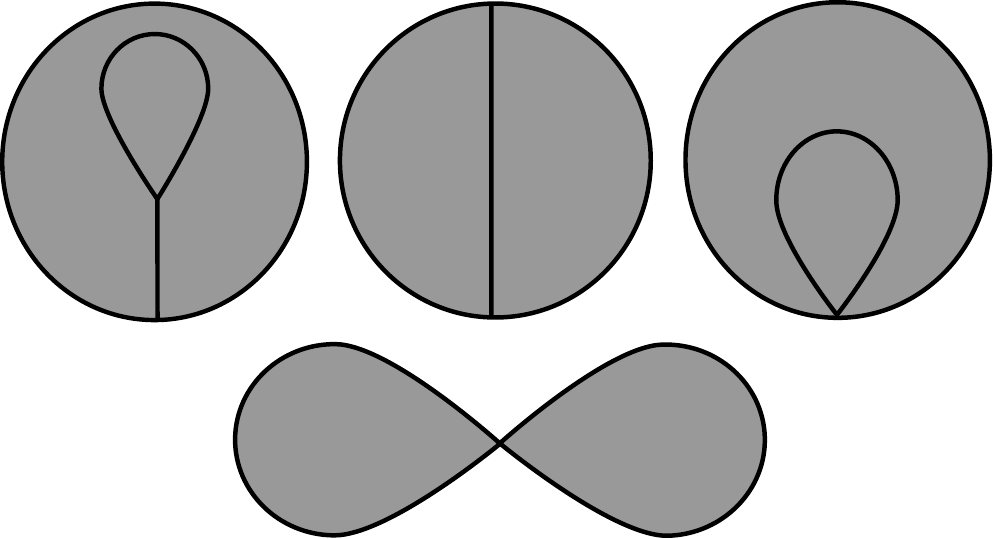}
        \caption{Disk diagrams of area two with no cut edges.}
        \label{fig:diagrams-area-two}
    \end{figure} 
  
    \noindent \textbf{Case 1: $D$ has an edge $e$ on the boundary of two distinct regions $R_1$ and $R_2$, and the closure of $e$ lies on the interior of $D$.} Removing $e$ from $D$ results in a 2-complex of area $a-1$ with a region $R=R_1\cup R_2$. If removal of $e$ results in a complex with no degree-one vertices and connected one-skeleton, the 2-complex is a disk diagram that by induction has at most $2(a-2)$ boundary arcs by induction, and therefore so is $D$. We claim the following: 
    \begin{enumerate}
        \item Removal of $e$ does not disconnect the 1-skeleton of $D;$
        \item If removal of $e$ results in a 2-complex with a degree-one vertex, the 2-complex has only one such vertex. 
    \end{enumerate}  
    
    For 1), suppose toward a contradiction that removal of $e$ did disconnect $D^{(1)}$, and let the components of $D^{(1)}\setminus\{e\}$ be $C_1$ and $C_2$.  One component, without loss of generality $C_1$, contains the boundary of the continent containing $e$, and $C_2$ lies inside this boundary as the closure of $e$ is in the interior of $D$.  Since $C_1$ and $C_2$ are disconnected after removing $e$ and $C_2$ lies inside $C_1$, there is a single region inside the continent containing $C_2$ which surrounds $C_2$ and intersects no edges of $D\setminus{e}$.  The edge $e$ is therefore on the boundary of this and no other region of $D$, which is a contradiction.

    For claim 2), recall that we assume all diagrams in this proof are nice. Therefore, if the removal of $e$ creates a degree-one vertex $v$ in the resulting 2-complex, the endpoints of $e$ are the same point, and this point is $v$.  Retract the edge $f$ incident to $v$ to its opposite endpoint $w$ to obtain a disk diagram $D'$ of area strictly less than $a$. Supposing $w$ is not on the boundary of $D$, removal of $e$ and the retraction of $f$ to $w$ does not change the number of boundary arcs, so $D$ has at most $2(a-2)$ boundary arcs by induction.  If $w$ is on the boundary of $D$, then this process removes at most one boundary arc.  By induction, $D'$ has at most $2(a-2)$ boundary arcs, so $D$ has at most $2(a-2)+1$ boundary arcs.\newline
    \noindent\textbf{Case 2: $D$ has an interior edge $e$ on the boundary of two distinct regions, and the closure of $e$ intersects the boundary of $D$}. Remove $e$ from $D$ to obtain a $2$-complex $D'$ of area $a-1$ and make $D'$ nice by eliminating valence-$2$ vertices.  By the argument above, removal of $e$ does not disconnect $D^{(1)}$ as $e$ is on the boundary of distinct regions. Additionally, since $D$ is nice, removal of $e$ does not create valence-$1$ vertices in $D'.$ Therefore $D'$ is a disk diagram with no valence-$1$ vertex. Consider how many boundary arcs are added to $D'$ when we replace $e$ to create $D$.  If $e$ is attached at vertices of the nice diagram $D'$, the number of boundary arcs of $D'$ does not change when we replace $e$, so by induction, $D$ has at most $2(a-2)$ boundary arcs.  Each endpoint of $e$ attached at the interior of an edge of $D'$ increases the number of boundary arcs by one.  By induction, $D'$ has at most $2(a-2)$ boundary arcs, so $D$ has at most $2(a-2)+2=2(a-1)$ boundary arcs.
\newline
    \noindent\textbf{Case 3: $D$ is a wedge of disks.} Construct the spine $S$ of $D$ as follows: for each region $R$ of $D$, there is one vertex of $S$ on the interior of $R$. $S$ also has a vertex at each wedge point of $D$.  The edges of $S$ are drawn between the vertices on the interior of regions $R$ of $D$ and the vertices at wedge points on the boundary of $R$. The disk diagram $D$ deformation retracts onto $S$, so $S$ is a tree. The tree $S$ is finite, so it contains degree-one vertices. Each degree-one vertex $v$ of $S$ corresponds to the interior of a region $R_v$ of $D$. Removing a $R_v$ and its boundary results in a disk diagram $D'$ of area $a-1$, and replacing $R_v$ and its boundary to reconstruct $D$ adds one or two boundary arcs, depending on whether $\partial R_v$ intersects the nice diagram $D'$ at a vertex or on the interior of an edge.  By induction, the proof is finished.
\end{proof}

By \Cref{lemma:paths=diagrams}, paths from $1$ to $[w]$ in $\Wh(\Gamma)$ are represented by Van Kampen diagrams over $\langle X_n\mid \Csep\cap\pi_1(\Gamma)\rangle$.  Paths between vertices $[v]$ and $[w]$ in $\Wh(\Gamma)$ where neither $[v]$ nor $[w]$ is the trivial vertex can also be visualized as planar $2$-complexes, as the following lemma demonstrates. Recall that our definitions of Van Kampen diagram (\Cref{def:Van_Kampen_diagram}) and annular diagram over a presentation (\Cref{def:annular-diagram}) require regions with immersed boundaries, and therefore the boundaries of regions of diagrams are assumed throughout the paper to be labeled by cyclically reduced words.

\begin{lemma}\label{lemma:existence-annular}
Let $[v], [w]\in \Wh(\Gamma)$ be vertices in the same component of $\Wh(\Gamma)$ such that $[v]\neq 1$ and $[w]\neq 1$. Let $p([v], [w])\in\pathwh(\Gamma)$ be a path between $[v]$ and $[w]$ with edges labeled $[\alpha_1]$, $[\alpha_2]$, \ldots, $[\alpha_l]$ representing the equation $\hat{v}\hat{\alpha_1}\hat{\alpha_2}\ldots\hat{\alpha_l}=\hat{w}$, where $\hat{v}\in[v]$, $\hat{\alpha_i}\in[\alpha_i]$, and $\hat{w}\in[w]$. Then at least one of the following is true.
\begin{enumerate}
    \item There is an annular diagram $A_p$ over $\langle X_n\mid \Csep\cap\pi_1(\Gamma)\rangle$ with boundaries labeled by (cyclically reduced) $\hat{v}\in[v]$ and $\hat{w}\in [w]$. The regions of $A_p$ have boundaries labeled by (cyclically reduced) representatives of some of the conjugacy classes labeling the edges of $p$. If $p([v], [w])$ is a path of minimal length from $[v]$ to $[w]$ in $\Wh(\Gamma)$, then the regions of $A_p$ have boundaries labeled by (cyclically reduced) representatives of all of the conjugacy classes labeling the edges of $p$ with multiplicity. If $[w]$ and $[v]$ have positive representatives, $A_p$ has no cut edges. If $[v]$ and $[w]$ are not in the component of $\Wh(\Gamma)$ containing the trivial vertex, $p$ is represented by an annular diagram.
    \item The length of $p$ is at least $\seplength{v}+\seplength{w}.$ The path $p$ is represented by a wedge of $D_p$ Van Kampen diagrams $D_1$ with boundary labeled $\inv{v}$ and $D_2$ with boundary labeled $w$ over $\langle X_n\mid \Csep\cap\pi_1(\Gamma)\rangle$. The regions of $D_p$ are labeled by (cyclically reduced) representatives of some of the conjugacy classes labeling the edges of $p$.  If $p$ is a path of minimal length, then the regions of $D_p$ are have boundaries labeled by (cyclically reduced) representatives of all of the conjugacy classes labeling the edges of $p$ with multiplicity. 
\end{enumerate}
\end{lemma}

\begin{proof}
Similar to the construction of the diagram in \Cref{lemma:paths=diagrams}, we wedge disks with boundaries labeled by $\hat{v}$, $\hat{\alpha_1}$, $\hat{\alpha_2}$, \ldots, $\hat{\alpha_l}$ in clockwise order at a point. The resulting 2-complex has boundary which freely reduces to $\hat{w}$ when read clockwise beginning at the wedge point. Then we fold adjacent edges in the clockwise ordering in the sense of Stallings.  As in \Cref{lemma:paths=diagrams}, if there is a point in the folding at which a disk $T$ appears such that the boundary of $T$ freely reduces to $1$, delete the interior of $T$ and fold the boundary edges of $T$ together. 

If no such $T$ contains the region labeled by $\hat{v}$, then the deletion of the interior of $T$ only eliminates regions labeled by the $\hat{\alpha_i}$. Deletions of regions exclusively labeled by the $\hat{\alpha_i}$ imply that $p$ is not a path of minimal length between $[v]$ and $[w]$ in $\Wh(\Gamma)$ as some subproduct of the $\hat{a_i}$ freely reduces to $1$.  If only deletions of regions labeled by the $\hat{\alpha_i}$ occur, then we obtain $A_p$ with boundaries $\hat{v}$ and $\hat{w}$ by deleting the interior of the region labeled $\hat{v}$ and any degree-one vertices and edges incident to degree-one vertices. If no deletions of regions occur, then every edge on $p([v], [w])$ is represented by a distinct region of $A_p$. See \Cref{fig:annular-path} for an example of an annular diagram corresponding to a specific path in $\Wh(\rose_n).$ 

If $[v]$ and $[w]$ have positive representatives, then all cyclically reduced representatives of $[v]$ and $[w]$ are positive.  If there were a cut edge in $A_p$, then this edge is on the boundary of $A_p$ as no internal edge of an annular diagram is a cut edge.  The cut edge cannot be shared by both boundaries of $A_p$ as cutting an edge on both boundaries does not disconnect $A_p$. Therefore this cut edge is an external edge on one of the boundaries of $A_p$.  It follows from the argument of \Cref{lemma: positive-words-wedge-disks} that no such edge can exist.

Now suppose that a disk subcomplex $T$ appears during the folding process such that the boundary of $T$ freely reduces to $1$ and that $T$ contains the region labeled $\hat{v}$.  This implies that a product of $\hat{v}$ and some of the elements of the set $\{\hat{\alpha_1}, \hat{\alpha_2}, \ldots, \hat{\alpha_l}\}$ is equal to $1$ in $\free_n$. Note that it also implies that $[v]$ and $[w]$ are both in the component of $\Wh(\Gamma)$ containing $1$, since $\hat{v}$ is then generated by separable elements and a path from $[v]$ to $[w]$ exists in $\Wh(\Gamma)$. Folding the boundary of $T$ together (obtaining a sphere) and then deleting the interior of the $\hat{v}$ region yields a Van Kampen diagram for $\inv{\hat{v}}$ over $\langle X_n\mid\Csep\cap\pi_1(\Gamma)\rangle$, so the area of $T$ is at least $1+\seplength{v}$ (the extra $1$ is for the region labeled by $v$). The product of the $\hat{\alpha_i}$ not on the interior of $T$ is $\hat{w}$, so there are at least $\seplength{w}$ such $\hat{\alpha_i}$ remaining.  Thus the length of $p$ is at least $\seplength{v}+\seplength{w}$.

\end{proof}

\begin{figure}
    \centering
    \includegraphics[scale=.47]{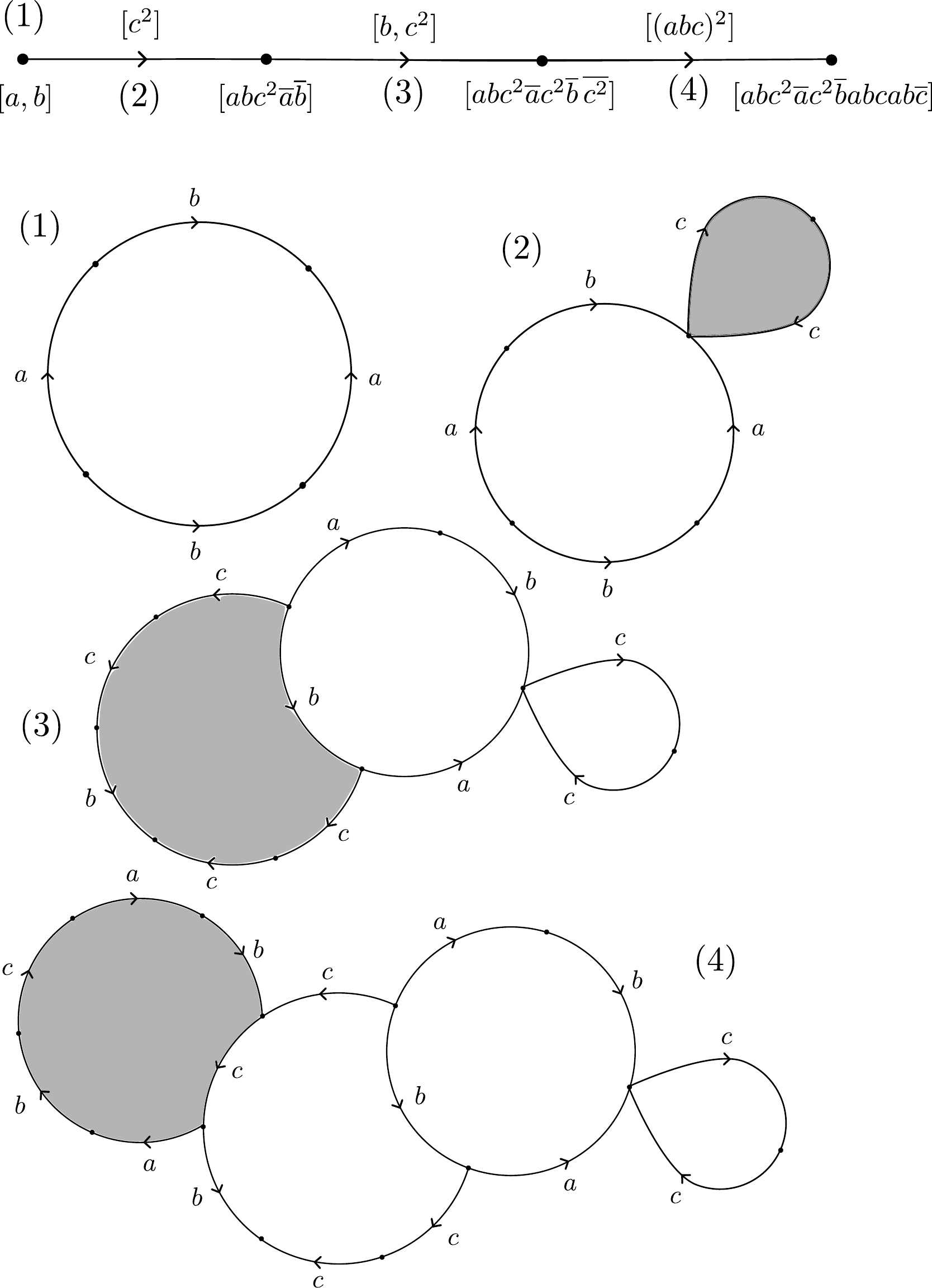}
    \caption{A path from $[a, b]$ to $[abc^2\inv{a}c^2\inv{b}abcab\inv{c}]$ in $\Wh(\rose_n)$ and a corresponding annular diagram.}
    \label{fig:annular-path}
\end{figure}

Our next lemma is very similar to \Cref{prop:arcs of disk diagrams}. It gives a bound on the number of boundary arcs of an annular diagram of area $a$.  This will be used to show that every component of $\Wh(\Gamma)$ has infinite diameter.

\begin{figure}
    \centering
    \includegraphics[scale=.6]{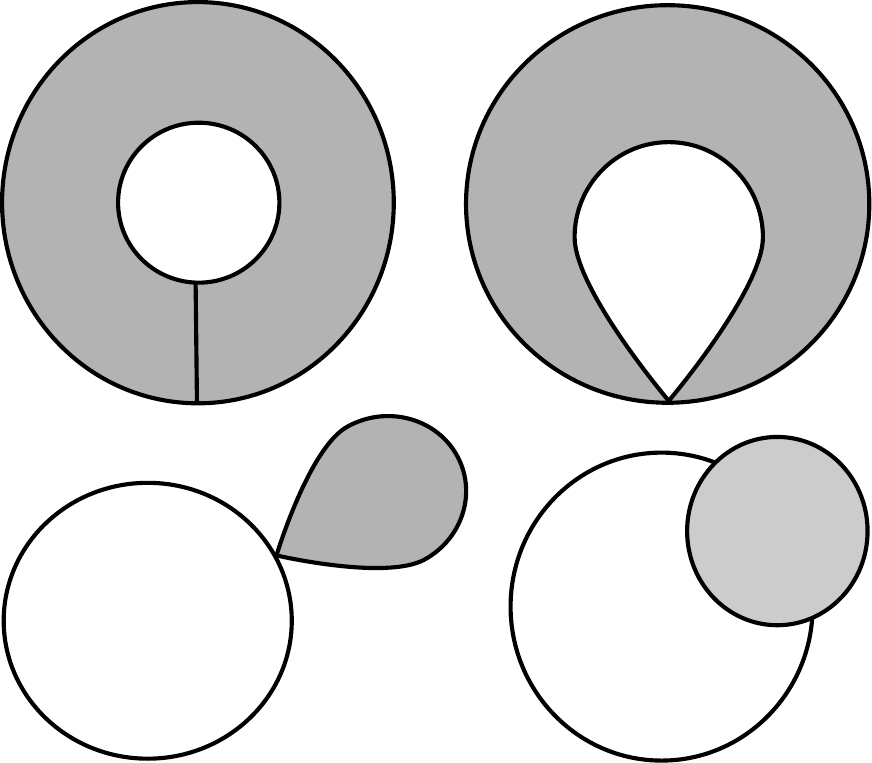}
    \caption{The annular diagrams of area one with no cut edges.}
    \label{fig:annular-diagrams-area-one}
\end{figure}

\begin{lemma}\label{lemma: boundary arcs of annular diagram}
Let $A$ be an annular diagram of area $a$ with no valence-$1$ vertices. Then the number of boundary arcs of $A$ is at most $2a$.
\end{lemma}

\begin{proof}
All diagrams in this proof are nice (\Cref{def:nice}). Like the proof of \Cref{prop:arcs of disk diagrams}, we induct on $a$. We assume that $A$ has no cut edges as we did in the proof of \Cref{prop:arcs of disk diagrams}. If $a=1$, $A$ is one of the diagrams in \Cref{fig:annular-diagrams-area-one}: note that each diagram has at most two boundary arcs. Suppose the conclusion holds for all annular diagrams of area less than $a>1$ and let $A$ be an annular diagram of area $a$. 

First, if there is an edge $e$ on the boundary of distinct regions of $A$, we remove this edge from $A$ to obtain a $2$-complex of area $a-1$. We must address the possibilities that the resulting $2$-complex either has disconnected $1$-skeleton or a valence-$1$ vertex. Suppose there is an edge $e$ in the interior of $A$ such that $e$ is on the boundary of two distinct regions, and suppose that removal of $e$ disconnects $A^{(1)}$. Then we claim that the boundaries of $A\setminus{e}$ are in the same connected component of $(A\setminus{e})^{(1)}$. If the boundaries were not in the same component, we could draw a closed loop $l$ on the interior of $A\setminus{e}$ so that $A$ deformation retracts to $l$ and $l$ does not intersect any edges of $A\setminus{e}$. Thus $e$ is on the boundary of only one region of $A$ (the region containing $l$).  The case where an internal edge $e$ disconnects $A^{(1)}$ and $e$ does not connect the boundaries of $A$ was handled in \Cref{prop:arcs of disk diagrams}, so we will not reproduce the argument here.  Similarly, if the removal of $e$ creates a valence-$1$ vertex in $A\setminus{e}$, the argument that $A$ has at most $2a$ boundary arcs is very similar to this case in \Cref{prop:arcs of disk diagrams}. 

Supposing $e$ bounds distinct regions of $A$ and the removal of $e$ results in an annular diagram with no valence-$1$ vertices, let $e$ bound regions $R_1$ and $R_2$.  The annular diagram $A\setminus{e}$ has the region $R=R_1\cup R_2$ and area $a-1$, so by induction, $A\setminus{e}$ has at most $2a-2$ boundary arcs.  Replacing $e$ adds at most $2$ boundary arcs, the number of which depends on how many endpoints of $e$ are attached at the interior of boundary edges of the nice diagram $A\setminus{e}$.

Now suppose there are no edges of $A$ on the boundary of two different regions. Construct the spine $S$ of $A$ as follows. The vertices of $S$ are of the following types:
\begin{itemize}
    \item There is one vertex of $S$ in the interior of every region of $A$;
    \item If $A$ contains a continent of area one like the first diagram in \Cref{fig:annular-diagrams-area-one}, the midpoint of the edge connecting the two boundaries of the continent is one vertex of $S$, and similarly if $A$ contains a continent like the second diagram in \Cref{fig:annular-diagrams-area-one}, the wedge point is a vertex of $S$;
    \item All wedge points between continents or between continents and edges of $A$ are vertices of $S.$
\end{itemize}
\noindent The edges of $S$ are described below.
\begin{itemize}
    \item If $A$ contains an annular region $R$ (see the first two diagrams of \Cref{fig:annular-diagrams-area-one}), draw two edges from the vertex on the interior of $R$ to the vertex of $S$ at the wedge point or interior edge of $R$ so that $R$ deformation retracts onto this part of the spine;
    \item For the remaining regions $R'$ of $A$, there is one edge of $S$ between the vertex interior to $R'$ and each wedge point on the boundary of $R'$.  We draw these edges so that they lie in the interior of $A$.
    \item Add edges of $A$ to $S$ which are not on the boundary of any region of $A$.
\end{itemize}
By construction, $S$ is a deformation retract of $A$. Suppose first that $S$ is a cycle with a nonzero number of finite trees attached. Every valence-$1$ vertex of $S$ corresponds to the interior of a region of $A$.  Let $v$ be such a vertex and $R$ its corresponding region.  Remove $R$ and its boundary from $A$ to obtain an annular diagram of area $a-1$.  By induction, the resulting annular diagram has at most $2(a-1)$ boundary arcs. Replacing $R$ and its boundary to reconstruct $A$ adds at most two boundary arcs (one corresponding to the boundary of $R$ and one more if $R$ is wedged at the interior of an edge of the nice diagram $A\setminus\{R\cup\partial R\}$).

If $S$ has no finite trees attached, $S$ is a cycle. Recall that $A$ is nice. Choose a vertex of $S$ corresponding to the interior of a region of $A$. Removing this region and one of the edges along its boundary reduces the number of boundary arcs by at most two, so by induction, the conclusion holds.
\end{proof}

We now prove that diameter of $\Wh(\rose_n)$ is infinite. Consider the presentation ${\mathcal{P}=\langle X_n\mid \Csep\rangle}$ of the trivial group.  Every expression for $w\in\free_n$ as a product of elements of $\Csep$ has an associated Van Kampen diagram $D$ for $w$ over $\mathcal{P}$. The key idea of the proof is that $\Wh(\rose_n)$ has finite diameter if and only if there is a universal bound on the areas of minimal diagrams for all words in $\free_n$ over $\langle X_n\mid \Csep\rangle$. We then take a sequence of initial subwords $(w_i)_i$ of infinite words given by \Cref{lemma: infinite-positive-filling-words} and consider a sequence of minimal nice Van Kampen diagrams $D_i$ for the $w_i$. We show that for large enough $i$, some boundary arc along a single region of $D_i$ has length at least $k$. This is a contradiction as subwords of $w_i$ of length at least $k$ are not subwords of any separable word. The proof that every component of $\Wh(\Gamma)$ has infinite diameter is similar but involves annular diagrams.

\infinitediameter*

\begin{proof}
    All diagrams in this proof are \emph{nice} (see \Cref{def:nice}). First, we prove that $\Wh(\rose_n)$ has infinite diameter. Take $W\in \widehat{\words_n^{\infty}}$ satisfying the conditions of \Cref{lemma: infinite-positive-filling-words}, so $W$ is positive, and for every initial subword $w\subset W$, every cyclic subword of $w$ of length at least $k$ is not a subword of a separable word. This implies no subword $u$ of $w$ such that $|u|\geq k$ can label a boundary arc of a diagram over $\langle X_n\mid\Csep\rangle$. Take a sequence $(w_i)_{i=1}^{\infty}$ of initial subwords of $W$ of strictly increasing length, and suppose toward a contradiction that $\Wh(\rose_n)$ has finite diameter.
    
    Let $m_i=\seplength{w_i}$ and $\prod_{j=1}^{m_i}\alpha_{i, j}$ be expressions for some positive $\hat{w_i}\in[w_i]$ as products of elements of $\Csep$. These products correspond to paths $p_i$ of length $m_i$ in $\Wh(\rose_n)$ with initial vertex $[1]$ and terminal vertex $[w_i]$.  By the finite diameter assumption, there is a universal bound on the lengths of minimal paths between vertices, hence there is a universal bound $B$ on $\seplength{w_i}$.  By \Cref{lemma:paths=diagrams}, these expressions for $w_i$ correspond to nice Van Kampen diagrams $D_i$ for $w_i$. The number of regions of $D_i$ is $m_i$.  By \Cref{prop:arcs of disk diagrams}, the number of boundary arcs of $D_i$ is bounded above by $2(B-1)$ for all $i$.

    Since all $\hat{w_i}$ are positive, $|\hat{w_i}|=|\partial D_i|$. Choose $i$ large enough so that $|\hat{w_i}|>2k(B-1)$.   By the pigeonhole principle, there is some boundary arc of $D_i$ of length at least $k$, which contradicts the assumption that each region in this diagram has boundary label in $\Csep$. Thus no universal bound on the areas of the $D_i$ exists. This concludes the proof that $\Wh(\rose_n)$ has infinite diameter.

    We move on to the proof for $\Gamma\neq\rose_n$. Observe first that every component $C$ of $\Wh(\Gamma)$ contains a vertex representing the conjugacy class of a positive word as $\Gamma$ represents a finite-index subgroup of $\free_n$. To see this, take a cyclically reduced $\hat{w}\in[w]\in V(C)$, and suppose $\hat{w}$ contains $y^{-m}$ as a subword where $y\in X_n.$ Let $l$ be the minimal power of $y$ such that $y^l$ lifts to a loop in $\Gamma$. Multiplying $\hat{w}$ by appropriate conjugates of $y^l$ creates a path from $[w]\in V(C)$ to $[w']\in V(C)$ where $w'$ satisfies the property that all subwords $y^{-k}$ of $w'$ satisfy $k<l$. Thus we assume $m<l$. Beginning at any vertex of $\Gamma$, the paths labeled $y^{-m}$ and $y^{l-m}$ end at the same vertex of $\Gamma$ as $\Gamma$ is regular. Taking this path in $\Gamma$ is equivalent to multiplying $w'$ by an appropriate conjugate of $y^l$, which is equivalent to traveling from $w'$ along an edge labeled $[y^l]$ in $\Wh(\Gamma).$ Thus we can find a vertex in $C$ which has no copies of $\inv{y}$ where $y$ is any element of $X_n$.
    
    Let $C$ be a component of of $\Wh(\Gamma)$ and $w$ a positive word so that $[w]\in C$. Similar to \Cref{lemma: infinite-positive-filling-words}, we can let $v$ be a product of positive powers of elements of $X_n$ which lift to loops in $\Gamma$ such that the Whitehead graph of $v$ is connected, has no cut vertices, and has the property that removing any edge from the Whitehead graph of $v$ results in a Whitehead graph with no cut vertices. As was the case in \Cref{lemma: infinite-positive-filling-words}, if $v$ contains two copies of all positive length-two words in the generators $X_n$, $v$ is necessarily such a word. See \Cref{example:v-for-small-n} for examples of words labeling loops in $\Gamma$ for small $n.$ By \Cref{lemma: subwords-separable-words}, $v$ is not a subword of any separable word.  Define $W=w\prod_{i=1}^{\infty}v$, and let $w_i = wv^i$, so that $\{w_i\}_{i=1}^{\infty}$ is an infinite sequence of initial subwords of $W.$ These subwords are represented by paths in $C$ beginning at $[w]$ and ending at $[w_i]$ as we ensured $v$ is a product of words labeling loops in $\Gamma$. Let $k=4(|w|+|v|).$ Observe every subword of $W$ of length $\frac{k}{2}$ is not a subword of any separable word as it contains a cyclic permutation of $v$. Consider the sequence of diagrams representing geodesic segments from $[w]$ to the $[w_i]$ in $C$. Then, as in the proof that $\Wh(\rose_n)$ has infinite diameter, if there is some upper bound $B$ on lengths of these paths, then $B$ is also an upper bound on the areas of diagrams representing these paths. Paths from $[w]$ to $[w_i]$ in $C$ are either represented by annular diagrams over $\mathcal{P}=\langle X_n\mid \Csep\cap \pi_1(\Gamma)\rangle$ with boundaries labeled $[w]$ and $[w_i]$ or by a wedge of Van Kampen diagrams over $\mathcal{P}$ with boundaries labeled $[\inv{w}]$ and $[w_i]$.
    
    In the former case, by \Cref{lemma: boundary arcs of annular diagram}, each diagram in the sequence has at most $2B$ boundary arcs. Recall that, as in the bottom annular diagrams of \Cref{fig:annular-diagrams-area-one}, the inner and outer boundaries of an annular diagram may share boundary \emph{edges} which are not part of the boundary of any region of the diagram. Boundary \emph{arcs}, by contrast, are by definition on the boundaries of regions of diagrams. If the two boundaries of the diagram do not share edges, the sum of the lengths of the boundaries, $2|w|+|w_i|$, will eventually be larger than $Bk$, so some boundary arc will have length at least $\frac{k}{2}$. If the two boundaries do share edges,  the sum of the lengths of their shared edges is at most the length of $w$, so the sum of the lengths of the boundary arcs is at least $|w_i|$. Thus there is an $i$ large enough so that some boundary arc must have length at least $\frac{k}{2}$. Any subword of $w_i$ of length $\frac{k}{2}$ cannot be a subword of a separable word, so this is a contradiction. 
    
    In the latter case, paths are represented wedge sums of Van Kampen diagrams with boundaries $\inv{w}$ and $w_i$. The diagram $D$ labeled by $w_i$ has fewer than $2(B-1)$ boundary arcs, which is again a contradiction as there is an $i$ large enough so that $|w_i|>(B-1)k$, implying it contains a boundary arc of length at least $\frac{k}{2}$. This boundary arc is labeled by a word which cannot be a subword of any separable word. \qedhere
\end{proof}

\begin{example} \label{example:v-for-small-n}
Suppose that $\Gamma$ is a finite regular cover of $\rose_n$ of index $l$. The following words are loops in $\Gamma$ which are not subwords of any separable word by \Cref{lemma: subwords-separable-words}.
    \begin{enumerate}
        \item For $n=2$, $v=(a^lb^l)^2$;
        \item For $n=3,$ $v=a^lb^lc^lb^la^lc^l$;
        \item For $n=4,$ $v=a^lb^lc^ld^lb^la^lc^lb^ld^lc^la^ld^l$.
    \end{enumerate}
\end{example}

\noindent A corollary of \Cref{thm:infinite-diameter} is below.  Note that this result was previously stated in a 2003 paper of Bardakov, Shpilrain, and Tolstykh (see Theorem 2.1 of \cite{bardakov2003palindromic}). 

\cayinfdiameter*

\begin{proof}
Observe that $\Wh(\rose_n)$ can be obtained as a quotient of $\Cay(\free_n, \Csep)$.  First, identify all vertices $v, w\in\Cay(\free_n, \Csep)$ satisfying $\lab(v)\in[\lab(w)]$. Here, the label of ${v\in\Cay(\free_n, \Csep)}$ is the element of $\free_n$ it represents. The result is a graph with edges labeled by elements of $\Csep$ and vertices labeled by conjugacy classes of $\free_n$.  Then, for all edges $e_1, e_2$ in the resulting graph satisfying $\lab(e_1)\in[\lab(e_2)]$, $\iota(e_1)=\iota(e_2),$ and $\tau(e_1)=\tau(e_2)$, identify $e_1$ and $e_2$. Relabel all edges $e$ in the resulting graph by $[\lab(e)]$ to obtain $\Wh(\rose_n)$.  This quotient map is 1-Lipschitz, so since $\Wh(\rose_n)$ has infinite diameter, so does $\Cay(\free_n, \Csep).$ The set $\Csep$ contains $\curves\prim$, so $\Cay(\free_n, \curves\prim)$ has infinite diameter, as well.
\end{proof}

\section{Nonhyperbolicity of the separability complex}\label{sec: nonhyperbolicity}

In this section, we prove that $\Wh(\rose_n)$ is nonhyperbolic for $n>1$ by finding, for each $n$, an infinite set of arbitrarily fat geodesic triangles. Before diving into the technicalities we give a detailed outline of the proof to serve as a guide to the reader. 

\medskip

\begin{enumerate}[wide, labelindent=0pt]
    \item[{\bf Step 1:}] We find a particular infinite set of positive separable words $\{r_i(n)\}_{i=1}^{\infty}$ such that none of them is contained in a free factor of rank less than $n-1$. For $n>2$, we show that if $i\neq j$ and $r_i$ and $r_j$ are subwords of a cyclically reduced $w\in\free_n$, then $w$ is inseparable. For $n=2$, we show that if $|j-i|>1$ and $r_i$ and $r_j$ are subwords of a cyclically reduced $w\in\free_n$, then $w$ is inseparable (see \Cref{lemma: only one ri}).

    \medskip
    
    \item[{\bf Step 2:}] For $n>2$, let \[w(a, k, s, n):=\prod_{i=a}^{a+s-1} r_i(n)^k,\] and for $n=2$, let \[w(a, k, s, n):=\prod_{i=0}^{s-1} r_{a+2i}(n)^k.\] We show that, for all $n$, and large enough $a$ and $k$, the word $w(a, k, s, n)$ has separable length $s$ (see \Cref{prop: seplength-w}). To get the lower bound, let $D$ be a nice Van Kampen diagram for $w(a, k, s, n)$ over $\langle X_n\mid \Csep\rangle$, and suppose toward a contradiction that $\area(D)<s$. By \Cref{prop:arcs of disk diagrams}, Van Kampen diagrams of area less than $s$ have at most $2(s-2)$ boundary arcs.  As a consequence of Step 1, we obtain an upper bound on the length of any boundary arc of $D$. Since $w(a, k, s, n)$ is positive, the length of  $w(a, k, s, n)$ is the sum of the lengths of the boundary arcs of $D$. Given sufficiently large $a$ and $k$, we show that the sum of the lengths of all of the boundary arcs of $D$ is less than the length of $w(a, k, s, n)$, a contradiction. 

    \medskip
    
    \item[{\bf Step 3:}] For $n>2$, let $b=a+s$, and for $n=2$, let $b=a+2s.$ We show that, for $a$ and $k$ large enough, the distance $\dsep{[w(a, k, s, n)]}{[w(b, k, s, n)]}=2s$ using a proof similar to the proof that $\seplength{w(a, k, s, n)}=s$. By \Cref{lemma:existence-annular}, a path in $\Wh(\rose_n)$ from $[w(a, k, s, n)]$ to $[w(b, k, s, n)]$ of length less than $\seplength{v}+\seplength{w}$ corresponds to an annular diagram with boundaries labeled by the positive words $w(a, k, s, n)$ and $w(b, k, s, n)$. We suppose toward a contradiction that $A$ is an annular diagram of area less than $2s$ with boundaries labeled as above. Let $O$ be the total length of the intersection of the two boundaries of $A$.  In \Cref{prop: dsepab}, we show that there are $a$ and $k$ sufficiently large so that the length of the boundary of $A$ is less than $|w(a, k, s, n)|+|w(b, k, s, n)|-O$. Here, by \emph{length of the boundary of $A$}, we mean the sum of the lengths of all edges on the boundary of $A$. Recall that positivity implies there are no cut edges of $A$, so the length of the boundary of an annular diagram with boundaries labeled by $w(a, k, s, n)$ and $w(b, k, s, n)$ is equal to $|w(a, k, s, n)|+|w(b, k, s, n)|-O$. Thus there are $a$ and $k$ sufficiently large such that there is no path in $\Wh(\rose_n)$ from $[w(a, k, s, n)]$ to $[w(b, k, s, n)]$ of length less than $2s$.

    \medskip
    
    \item[{\bf Step 4:}] There are at least two paths of length $2s$ from $[w(a, k, s, n)]$ to $[w(b, k, s, n)]$, one of which goes through $1\in\Wh(\rose_n)$, and the other goes through \[[w(a, k, s, n)w(b, k, s, n)]=[w(a, k, 2s, n)]\] (see \Cref{fig:fat-triangle}). The latter path lies above a ball of radius $s$ around $1$ by \Cref{prop: seplength-w}.
    For $\delta < s$, the vertex $[w(a, k, 2s, n)]$ lies outside the $\delta$-neighborhood of the sides $([1], [w(a, k, s, n)])$ and $([1], [w(b, k, s, n)])$. Since these fat triangles exist for all separable lengths, and since there are elements of $\free_n$ with arbitrarily large separable length, we conclude that there is no $\delta$ satisfying the $\delta$-thin condition for hyperbolicity, completing the proof of \Cref{theorem:nonhyperbolicity}. 
\end{enumerate} 

\begin{figure}
    \centering
    \includegraphics[scale=.4]{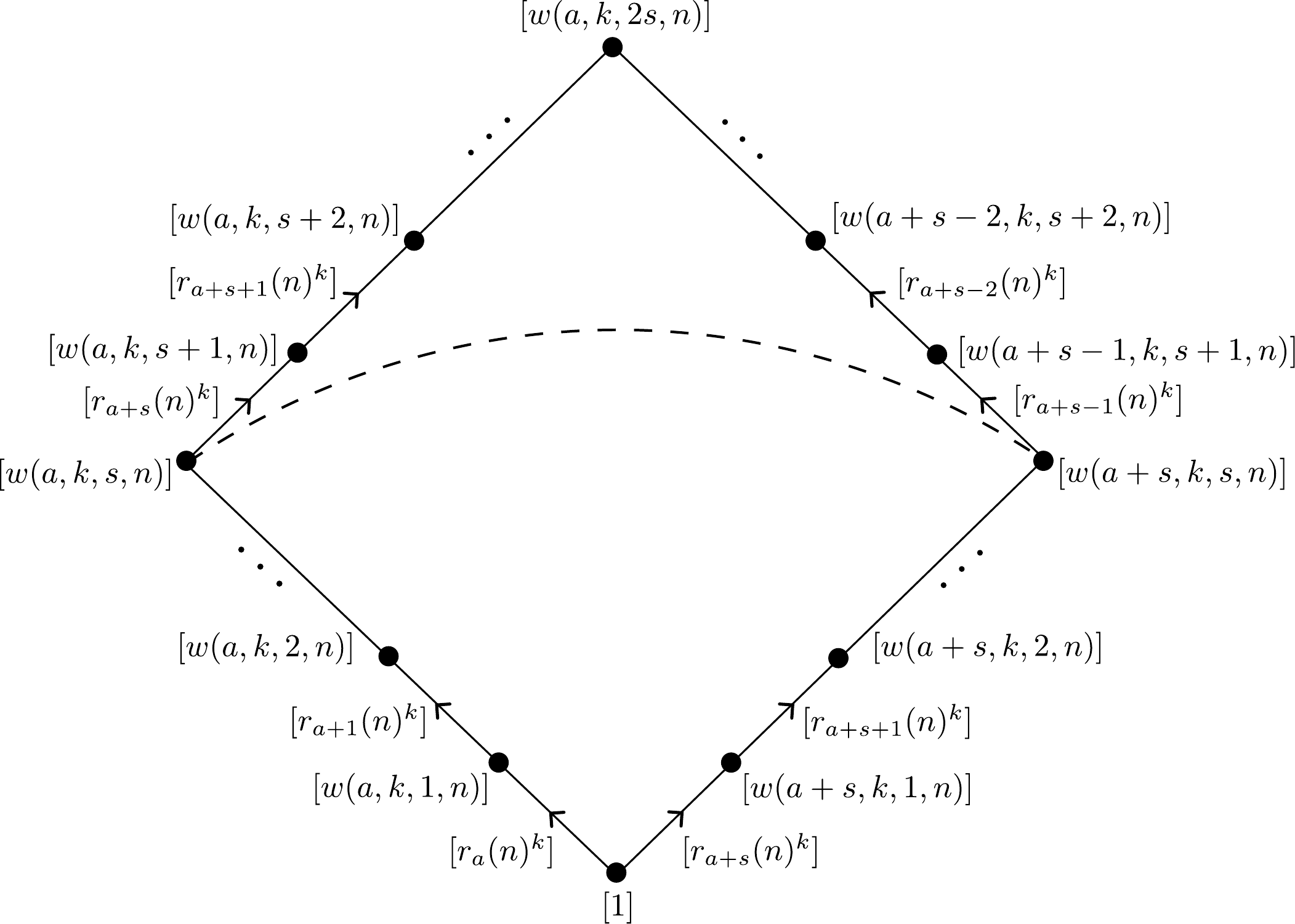}
    \caption{We will show that, for $n>2,$ this is a geodesic triangle on the vertices $\{[1], w(a, k, s, n), w(b, k, s, n)\}$ in $\Wh(\rose_n)$. The dashed line represents the boundary of a ball of radius $s$ centered at the vertex $[1]$. A similar triangle with slightly different vertex and edge labels exists for $n=2$. See \Cref{fig:fat-triangle-n-2}.}  
    \label{fig:fat-triangle}
\end{figure}

\subsection{Step 1: Properties of $\mathbf{r_i(n)}$}
For $n>3$ define:  
\[
    r_i(n):= (x_1^2x_2^2\ldots x_{n-2}^2x_{n-1}^ix_n)^2x_{n-1}^2x_{n-2}^2\ldots x_2^2;
\]
for $n=3$, define:
\[  
    r_i(n):= \big((x_1^2x_2^ix_3)^2x_2^2\big)^2;
\]
for $n=2,$ define:
\[  
    r_i(n):=(x_1^ix_2)^2.
\]
We will refer to $r_i(n)$ as $r_i$ when $n$ is clear from context. 
The need for separate cases for $n=3$ and $n=2$ will become clear. Let $i\geq 0$ and $n>2$. The set \[\{x_1, x_2, \ldots, x_{n-1}, x_1^2x_2^2\cdot\ldots\cdot x_{n-2}^2x_{n-1}^ix_n\}\] is a basis for $\free_n$. For $n=2$, $\{x_1, x_1^2x_2\}$ is a basis for $\free_n$. Since $r_i$ is a product of squares of $n-1$ of these basis elements, $r_i$ is separable for all $n$. Note that we take $r_i$ to be a product of squares because this guarantees the presence of edges between $x_j$ and $\inv{x_j}$ for $j<n$ in the Whitehead profile of $r_i$. This will matter when we prove that no separable word can have both $r_i$ and $r_j$ as subwords if $i\neq j$ (see \Cref{lemma: only one ri}).

Let $\phi$ be the following automorphism of $\free_n$:

\[   \phi(x_i) = 
     \begin{cases}
        x_{n-1}x_i\inv{x_{n-1}} &\quad\text{if } i<n-1;  \\
        x_{n-1} &\quad\text{if } i=n-1; \\
        x_n\inv{x_{n-1}} &\quad\text{if } i=n. \\
     \end{cases}
\]

\noindent Recall that $\words_n$ is the set of finite, possibly unreduced words with letters in $X_n^*$ and that an automorphism $\sigma\in\Aut(\free_n)$ induces a map $\Tilde{\sigma}:\words_n\rightarrow\words_n$ which sends $w=y_1y_2\ldots y_k$ for $y_i\in X_n^*$ to $\phi(y_1)\phi(y_2)\ldots\phi(y_k)$. For $w\in\words_n$, $\Tilde{\sigma}(w)$ is the possibly freely unreduced image of $w$ under $\Tilde{\sigma}$, and $\widehat{\Tilde{\sigma}(w)}$ is the free reduction of $\Tilde{\sigma}(w)$ in $\words_n.$
\begin{lemma}\label{lemma:phi-only-cancels-xn-1}
    Let $w\in\words_n$ be freely reduced and let $\Tilde{\phi}$ be the map $\words_n\rightarrow\words_n$ corresponding to $\phi$.  Then the only letters in $\Tilde{\phi}(w)$ not present in $\widehat{\Tilde{\phi}(w)}$ are copies of $x_{n-1}$ and its inverse.
\end{lemma}

\begin{proof}
Observe that the inverse of $\phi$ is the map $\inv{\phi}$ below:
\[   \inv{\phi}(x_i) = 
     \begin{cases}
        \inv{x_{n-1}}x_ix_{n-1} &\quad\text{if } i<n-1;  \\
        x_{n-1} &\quad\text{if } i=n-1; \\
        x_nx_{n-1} &\quad\text{if } i=n. \\
     \end{cases}
\]  Let $y_i\in\{x_i^{\pm 1}\}$ where $i$ satisfies $1\leq i \leq n$ and $i\neq n-1$.  Note that $\phi(y_i)$ and $\inv{\phi}(y_i)$ both contain exactly one copy of $y_i$. Also, if $y_i$ is contained in $\phi(y)$ or $\inv{\phi}(y)$ for $y\in X_n^*$, then $y=y_i$. Thus every copy of $y_i$ in $\Tilde{\phi}(w)$ is contained in the image of a distinct copy of $y_i$ in $w$. If there is a copy of $y_i$ in $\Tilde{\phi}(w)$ not present in $\widehat{\Tilde{\phi}(w)}$, then $\widetilde{\inv{\phi}}(\widehat{\Tilde{\phi}(w)})$ contains fewer copies of $y_i$ than $w$ does, so $\inv{\phi}(\phi(w))\neq w$.

\end{proof}

Our second lemma states that, for $n\geq 2$, no separable word can have subwords $r_i$ and $r_j$ if $i\neq j$ (or if $|j-1|>1$ if $n=2$).  It is used to bound the lengths of boundary arcs of diagrams over $\langle X_n\mid \Csep\rangle$ with boundary $w(a, k, s, n)$.  The bound is used in the proofs of \Cref{prop: seplength-w} and \Cref{prop: dsepab} which, respectively, compute $\seplength{w(a, k, s, n)}$ and $\dsep{w(a, k, s, n)}{w(b, k, s, n)}.$

\begin{lemma}\label{lemma: only one ri}
    Let $i, j$ satisfy $1<i<j$ and let $n>1$.  
    \begin{enumerate}
        \item For $n\geq 3$ and $i\neq j$, no separable $w\in\free_n$ can have both $r_i$ and $r_j$ as disjoint subwords;
        \item For $n=2$ and $j>i+1$, no separable $w\in\free_n$ can have both $r_i$ and $r_j$ as disjoint subwords.
    \end{enumerate}
\end{lemma}

\begin{proof}
\noindent For all $n$, we will show that, even though the Whitehead profile of $w$ may contain a cut vertex, the Whitehead profile of the image of $w$ after $i$ steps of Whitehead's algorithm does not. Therefore, $w$ is inseparable. We divide the proof into three cases: $n>3, n=3,$ and $n=2$.
\begin{enumerate}[wide, labelindent=0pt]
\item[(1i)] Let $n>3$ and assume $w\in\free_n$ has subwords $r_i$ and $r_j$ where $j>i$.  A subgraph of the Whitehead profile of $w$ can be seen in \Cref{fig:Whitehead-profile-rin}. Observe that the only cut vertex of this graph is $x_{n-1}.$ The partition of $X_n^*$ associated with this Whitehead profile is $\{x_{n-1}, \inv{x_n}\}\bigsqcup\{x_1, \inv{x_1}, x_2, \inv{x_2}, \ldots, x_{n-2}, \inv{x_{n-2}}, \inv{x_{n-1}}, x_n\}$, and the Whitehead automorphism associated with this partition is the $\phi$ defined above \Cref{lemma:phi-only-cancels-xn-1}

\begin{figure}
    \centering
    \includegraphics[scale=.45]{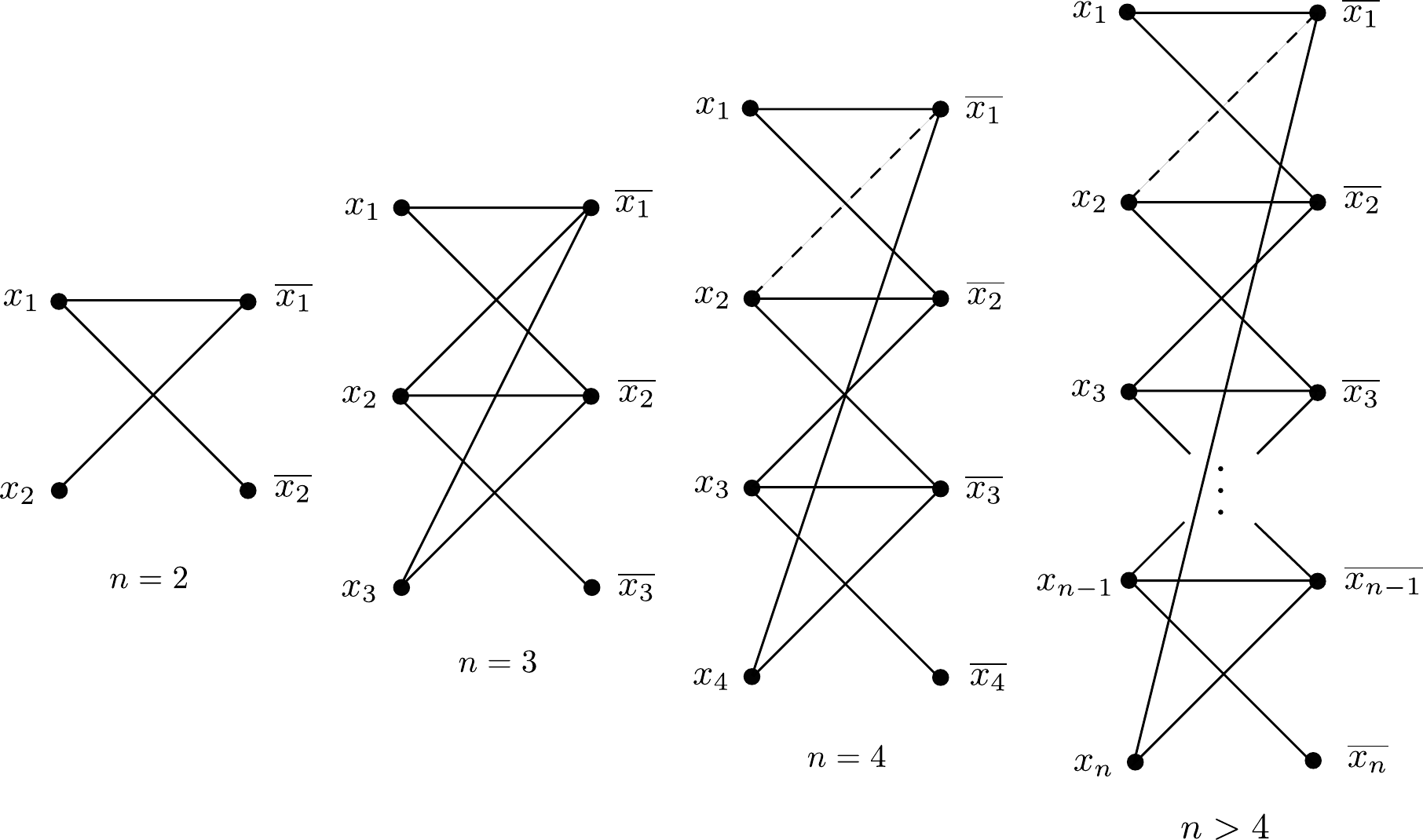}
    \caption{The Whitehead profiles of $r_i(n)$ for $n\geq 2$.}
    \label{fig:Whitehead-profile-rin}
\end{figure}

We have:
\[
\phi(r_i)=x_{n-1}(x_1^2x_2^2\ldots x_{n-2}^2x_{n-1}^{i-1}x_n)^2x_{n-1}^2x_{n-2}^2\ldots x_2^2\inv{x_{n-1}}. 
\]

\noindent Since $i>1$, the Whitehead profile of $\phi(r_i)$ contains the Whitehead profile in \Cref{fig:Whitehead-profile-rin}. A calculation shows that 
\[
    \phi^m(r_i) = x_{n-1}^m(x_1^2x_2^2\ldots x_{n-2}^2x_{n-1}^{i-m}x_n)^2x_{n-1}^2x_{n-2}^2\ldots x_2^2x_{n-1}^{-m}.
\] The first step of Whitehead's algorithm applies the automorphism $\phi$ to $w$. Note that, since $x_{n-1}$ is the only possible cut vertex of $\phi(w)$, Whitehead's algorithm repeatedly applies the automorphism $\phi$ to $w$. If for some $m<i$, the Whitehead profile of $\phi^m(w)$ has no cut vertices, $w$ is inseparable.

Now consider $\phi^i(w)$. Let \[u_m(i):=(x_1^2x_2^2\ldots x_{n-2}^2x_{n-1}^{i-m}x_n)^2x_{n-1}^2x_{n-2}^2\ldots x_2^2.\]  If $\phi^i(w)$ contains $u_i(i)$ and $u_j(i)$ as subwords, $w$ is inseparable, as the Whitehead graph of $u_i(i)$ contains the Whitehead profile in \Cref{fig:Whitehead-profile-rin-i-steps}, the Whitehead graph of $u_j(i)$ contains the Whitehead profile in \Cref{fig:Whitehead-profile-rin}, and the union of these Whitehead profiles is connected with no cut vertices. Note that because $w$ is reduced, $\inv{x_1}$ does not directly precede $r_i$ or $r_j$ in $w$; also, $\inv{x_2}$ does not directly follow $r_i$ or $r_j$ in $w$. By \Cref{lemma:phi-only-cancels-xn-1}, no letters are canceled in $\phi^i(w)$ but $x_{n-1}^{\pm 1}$, so since $n-1>2$, $\phi^i(w)$ contains $u_i(i)$ and $u_j(i)$ as disjoint subwords.

\medskip

\begin{figure}
    \centering
    \includegraphics[scale=.45]{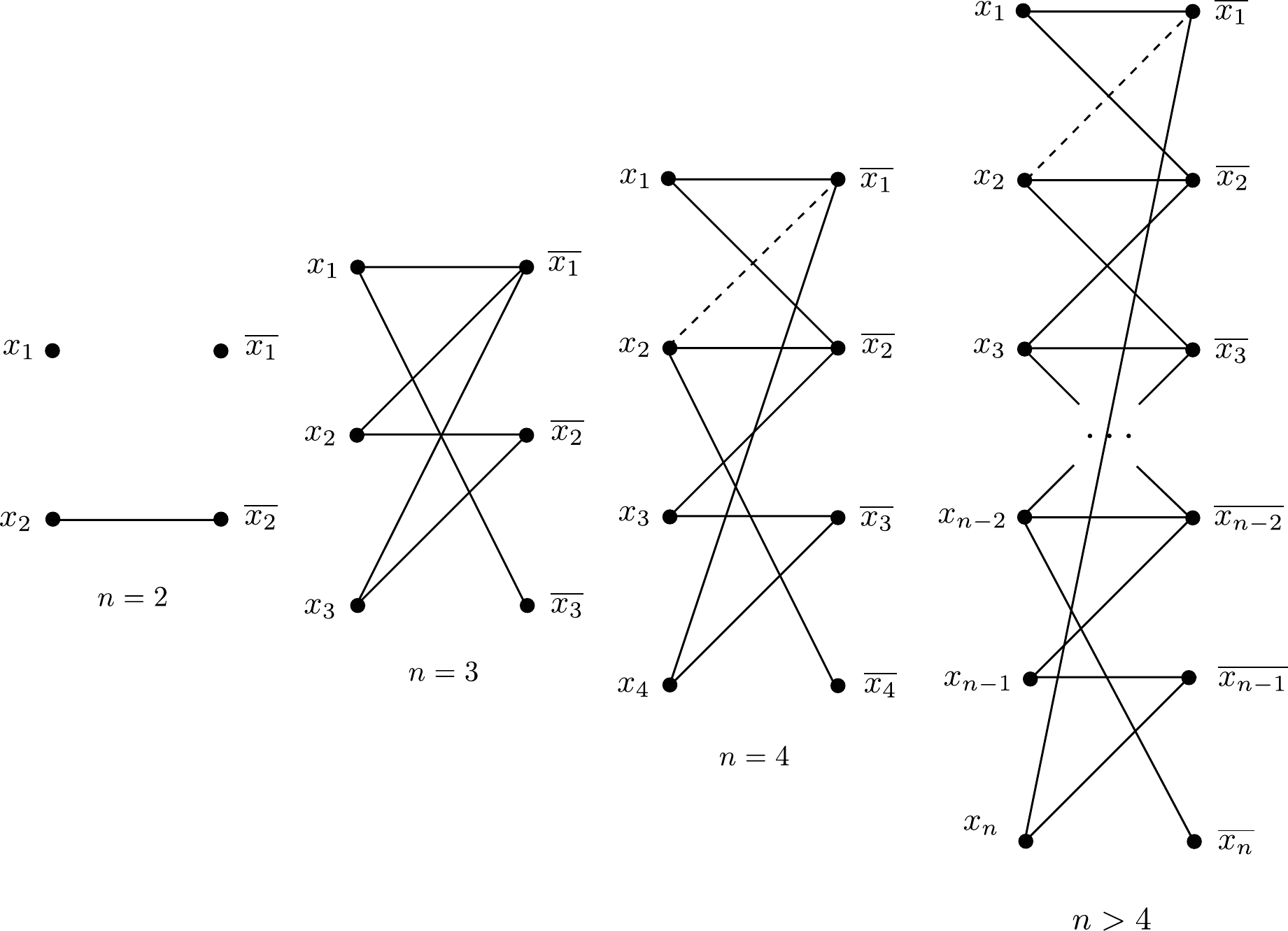}
    \caption{The Whitehead profiles of $u_i(i)$ for $n\geq2.$}
    \label{fig:Whitehead-profile-rin-i-steps}
\end{figure}

\item[(1ii)] Now assume $n=3$.  We have
    \[
        \phi(r_i) =\big(x_2(x_1^2x_2^{i-1}x_3)^2x_2\big)^2.
    \]
    A calculation shows that, for $m\leq i$, 
    \[
        \phi^m(r_i)=\big(x_2^m(x_1^2x_2^{i-m}x_3)^2x_2^{2-m}\big)^2.
    \]
    Similar to the proof of (1i), we want to show that if $r_i$ and $r_j$ are disjoint subwords of $w$, then $\phi^m(w)$ contains disjoint subwords $u_i(m), u_j(m)$ satisfying the property that the union of the Whitehead profiles of $u_i(i)$ and $u_j(i)$ is connected with no cut vertices. Let \[u_m(i)=(x_1^2x_2^{i-m}x_3^2)^2x_2^2(x_1^2x_2^{i-m}x_3)^2.\] The Whitehead profiles of $u_i(i)$ and $u_j(i)$ are shown in \Cref{fig:Whitehead-profile-rin-i-steps} and \Cref{fig:Whitehead-profile-rin}, respectively. Note that their union is connected without cut vertices.  By \Cref{lemma:phi-only-cancels-xn-1}, the only letters that can be cancelled under powers of $\phi$ are $x_{n-1}$ and its inverse, so it suffices to show that the copy of $x_2^2$ is preserved by $\phi$.  Note that \[\phi(x_3x_2^2x_1)=x_3\inv{x_2}x_2^2x_2x_1\inv{x_2},\] and the conclusion follows.

    \medskip

    \item[(2)] For $n=2$, we have:
    \[
        \phi(r_i)=x_1(x_1^{i-1}x_2)^2\inv{x_1}.
    \]
    A calculation shows that, for $m\leq i,$
    \[
        \phi^m(r_i) = x_1^m(x_1^{i-m}x_2)^2x_1^{-m}.
    \]
    Again, we want to show that $\phi^i(w)$ contains $u_i(i)$ and $u_j(i)$, where \[u_m(i) = x_2x_1^{i-m}x_2.\]  The Whitehead profiles of $u_i(i)$ and $u_j(i)$ are visible in \Cref{fig:Whitehead-profile-rin-i-steps} and \Cref{fig:Whitehead-profile-rin}, respectively. Again, their union is connected without cut vertices.  \Cref{lemma:phi-only-cancels-xn-1} implies that the only letters cancelled by $\phi$ are $x_1$ and its inverse. Thus $u_j(i)$ and $u_i(i)$ are both subwords of $\phi^i(w)$, and the conclusion follows. \qedhere
\end{enumerate}
\end{proof}

\subsection{Step 2: Separable length of $\mathbf{w(a, k, s, n)}$} 
Recall that for $n>3,$
\[
    r_i(n):= (x_1^2x_2^2\ldots x_{n-2}^2x_{n-1}^ix_n)^2x_{n-1}^2x_{n-2}^2\ldots x_2^2;
\]
for $n=3$, 
\[  
    r_i(n):= \big((x_1^2x_2^ix_3)^2x_2^2\big)^2;
\]
for $n=2,$ 
\[  
    r_i(n):=(x_1^ix_2)^2.
\]
For all $n\geq 3$, we let
\begin{equation}\label{equation:w-for-large-n}
    w(a, k, s, n):=\displaystyle\prod_{i=a}^{a+s-1}r_i^k;
\end{equation}
for $n=2$, let 
\begin{equation}\label{equation:w-for-n-2}
    w(a, k, s, n):=\displaystyle\prod_{i=0}^{s-1} r_{a+2i}^k.
\end{equation}
\noindent Recall that a \emph{boundary arc} of $D$ is a connected component of the intersection of the boundary of a region of $D$ with the boundary of $D$. Sometimes we will refer to a ``boundary arc of $R$", where $R$ is a region of $D$: we mean a boundary arc of $D$ contained in $\partial R$. For convenience, we will often refer to a subword of the label of a boundary arc as a subword of a boundary arc, conflating a boundary arc with its label. 

\begin{restatable}{prop}{seplengthw}
\label{prop: seplength-w}
    Let $n\geq 2$. For sufficiently large $a=a(s)$ and $k=k(a, n, s)$, the separable length of $w(a, k, s, n)$ is $s$.
\end{restatable}

The proof of this proposition is highly computational. For the sake of brevity, we will prove the proposition here only for $n>3$.  The other two cases ($n=2$ and $n=3$) are very similar and can be found in the Appendix (\Cref{sec: appendix}).

\begin{proof}
    Let $n>3$. We have $\seplength{w(a, k, 1, n)}=1$ as $w(a, k, 1, n)$ is separable for all $a, k$ and $n$. By \Cref{lemma: only one ri}, $w(a, k, 2, n)$ is inseparable for all $a, k$ and $n$, so $\seplength{w(a, k, 2, n)}\geq 2$. Thus, $\seplength{w(a, k, 2, n)}=2$ as $w(a, k, s, n)=r_a^kr_{a+1}^k$, and both $r_a^k$ and $r_{a+1}^k$ are separable. Note that it is always the case that $\seplength{w(a, k, s, n)}\leq s$ as each $w(a, k, s, n)$ is a product of the $s$ separable words $r_a^k, r_{a+1}^k, \ldots r_{a+s-1}^k$. It remains to show that there exist integers $a$ and $k$ such that $\seplength{w(A, K, s, n)}=s$ for all $A\geq a$ and $K\geq k$. 

    Suppose toward a contradiction that $\seplength{w(a, k, s, n)}\leq s-1.$  Then there is a Van Kampen diagram $D$ for $w(a, k, s, n)$ over $\langle X_n\mid \Csep\rangle$ of area at most $s-1$. By \Cref{prop:arcs of disk diagrams}, a Van Kampen diagram of area at most $s-1$ has at most $2(s-2)$ boundary arcs. By \Cref{lemma: only one ri}, each region of $D$ has at most one of the elements of $\{r_i\}_{i=a}^{a+s-1}$ as a subword of its boundary relation.  Since $D$ is a Van Kampen diagram for a positive word, the sum of the lengths of boundary arcs of $D$ is equal to $|w(a, k, s, n)|$ (see \Cref{lemma: positive-words-wedge-disks}). The contradiction we will obtain is that the sum of the lengths of all boundary arcs of $D$ is less than $|w(a, k, s, n)|$ for large enough values of $a$ and $k$. 
    
    We will complete the proof of the proposition assuming the following two claims, the first of which gives an upper bound for the length of any boundary arc of $D$. The second claim gives an upper bound for both the length of subwords of $w(a, k, s, n)$ containing no $r_i$ as a subword and the sum of the lengths of boundary arcs of $D$. We then prove the claims.

    \begin{restatable}{clm}{longestBoundaryArc}\label{clm:longest-bdry-arc}
         Let $n\geq 2$, let $D$ be either a Van Kampen diagram or an annular diagram over $\langle X_n\mid \Csep\rangle$, and suppose $D$ has a boundary component $B$ labeled by $w(a, k, s, n)$. For sufficiently large $k$ and $s\geq 2$, there is a degree-two polynomial $M_n(a, k, s)$ such that the length of any boundary arc of $D$ along $B$ has length less than $M_n(a, k, s).$  
   \end{restatable}

    \begin{restatable}{clm}{upperBoundD}\label{clm:upper-bound-on-boundary-D}
         Let $n\geq 2$. Let $D$ be a Van Kampen or annular diagram over the presentation $\langle X_n\mid \Csep\rangle$ with a boundary component $B$ labeled by $w(a, k, s, n)$. The following hold for sufficiently large $k$ and $s\geq 2$.
         \begin{enumerate}
             \item There is a linear function $m_n(a, s)$ such that every boundary arc along $B$ containing no $r_i$ as a subword has length less than $m_n(a, s)$. 
             \item Suppose $D$ is a Van Kampen diagram for $w(a, k, s, n)$ with at most $s-1$ regions.  Then, for sufficiently large values of $a$ and $k$, the sum of the lengths of the boundary arcs of $D$ is at most $(s-1)M_n(a, k, s)+(s-3)m_n(a, s).$
         \end{enumerate}
    \end{restatable}

\noindent We now begin the proof of the proposition assuming the claims. By \Cref{clm:upper-bound-on-boundary-D}, the sum of the lengths of the boundary arcs of $D$ is less than $(s-1)M_n(a, k, s)+(s-3)m_n(a, s).$ The proof splits into three cases: $n>3$, $n=3$, and $n=2$. The cases $n=3$ and $n=2$ can be found in \Cref{sec: appendix}.
    
    \noindent\textbf{Case 1: $\mathbf{n>3}$.} A calculation shows
    \begin{equation*}
        \begin{split}
            |w(a, k, s, n)|&=\displaystyle\sum_{i=a}^{a+s-1}\big(k|r_i|\big)\\
            &=k(6ns-11s+2as+s^2).
        \end{split}
    \end{equation*}
    By the definitions of $M_n(a, k, s)$ and $m_n(a, s)$ in the proofs of \Cref{clm:longest-bdry-arc} and \Cref{clm:upper-bound-on-boundary-D}, we have
    \begin{equation*}
            (s-1)M_n(a, k, s)+(s-3)m_n(a, s)=(s-1)(k+2)(6n+2a+2s)
            +(s-3)(12n+4a+4s).
    \end{equation*}
    So, $(s-1)M_n(a, k, s)+(s-3)m_n(a, s)<|w(a,k,s, n)|$ if and only if
    \begin{inequality}\label{inequality:contradiction-inequality}
            (s-1)(k+2)(6n+2a+2s)+(s-3)(12n+4a+4s)
            <k(6ns-11s+2as+s^2)
    \end{inequality}
    which happens if and only if 
    \begin{inequality}\label{inequality:contradiction-subtract-kMn}
        \begin{split}
            2(s-1)&(6n+2a+2s)+(s-3)(12n+4a+4s)\\
        &<k\Big[6ns-11s+2as+s^2-(s-1)(6n+2a+2s)\Big] 
        \end{split}
    \end{inequality}    
    The right side of \Cref{inequality:contradiction-subtract-kMn} can be simplified to 
    \begin{equation*}
        k(6n+2a-9s-s^2).
    \end{equation*}  
    For $a$ satisfying $a>s^2+9s$, $6n+2a-9s-s^2$ is positive. To satisfy \Cref{inequality:contradiction-inequality}, take $k$ such that
    \begin{inequality}\label{inequality:k-for-large-n}
        k>\frac{(4s-8)(6n+2a+2s)}{6n+2a-9s-s^2}.
    \end{inequality}
    Thus, choosing $k$ so that \Cref{inequality:k-for-large-n} holds, we arrive at the desired contradiction. This completes the proof of \Cref{prop: seplength-w} for $n>3$ assuming \Cref{clm:longest-bdry-arc} and \Cref{clm:upper-bound-on-boundary-D}. \end{proof}
    
    We now prove the claims for $n>3$. The remaining cases of the claims can be found in \Cref{sec: appendix}.

   \longestBoundaryArc*
   
    \begin{proof}[Proof of \Cref{clm:longest-bdry-arc}]
    \noindent\textbf{Case 1: $\mathbf{n>3}$}. Recall that 
    \[
        r_i(n)=(x_1^2\ldots x_{n-2}^2x_{n-1}^ix_n)^2x_{n-1}^2\ldots x_2^2
    \]
    and 
    \[
        w(a, k, s, n) = \displaystyle\prod_{i=a}^{a+s-1} r_i(n)^k.
    \]
    By \Cref{lemma: only one ri}, the length of each boundary arc of $B$ is bounded above by the length of the longest subword of $w(a, k, s, n)$ containing only one $r_i$ as a subword, which is less than 
    \[\displaystyle\max_{a\leq i\leq a+s-3}\{|r_ir_{i+1}^kr_{i+2}|, |r_{a+s-1}r_{a}^kr_{a+1}|, |r_{a+s-2}r_{a+s-1}^kr_a|\}.\]  Since
    \begin{align*}
        |r_i| &= |(x_1^2\ldots x_{n-2}^2x_{n-1}^ix_n)^2x_{n-1}^2\ldots x_2^2|\\
        &= 6n+2i-10,
    \end{align*} the word $r_ir_{i+1}^kr_{i+2}$ is shorter than $r_jr_{j+1}^kr_{j+2}$ when $i<j$.  Therefore, the longest boundary arc of $B$ has length less than \[\displaystyle\max\{|r_{a+s-2}r_{a+s-1}^kr_a|, |r_{a+s-3}r_{a+s-2}^kr_{a+s-1}|, |r_{a+s-1}r_a^kr_{a+1}|\}.\]  A calculation shows that the lengths of the three possible maxima are as follows:
    \begin{align*}
        |r_{a+s-3}r_{a+s-2}^kr_{a+s-1}| &=|r_{a+s-3}| +k|r_{a+s-2}|+|r_{a+s-1}|\\
        &=  12n+4a+4s-28+k(6n+2a+2s-14);
    \end{align*}
    \begin{align*}
        |r_{a+s-2}r_{a+s-1}^kr_a|&=|r_{a+s-2}| +k|r_{a+s-1}|+|r_a|\\
        &=12n+4a+2s-24+k(6n+2a+2s-12);
    \end{align*}
    \begin{align*}
        |r_{a+s-1}r_a^kr_{a+1}| &=|r_{a+s-1}| +k|r_a|+|r_{a+1}|\\
        &=12n+4a+2s-20+k(6n+2a-10).
    \end{align*}
    We can assume that $s>2$ by the first paragraph of the proof of this proposition. Now take $k$ such that $k>s-2$, and recall $s-2\geq1$. This guarantees that  $|r_{a+s-2}r_{a+s-1}^kr_a|$ is longest. To see this, consider our three possible maxima:
    \begin{enumerate}
        \item $|r_{a+s-3}r_{a+s-2}^kr_{a+s-1}|=12n+4a+4s-28+k(6n+2a+2s-14);$
        \item $|r_{a+s-2}r_{a+s-1}^kr_a|=12n+4a+2s-24+k(6n+2a+2s-12)$;
        \item $|r_{a+s-1}r_a^kr_{a+1}|=12n+4a+2s-20+k(6n+2a-10).$
    \end{enumerate}
    Subtracting the third expression from each of the possible maxima, we obtain:
    \begin{enumerate}[label=(\alph*)]
        \item $2s-8+k(2s-4);$
        \item $-4+k(2s-2);$
        \item $0$.
    \end{enumerate}
    Since $s>2$ and $k>s-2\geq 1$, both $(a)$ and $(b)$ above are positive. We also have
    \begin{align*}
        2s-8+k(2s-4)&<-4+k(2s-2) \\
        \iff 2s-8+2ks-4k&<-4+2ks-2k\\
        \iff 2s-4&<2k\\
        \iff s-2&<k.
    \end{align*}
    Thus, for $s>2$, $k>s-2$, and $n>3$, the longest boundary arc of $B$ has length less than \[12n+4a+2s-24+k(6n+2a+2s-12).\] Note \[
    12n+4a+2s-24+k(6n+2a+2s-12)<(k+2)(6n+2a+2s)=:M_n(a, k, s).\]
\end{proof}

    \upperBoundD*

    \begin{proof}[Proof of \Cref{clm:upper-bound-on-boundary-D}]
    Recall that, for $n>3,$
    \begin{equation*}
        r_i(n)=(x_1^2x_2^2\ldots x_{n-2}^2x_{n-1}^ix_n)^2x_{n-1}^2x_{n-2}^2\ldots x_2^2.
    \end{equation*}
    and 
    \begin{equation*}
        w(a, k, s, n)=\displaystyle\prod_{i=a}^{a+s-1} \big((x_1^2\ldots x_{n-2}^2 x_{n-1}^ix_n)^2x_{n-1}^2\ldots x_2^2\big)^k.
    \end{equation*}
   
    Observe that, for all $n>3$, the only instance of $r_i$ as a cyclic subword of $w(a, k, s, n)$ is in the term $r_i^k$. Thus if a region $R$ of $D$ has a boundary arc containing $r_i$ as a subword, $R$ intersects this boundary within the section labeled $r_i^k$. By \Cref{lemma: only one ri}, each region of $D$ contains at most one $r_i$ as a subword of its boundary relation. \Cref{clm:longest-bdry-arc} implies that every boundary arc of $D$ containing any $r_i$ as a subword has length less than $M_n(a, k, s)$. Hence if $R$ is a region of $D$ and $i$ is an integer, the total length of the boundary arcs of $R$ which contain any subword $r_i$ is less than $M_n(a, k, s)$. By assumption, $D$ has at most $s-1$ regions, so the total length of the boundary arcs of $D$ containing an $r_i$ as a subword is at most $(s-1)M_n(a, k, s).$  The remaining boundary arcs of $D$ do not contain a subword $r_i$ for any $i$. For $n>3,$ by inspection, the longest subwords of $w(a, k, s, n)$ containing no $r_i$ are of length less than
    \begin{equation*}
        |r_{a+s-1}(n)^2|=12n+4a+4s-24<12n+4a+4s=:m_n(a, s).
    \end{equation*} 
    Since $D$ has at most $2(s-2)$ boundary arcs, assuming each of the $s-1$ regions of $D$ contains at least one boundary arc with an $r_i$ subword, there are at most $(s-3)$ boundary arcs of $D$ with no $r_i$ subword. Thus, the sum of the lengths of boundary arcs of $D$ is less than $(s-1)M_n(a, k, s)+(s-3)m_n(a, s)$. If some region of $D$ contains no boundary arcs with a subword $r_i$, each of the $2(s-2)$ boundary arcs is bounded above by $m_n(a, s)<M_n(a, k, s)$, so we get a larger upper bound for the length of the boundary of $D$ by assuming each region has a boundary arc with an $r_i$ subword. This completes the proof of \Cref{clm:upper-bound-on-boundary-D} for $n>3$. See \Cref{sec: appendix} for the remaining cases.
    \end{proof}  
\noindent With the proofs of both \Cref{clm:longest-bdry-arc} and \Cref{clm:upper-bound-on-boundary-D} complete, we have completed the proof of \Cref{prop: seplength-w}. 

\subsection{Step 3: Computing $\mathbf{\dsep{[w(a, k, s, n)]}{[w(b, k, s, n)]}}$.} For $n>2$, let $b=a+s$, and for $n=2$, let $b=a+2s.$ We have shown so far that if $a$ and $k$ are large enough, the words $w(a, k, s, n)$ and $w(b, k, s, n)$ have separable length $s$. Now we want to show that we can choose $a$ and $k$ large enough so that $\dsep{[w(a, k, s, n)]}{[w(b, k, s, n)]}=2s$. The following lemma gives a way to visualize short paths between nontrivial vertices in $\Wh(\Gamma),$ similar to \Cref{lemma:paths=diagrams}.

Our proof that there are $a$ and $k$ large enough so that 
\begin{equation*}
    \dsep{[w(A, K, s, n)]}{[w(b, K, s, n)]}=2s
\end{equation*} for all $A>a$ and $K>k$ is similar to the proof that there are $a$ and $k$ large enough such that $\seplength{w(A, K, s, n)}=s$ for $A>a$ and $K>k$. A path $p\in\pathwh(\rose_n)$ between $[w(a, k, s, n)]$ and $[w(b, k, s, n)]$ of length less than $2s$ would be represented by an annular diagram of area less than $2s$. We first find an upper bound on the number of boundary arcs of annular diagrams of area less than $2s$.  Then we use \Cref{lemma: only one ri} to show that if $a$ and $k$ are sufficiently large and $A$ is an annular diagram with boundaries $w(a, k, s, n)$ and $w(b, k, s, n)$, and the area of $A$ is less than $2s$, then the length of the boundary of $A$ is strictly less than $|w(a, k, s, n)|+|w(b, k, s, n)|-O,$ where $O$ is the total length of the overlapping components of the boundaries (see \Cref{def:overlapping}). This is a contradiction as $w(a, k, s, n)$ and $w(b, k, s, n)$ are both positive, so every edge on the boundary of $A$ is either contained within the boundary of a region or is shared by both boundaries.

\begin{definition}\label{def:overlapping}
    Let $A$ be an annular diagram with no valence-$1$ vertices and boundaries labeled by $w_1$ and $w_2$. The \emph{overlapping components} of the boundary of $A$ are connected components of the intersection of the two boundaries of $A$. Such components are necessarily labeled by cyclic subwords of both $w_1$ and $w_2$. See \Cref{fig:overlapping-components}.

    \begin{figure}
        \centering
        \includegraphics[scale=.5]{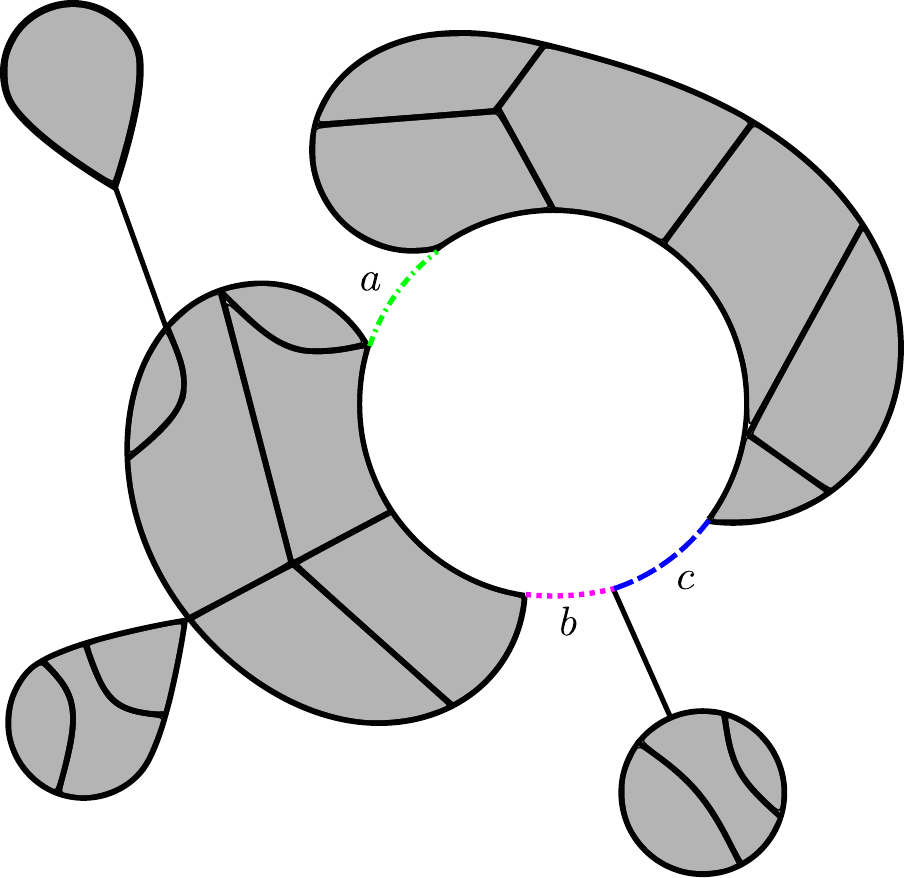}
        \caption{The colored dashed lines are the overlapping components of the annular diagram pictured. Note that $b$ and $c$ are distinct overlapping components: $bc$ is a cyclic subword of the word along the inner boundary of this diagram, but it is not a cyclic subword of the word along the outer boundary.}
        \label{fig:overlapping-components}
    \end{figure}
\end{definition}

\begin{restatable}{prop}{dsepab}
\label{prop: dsepab}
    For $n>2$, let $b=a+s;$ for $n=2$, let $b=a+2s.$  Then, for large enough integers $a$ and $k$, $\dsep{[w(a, k, s, n)]}{[w(b, k, s, n)]}=2s.$
\end{restatable}

The proof of this proposition is computational.  For brevity, we prove the proposition here for $n>3$, and we prove the remaining cases in the Appendix (\Cref{sec: appendix}).

\begin{proof}
    By \Cref{prop: seplength-w}, for sufficiently large $a$ and $k,$ both $w(a, k, s, n)$ and $w(b, k, s, n)$ have separable length $s,$ so there is a path of length $2s$ from $w(a, k, s, n)$ to $w(b, k, s, n)$ through the trivial vertex. Suppose $a$ and $k$ are large enough to satisfy the conclusion of \Cref{prop: seplength-w} and that $\dsep{[w(a, k, s, n)]}{[w(b, k, s, n)]}<2s$. Let $p$ be a path of minimal length from $[w(a, k, s, n)]$ to $[w(b, k, s, n)]$. By \Cref{lemma:existence-annular}, $p$ is represented by an annular diagram $A$ over $\langle X_n\mid \Csep\rangle$ with boundaries $C$ labeled by $w(a, k, s, n)$ and $B$ labeled by $w(b, k, s, n)$. The area of $A$ is no more than $2s-1$. By \Cref{lemma: boundary arcs of annular diagram}, such a diagram has at most $2(2s-1)$ boundary arcs. The number of overlapping components of $A$ is bounded above by the length of $p$: $p$ is a path of minimal length, so the area of $A$ is the length of $p$. Thus there are at most $2s-1$ regions of $A$ and at most $2s-1$ overlapping components connecting them. The overlapping components of $A$ are labeled by subwords shared by $w(a, k, s, n)$ and $w(b, k, s, n).$ Let $O$ be the sum of the lengths of the overlapping components of $A$. Since $w(a, k, s, n)$ and $w(b, k, s, n)$ are positive, \Cref{lemma:existence-annular} implies that the sum of the lengths of the edges on the boundary of $A$ is $|w(a, k, s, n)|+|w(b, k, s, n)|-O.$ We show that $a$ or $k$ must be small by finding an upper and lower bound for the sum of the lengths of all boundary arcs and overlapping components of $A$ and observing that the lower bound is larger than the upper bound for sufficiently large $a$ and $k$. 

    \medskip
    
    \noindent\textbf{Case 1: $\mathbf{n>3}$.} By \Cref{lemma: only one ri}, no region of $A$ can have boundary label containing both $r_i$ and $r_j$ as subwords if $i\neq j$.  By \Cref{clm:longest-bdry-arc}, the longest subword of $w(b, k, s, n)$ containing only one of the $r_i$ as a subword is of length less than $M_n(b, k, s)=(k+2)(6n+2b+2s)$. Subwords of $w(a, k, s, n)$ containing only one of the $r_i$ as a subword are strictly shorter. By \Cref{clm:upper-bound-on-boundary-D}, the longest subword of $w(b, k, s, n)$ containing no $r_i$ subword is of length strictly less than $m_n(b, s)=12n+4b+4s$. Subwords of $w(a, k, s, n)$ containing no $r_i$ subwords are strictly shorter. By the argument in \Cref{clm:upper-bound-on-boundary-D}, supposing $A$ contains two or more boundary arcs with the subword $r_j$ for some particular $j$, the sum of the lengths of these boundary arcs is less than $M_n(b, k, s)$. Since each region can have only one of the $r_i$ as a subword of its boundary label, the sum of the lengths of boundary arcs of $A$ containing any $r_i$ as a subword is less than $(2s-1)M_n(b, k, s).$ Assuming each region of $A$ contains a boundary arc with an $r_i$ subword, there are at most $2s-1$ remaining boundary arcs of $A$ with no $r_i$ as a subword, and these arcs have length less than $m_n(b, s)$ by the first part of \Cref{clm:upper-bound-on-boundary-D}. If some region has no boundary arc containing an $r_i$ subword, all of its boundary arcs have length less than $m_n(b, s)<M_n(b, k, s)$. Thus we get a higher upper bound on the total length of boundary arcs of $A$ when we assume each region contains a boundary arc with a subword $r_i$ for some $i$. Thus the sum of the lengths of the boundary arcs of $A$ is at most $(2s-1)M_n(b, k, s)+(2s-1)m_n(b, s).$ Let $O_n(a, s)$ be the length of the longest subword shared by $w(a, k, s, n)$ and $w(b, k, s, n)$.  Since there are at most $2s-1$ overlapping components of $A$, the length of the boundary of $A$ is less than $(2s-1)M_n(b, k, s)+(2s-1)m_n(b, s) + (2s-1)O_n(a, s).$
    
    We now compute a lower bound on the length of the boundary of $A$. Consider the total length of the boundary of an annular diagram of area less than $2s$ with boundaries labeled by $w(a, k, s, n)$ and $w(b, k, s, n)$. Since $w(a, k, s, n)$ and $w(b, k, s, n)$ are both positive, every edge of the boundary of $A$ is either contained in the boundary of a region of $A$ or is a subword of both $w(a, k, s, n)$ and $w(b, k, s, n).$ Thus the length of the boundary of $A$ is at least $|w(a, k, s, n)|+|w(b, k, s, n| - (2s-1)O_n(a, s)$.
    
    We want to compute an upper bound for $O_n(a, s)$. Consider the longest subword shared by 
    \[w(a, k, s, n)=\displaystyle\prod_{i=a}^{a+s-1}\big((x_1^2x_2^2\ldots x_{n-2}^2x_{n-1}^ix_n)^2x_{n-1}^2x_{n-2}^2\ldots x_2^2\big)^k\]
    and 
    \[w(b, k, s, n)=\displaystyle\prod_{i=a+s}^{a+2s-1}\big((x_1^2x_2^2\ldots x_{n-2}^2x_{n-1}^ix_n)^2x_{n-1}^2x_{n-2}^2\ldots x_2^2\big)^k\]
    By inspection, \[O_n(a, s)=|x_{n-1}^{a+s-1}x_nx_{n-1}^2x_{n-2}^2\ldots x_2^2x_1^2x_2^2\ldots x_{n-2}^2x_{n-1}^{a+s-1}|=2a+2s+4n-9.\]
      Thus, to show that there are $a$ and $k$ such that 
      \begin{equation*}
          \dsep{[w(a, k, s, n)]}{[w(b, k, s, n)]}=2s,
      \end{equation*} it suffices to show that there are $a$ and $k$ such that
    \begin{inequality}\label{inequality:contradiction-annular-large-n}
        |w(a, k, s, n)|+|w(b, k, s, n)|-2(2s-1)O_n(a, s)>(2s-1)M_n(b, k, s)+(2s-1)m_n(b, s).
    \end{inequality}

    Recall that
    \begin{equation*}
        |w(a, k, s, n)|=k(6ns-11s+2as+s^2).
    \end{equation*} A computation shows that
    \begin{equation*}
        |w(a, k, s, n)|+|w(b, k, s, n)| = k(12ns-22s+4as+4s^2).
    \end{equation*} Therefore, \Cref{inequality:contradiction-annular-large-n} is equivalent to 
    \begin{inequality}\label{inequality:constants-contradiction-annular-large-n}
        \begin{split}
            k(12ns&-22s+4as+4s^2)-2(2s-1)(2a+2s+4n-9)\\&>(2s-1)(k+2)(6n+2a+4s)+(2s-1)(12n+4a+8s),
        \end{split} 
    \end{inequality} which simplifies to 
    \begin{equation*}
        k(6n+2a-18s-4s^2)>(2s-1)(12a+20s+32n-18).
    \end{equation*}
    The left side of this inequality is positive if $a>9s+2s^2$. Therefore, letting $a>9s+2s^2$ and
    \begin{equation*}
        k>\displaystyle\frac{(2s-1)(12a+20s+32n-18)}{6n+2a-4s^2-18s}
    \end{equation*}
    gives us our contradiction. Thus, for $n>3$ and sufficiently large $a$ and $k$, 
    \[
        \dsep{[w(a, k, s, n)]}{[w(b, k, s, n)]}=2s.
    \] \end{proof} 

\subsection{Step 4: Fat triangles.} We have proven that for every $s>0$ we have geodesic triangles on the vertices $[w(a, k, s, n)], [w(b, k, s, n)],$ and $[1]$ where \[\dsep{[1]}{[w(a, k, s, n)]}=\dsep{[1]}{[w(b, k, s, n)]}=s\] and $\dsep{[w(a, k, s, n)]}{[w(b, k, s, n)]}=2s.$  We want to show that there is a geodesic between $[w(a, k, s, n)]$ and $[w(b, k, s, n)]$ which lies above an open ball of radius $s$ centered at $[1]$. We conclude that these triangles violate the $\delta$-thin condition for hyperbolicity, so $\Wh(\rose_n)$ is nonhyperbolic for all $n\geq 2$.

\nonhyperbolicity*

\begin{proof}
Let $n\geq 2$ and $s>0$.  For $n>2,$ let $b=a+s$, and for $n=2$, let $b=a+2s$. For large enough $a$ and $k$, there is a geodesic $p$ from $w(a, k, s, n)$ to $w(b, k, s, n)$ with vertices $w(c_i, k, t_i, n)$ where all the $c_i\geq a$ and $t_i\geq s$ (see \Cref{fig:fat-triangle}).  By \Cref{prop: seplength-w}, all vertices on the interior of $p$ have separable length strictly greater than $s$. This triangle is pictured in \Cref{fig:fat-triangle} for $n>2$. For $n=2$, see \Cref{fig:fat-triangle-n-2}.

\begin{figure}
    \centering
    \includegraphics[scale=.4]{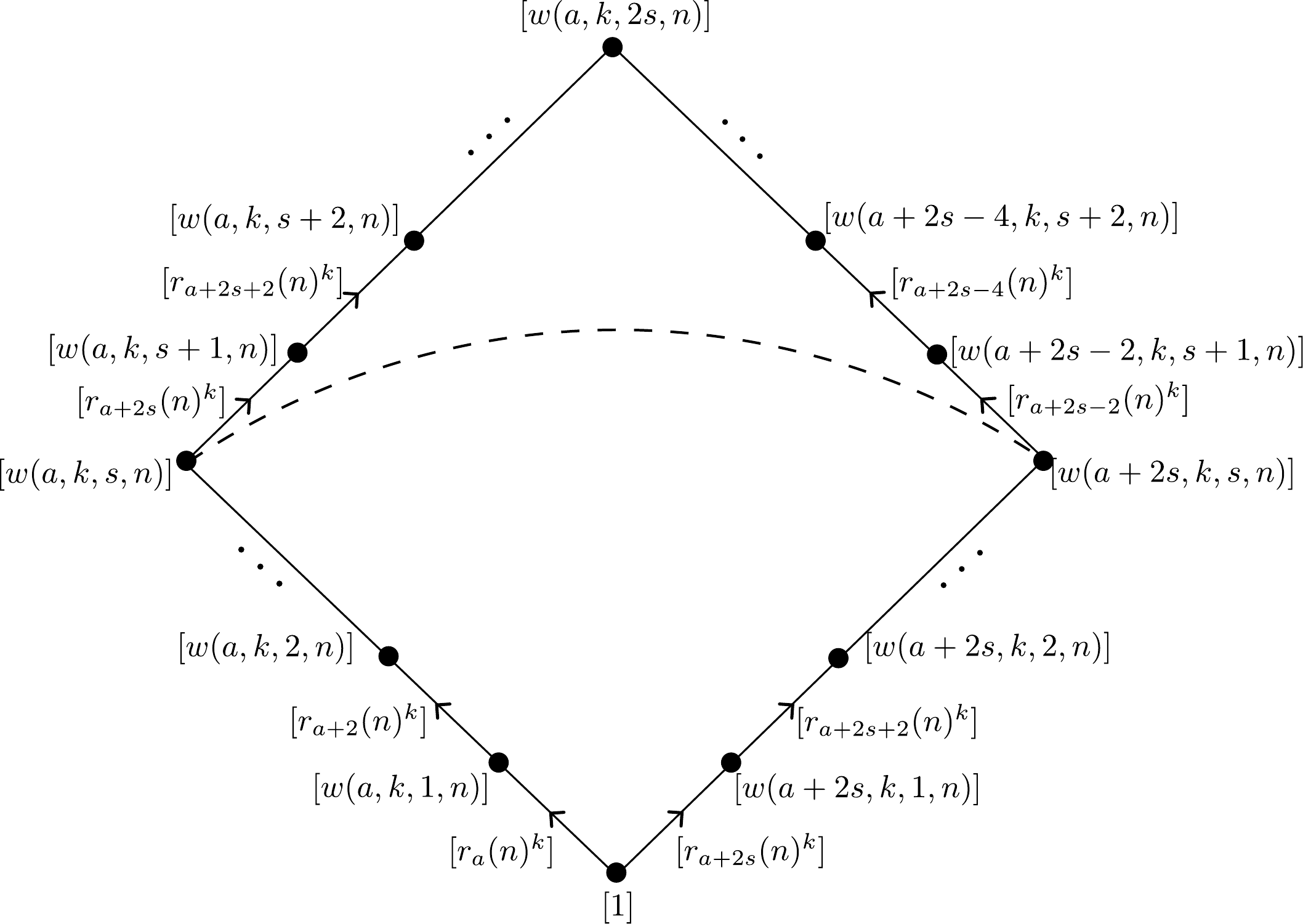}
    \caption{Fat geodesic triangles in $\Wh(\rose_2).$}
    \label{fig:fat-triangle-n-2}
\end{figure}

To show that $\Wh(\rose_n)$ is nonhyperbolic, it suffices to show that there is no constant $\delta>0$ such that the infinite collection of geodesic triangles on vertices \[\{[1], [w(a, k, s, n)], [w(b, k, s, n)]\}_{s=1}^{\infty}\] described above satisfies the $\delta-$thin definition of hyperbolicity.  But for any $\delta<\infty$, there is some $s>\delta$ such that $\dsep{[v]}{[w(a, k, 2s, n)}\geq s$ for any vertex $[v]$ on the geodesic segments $([1], [w(a, k, s, n)])$ and $([1],[w(b, k, s, n)])$ in \Cref{fig:fat-triangle} and \Cref{fig:fat-triangle-n-2}. The conclusion follows.
\end{proof}

The geodesic triangles in \Cref{theorem:nonhyperbolicity} pull back to geodesic triangles in $\Cay(\free_n, \Csep)$, so we obtain the following result.

\begin{figure}
        \centering
        \includegraphics[scale=.4]{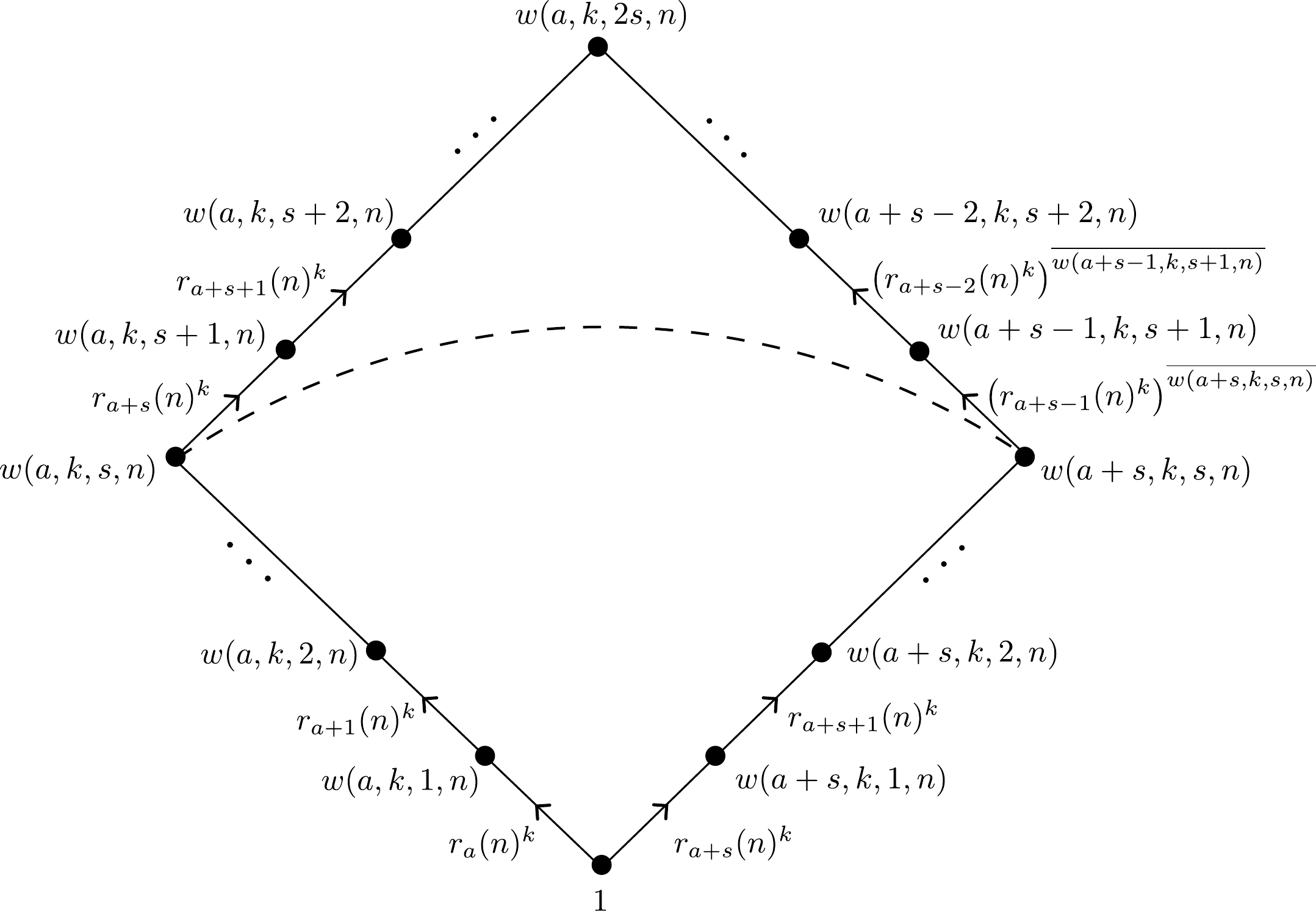}
        \caption{Geodesic triangle in $\Cay(\free_n, \Csep)$ for $n>2$. Here, for $t, v\in\free_n,$ $t^v:=vt\inv{v}.$}
        \label{fig:geodesic-triangle-Cay-Csep}
\end{figure} 

\caynonhyperbolicity*

\begin{proof}
    Recall from the proof of \Cref{cor:Cay-inf-diameter} that $\Wh(\rose_n)$ is a quotient of $\Cay(\free_n, \Csep)$ obtained by first identifying vertices labeled by conjugate elements of $\free_n,$ then identifying edges between vertices which have conjugate labels.  Let $\mathcal{T}$ be the set of the geodesic triangles in $\Wh(\rose_n)$ on the vertices \[\{[1], w(a, k, s, n), w(b, k, s, n)\}\] as described in the proof of \Cref{theorem:nonhyperbolicity}. Let $T\in\mathcal{T}$ and note that the preimage of $T$ in $\Cay(\free_n, \Csep)$ contains a triangle $\widetilde{T}$ (see \Cref{fig:geodesic-triangle-Cay-Csep}). The quotient map from $\Cay(\free_n, \Csep)$ to $\Wh(\rose_n)$ is $1$-Lipschitz, so $\widetilde{T}$ is a geodesic triangle in $\Cay(\free_n, \Csep)$. It follows that the path from $w(a, k, s, n)$ to $w(b, k, s, n)$ crossing through $w(a, k, s, n)w(b, k, s, n)$ lies outside the ball of radius $s$ centered at $1\in\Cay(\free_n, \Csep)$. Thus, there is an infinite family of geodesic triangles in $\Cay(\free_n, \Csep)$ violating the $\delta$-thin condition for nonhyperbolicity.
    
    To prove that $\Cay(\free_n, \curves\prim)$ is nonhyperbolic, it suffices to show that $\Cay(\free_n, \curves\prim)$ is quasi-isometric to $\Cay(\free_n, \Csep)$. To see this, let $w\in\Csep$. There is a basis $\mathcal{B}$ for $\free_n$ such that $w$ can be expressed as a word in at most $(n-1)$ elements of $\mathcal{B}$. Let $\mathcal{B}^*:=\mathcal{B}\cup\inv{\mathcal{B}}$, where $\inv{\mathcal{B}}$ is the set of inverses of elements of $\mathcal{B}$. Let $A$ be the set of elements of $\mathcal{B}^*$ present in $w$. The set $\mathcal{B}^*\setminus A$ contains an element $a$ and its inverse $\inv{a}$.  Write $w$ as a product of two nontrivial words $w_1, w_2\in\free\langle\mathcal{B}\rangle$.  Then $w_1a$ and $\inv{a}w_2$ are both primitive, since they only contain one copy of $a$ or its inverse (see \Cref{remark:one-letter-primitive}). Therefore, every separable word is a product of two primitive elements of $\free_n$, and it follows that $\Cay(\free_n, \Csep)$ is quasi-isometric to $\Cay(\free_n, \curves\prim),$ as desired. \end{proof}   

\section{Homology version of the separability complex}\label{sec: homology version}

\noindent We define a nested sequence of spaces 
\[\W_1(\Gamma) \supset \W_2(\Gamma)\supset \W_3(\Gamma)\supset\ldots\] with intersection $\Wh(\Gamma)$. If $\Gamma$ is a regular cover of $\rose_n$ and the deck group $G_{\Gamma}$ is nilpotent, this sequence will terminate after finitely many steps. 
\begin{definition}
Let $\Gamma$ be a finite cover of $\rose_n$.  We define a sequence of spaces of curves whose intersection is $\Wh(\Gamma)$.  For all $i\geq 1$, we let $\W_i(\Gamma)$ be the directed graph defined below.
\begin{itemize}
    \item The set of vertices of $\W_i(\Gamma)$ is $\Wh(\Gamma)^{(0)};$
    \item  Edges are labeled by free conjugacy classes of elements of the set \newline $E_i:=(\Csep\cap\pi_1(\Gamma))\cup \big(\pi_1(\Gamma)\big)_{i+1}$, where $ \big(\pi_1(\Gamma)\big)_j \text{ is the } j$th term in the lower central series of $\pi_1(\Gamma)$;
    \item There is an edge labeled $[\alpha]\in E_i$ from $[v]$ to $[w]$ if and only if there are some $g\in[v], {h\in[w]}, {c\in\free_n}$ such that $gc\alpha\inv{c} = h.$
\end{itemize}
\end{definition}
Note that if any of the $\W_i(\Gamma)$ are connected, so is $\W_1(\Gamma)$.  $\W_1(\Gamma)$ is the ``integral homology version" of the separability complex. It is connected if and only if $\Hom_1\sep(\Gamma; \Z) = \Hom_1(\Gamma; \Z).$

\connectedhomology*

\begin{proof}
Suppose $\Hom^{\text{sep}}_1(\Gamma; \Z) = \Hom_1(\Gamma; \Z)$, and take an element $v$ of $\pi_1(\Gamma)$. Since $\Hom_1(\Gamma; \Z)$ is the abelianization of $\pi_1(\Gamma)$, $v$ is expressible as a product $\prod_{i\leq m} \alpha_i$ such that \newline ${\alpha_i\in (\Csep\cap\pi_1(\Gamma))\cup [\pi_1(\Gamma), \pi_1(\Gamma)]}.$  Thus there is a directed path from each vertex $[v]\in \W_1(\Gamma)$ to the trivial vertex corresponding to multiplying $v$ by the inverses of the $\alpha_i$ from $i = m$ to $i = 1$.  

Now suppose $\W_1(\Gamma)$ is connected.  Then there is an undirected path from every vertex $v\in \Gamma$ to the trivial loop.  Observe that for every edge labeled $[\beta]$ in $\W_1(\Gamma)$ connecting $[v]$ to $[w]$, there is an edge labeled $[\inv{\beta}]$ from $[w]$ to $[v]$.  Therefore, there is a directed path in $\W_1(\Gamma)$ from the trivial vertex to every vertex in the cover, and a directed path from the trivial loop to a vertex $[w]$ is an expression for $w$ as a product of elements of $(\Csep\cap\pi_1(\Gamma))\cup [\pi_1(\Gamma), \pi_1(\Gamma)]$.  Connectedness of $\W_1(\Gamma)$ implies there exists such a path for every vertex in $\W_1(\Gamma)$, so $\mathrm{H}_1^{\text{sep}}(\Gamma; \Z) = \mathrm{H}_1(\Gamma; \Z).$
\end{proof}

While we do not know whether $\W_1(\Gamma)$ is always connected, we do know that $\W_1(\Gamma)$ always has finitely many components. This follows from work of Boggi, Putman, and Salter in \cite{boggi2023generating}.

\begin{thm}[Boggi--Putman--Salter]
    Let $\pi:\widetilde{\Sigma}\rightarrow \Sigma$ be a finite branched covering between closed oriented surfaces. Then $\Hom_1^{\mathrm{pant}}(\widetilde{\Sigma}; \Q)=\Hom_1(\widetilde{\Sigma}; \Q)$.
\end{thm}

Here $\Hom_1^{\mathrm{pant}}(\widetilde{\Sigma}; \Q)$ is the subspace of $\Hom_1(\widetilde{\Sigma}; \Q)$ spanned by \[\mathcal{C}^{\mathrm{pant}}:=\{[w]\in\Hom_1(\widetilde{\Sigma};\Q)\mid \pi(w) \text { is contained in a pair of pants in } \Sigma\}.\]  As the authors state in Remark 1.5 of \cite{boggi2023generating}, this result also applies to punctured $\Sigma$ of finite type, as if $\widetilde{\Sigma}\rightarrow\Sigma$ is a finite branched cover of punctured surfaces, forgetting the punctures of $\Sigma$ and $\widetilde{\Sigma}$ yields a branched cover closed surfaces.  Forgetting the punctures kills exactly the homology classes surrounding punctures, and curves surrounding punctures are simple.  Since $\curves\pant$ contains all simple closed curves, it follows that $\Hom_1^{\mathrm{pant}}(\widetilde{\Sigma}; \Q)=\Hom_1(\widetilde{\Sigma}; \Q)$. We now show that $\W_1(\Gamma)$ has finitely many components.

\begin{prop}
    Let $\Gamma\rightarrow\rose_n$ be a finite regular cover. Then $\W_1(\Gamma)$ has finitely many components.
\end{prop}

\begin{proof}
    Since $\curves\pant\subset\Csep$, $\Hom_1\sep(\Gamma; \Z)$ is finite-index inside $\Hom_1(\Gamma; \Z)$. Let $T$ be a traversal of the left cosets of $\Hom_1\sep(\Gamma; \Z)$ in $\Hom_1(\Gamma; \Z)$. Then every element of $\Hom_1(\Gamma; \Z)$ can be written as a product $ts$, where $t\in T$ and $s\in\Hom_1\sep(\Gamma; \Z)$, so there is a path in the component of $\W_1(\Gamma)$ containing $[t]$ to $[ts].$ Thus there are at most $[\Hom_1(\Gamma; \Z): \Hom_1\sep(\Gamma; \Z)]$ components of $\W_1(\Gamma)$.
\end{proof}

\section{Appendix}\label{sec: appendix}
Here, we prove \Cref{prop: seplength-w} and \Cref{prop: dsepab} in the cases where $n=3$ and $n=2$. The structure of these proofs is not different from the case where $n>3$, but the values of $M_n(a, k, s)$, $m_n(a, s)$, and $O_n(a, s)$ are different. As before, we will prove \Cref{prop: seplength-w} assuming \Cref{clm:longest-bdry-arc} and \Cref{clm:upper-bound-on-boundary-D}, then we will prove the claims. Recall that we defined
\[  
    r_i(3):= \big((x_1^2x_2^ix_3)^2x_2^2\big)^2,
\]
and 
\[  
    r_i(2):=(x_1^ix_2)^2.
\]
For $n=3$, we let
\begin{equation}\label{equation:w-for-large-n-2}
    w(a, k, s, n):=\displaystyle\prod_{i=a}^{a+s-1}r_i(n)^k,
\end{equation}
and for $n=2$, we let 
\begin{equation}\label{equation:w-for-n-2-2}
    w(a, k, s, n):=\displaystyle\prod_{i=0}^{s-1} r_{a+2i}(n)^k.
\end{equation}

\seplengthw*

\begin{proof} 
For all $s$, there is a path in $\Wh(\rose_3)$ from $[1]$ to $[w(a, k, s, 3)]$ of length $s$ with edge labels $[r_i(3)^k]$, where $a\leq i\leq a+s-1$ (see \Cref{equation:w-for-large-n-2}). Similarly, there is a path from $[1]$ to $[w(a, k, s, 2)]$ of length $s$ with edge labels $[r_i(2)^k]$, where $i$ ranges over the set $\{a, a+2, a+4, \ldots, a+2s-2\}$ (see \Cref{equation:w-for-n-2-2}). Thus, for all $a, k, s,$ and $n=2$ or $n=3$, $\seplength{w(a, k, s, n)}\leq s$.  Since $w(a, k, 2, n)$ is inseparable for all $n$ by \Cref{lemma: only one ri}, we conclude that $\seplength{w(a, k, 2, n)}=2$. It remains to show that, for $s>2$, $\seplength{w(a, k, s, n)}\geq s$ for $n=3$ and $n=2$ and sufficiently large values of $a$ and $k$. Suppose toward a contradiction that, for all $a$ and $k$, $\seplength{w(a, k, s, n)}\leq s-1$, where $n=3$ or $n=2$. Let $D$ be a Van Kampen diagram for $w(a, k, s, n)$ over the presentation $\langle X_n\mid \Csep\rangle$ of area at most $s-1$.

\medskip

\noindent\textbf{Case 2: $\mathbf{n=3}$.} Observe $|r_i(3)|=16+4i$. A calculation shows that 
    \begin{equation*}
        \begin{split}
            |w(a, k, s, 3)| &=\displaystyle\sum_{i=a}^{a+s-1}|r_i(3)^k|\\
            &=k(14s+4as+2s^2).
        \end{split}
    \end{equation*}

    By the second part of \Cref{clm:upper-bound-on-boundary-D}, an upper bound for the sum of the length of boundary arcs of $D$ is given by the expression \[(s-1)M_3(a, k, s)+(s-3)m_3(a, s).\] Since $w(a, k, s, 3)$ is positive, the sum of the lengths of the boundary arcs of $D$ should be equal to $|w(a, k, s, 3)|$, but we will obtain a contradiction by showing that \[|w(a, k, s, 3)|>(s-1)M_3(a, k, s)+(s-3)m_3(a, s).\] By the proofs of \Cref{clm:longest-bdry-arc} and \Cref{clm:upper-bound-on-boundary-D}, we have
    \begin{equation*}
        (s-1)M_3(a, k, s)+(s-3)m_3(a, s) = (s-1)(k+2)(12+4a+4s)+(s-3)(18+6a+6s).
    \end{equation*}   
    Now, $(s-1)M_3(a, k, s)+(s-3)m_3(a, s)<|w(a,k,s,3)|$ if and only if 
    \begin{inequality}\label{inequality:contradiction-inequality3}
        k(14s+4as+2s^2)>(s-1)(k+2)(12+4a+4s)+(s-3)(18+6a+6s),
    \end{inequality}
    which happens if and only if
    \begin{inequality}\label{inequality:contradiction-subtract-kM3}
         k(14s+4as+2s^2)-k(s-1)(12+4a+4s)>2(s-1)(12+4a+4s)+(s-3)(18+6a+6s).
    \end{inequality}
    The left side of \Cref{inequality:contradiction-inequality3} simplifies to $2k(3s-s^2+6+2a),$ which is positive if $a>s^2.$ Thus, letting $a>s^2$ and 
    \begin{inequality}  
        k>\displaystyle\frac{(7s-13)(6+2a+2s)}{2(3s-s^2+6+2a)}
    \end{inequality}
    implies \Cref{inequality:contradiction-inequality3} holds, which in turn implies the area of a Van Kampen diagram for $w(a, k, s, 3)$ over $\langle X_n\mid \Csep\rangle$ is at least $s.$\newline

    \noindent\textbf{Case 3: $\mathbf{n=2}$.} Recall that
    \begin{equation*}
        w(a, k, s, 2)=\displaystyle\prod_{j=0}^{s-1} r_{a+2j}(2)^k
    \end{equation*} and
    \begin{equation*}
        r_i(2)=(x_1^ix_2)^2.
    \end{equation*}
    \noindent The length of $r_i(2)$ is $2i+2$, so, after a calculation, we obtain
    \begin{equation*}
        \begin{split}
            |w(a, k, s, 2)| &=  k\displaystyle\sum_{j=0}^{s-1} \big(2
            (a+2j)+2\big)\\
            &=k(2as+2s^2).
        \end{split}
    \end{equation*} 
    \noindent By the proofs of claims \Cref{clm:longest-bdry-arc} and \Cref{clm:upper-bound-on-boundary-D}, we have $M_2(a, k, s)=(k+2)(2a+4s)$ and $m_2(a, s)=3a+6s.$ Thus, by \Cref{clm:upper-bound-on-boundary-D}, an upper bound for the length of the boundary of $D$ is given by
    \begin{equation*}
        (s-1)M_2(a, k, s)+(s-3)m_2(a, s) = (s-1)(k+2)(2a+4s)+(s-3)(3a+6s).
    \end{equation*}
    \noindent As before, it suffices to show there are $a$ and $k$ large enough so that \[|w(a, k, s, 2)|>(s-1)M_2(a, k, s)+(s-3)m_2(a, s).\] This is true if and only if 
    \begin{inequality}\label{inequality:contradiction-inequality2}
        k(2as+2s^2)>(s-1)(k+2)(2a+4s)+(s-3)(3a+6s),
    \end{inequality}
    which is true if and only if
    \begin{inequality}\label{inequality:contradiction-subtract-kM2}
        k(2as+2s^2)-k(s-1)(2a+4s)>2(s-1)(2a+4s)+(s-3)(3a+6s).
    \end{inequality}
    The left side of \Cref{inequality:contradiction-subtract-kM2} simplifies to $k(4s+2a-2s^2)$, which is positive if $a>s^2.$ Letting $a>s^2$ and 
    \begin{inequality}
        k>\displaystyle\frac{(7s-13)(a+2s)}{4s+2a-2s^2}
    \end{inequality} ensures that \Cref{inequality:contradiction-inequality2} holds.
\end{proof}

We now prove \Cref{clm:longest-bdry-arc} for $n=3$ and $n=2$.

\longestBoundaryArc*

\begin{proof}
    \noindent\textbf{Case 2: $\mathbf{n=3}$.} Recall that 
        \[
            r_i(3)=\big((x_1^2x_2^ix_n)^2x_2^2\big)^2
        \]
        and 
        \[
            w(a, k, s, 3) = \displaystyle\prod_{i=a}^{a+s-1} r_i(3)^k.
        \] As in Case 1, the length of each boundary arc of $B$ is bounded above by the length of the longest subword of $w(a, k, s, 3)$ containing only one $r_i$ as a subword, which is less than
        \begin{align*}
            \displaystyle\max_{a\leq i\leq a+s-3}\{|r_ir_{i+1}^kr_{i+2}|, |r_{a+s-2}r_{a+s-1}^kr_a|, |r_{a+s-1}r_a^kr_{a+1}|\}.
        \end{align*}
        Since
        \begin{align*}
            |r_i(3)| = |\big((x_1x_2^ix_3)^2x_2^2\big)^2|=16+4i,
        \end{align*}
        the word $r_ir_{i+1}^kr_{i+2}$ is shorter than $r_jr_{j+1}^kr_{j+2}$ when $i<j$, so the longest boundary arc of $D$ along $w(a, k, s, 3)$ has length less than
        \[
            \displaystyle\max\{|r_{a+s-2}r_{a+s-1}^kr_a|, |r_{a+s-3}r_{a+s-2}^kr_{a+s-1}|, |r_{a+s-1}r_a^kr_{a+1}|\}.
        \]
        We compute the lengths of the three possible maxima:
        \begin{align*}
            |r_{a+s-3}r_{a+s-2}^kr_{a+s-1}| &=|r_{a+s-3}| +k|r_{a+s-2}|+|r_{a+s-1}|\\
            &=16+8a+8s+k(8+4a+4s);
        \end{align*}
        \begin{align*}
            |r_{a+s-2}r_{a+s-1}^kr_a|&=|r_{a+s-2}| +k|r_{a+s-1}|+|r_a|\\
            &=24+8a+4s+k(12+4a+4s);    
        \end{align*}
        \begin{align*}
            |r_{a+s-1}r_a^kr_{a+1}| &=|r_{a+s-1}| +k|r_a|+|r_{a+1}|\\
            &=32+8a+4s+k(16+4a).
        \end{align*}
        A calculation similar to that in Case 1 shows that, since $s>2$ and $k>s-2,$ the longest of these is $r_{a+s-2}r_{a+s-1}^kr_a.$ Thus, for $s>2$, $k>s-2$, and $n=3$, the longest boundary arc of $B$ has length less than \[24+8a+4s+k(12+4a+4s).\]  Observe 
        \[
            24+8a+4s+k(12+4a+4s)<(k+2)(12+4a+4s)=:M_3(a, k, s).
        \]

        \noindent\textbf{Case 3: $\textbf{n=2}$.} The proof mirrors the proofs of the other two cases. Recall that
        \[
            r_i(2)=(x_1^ix_2)^2
        \] and
        \[
            w(a, k, s, 2)=r_a^kr_{a+2}^kr_{a+4}^k\ldots r_{a+2(s-1)}.
        \]
        We calculate
        \[
        |r_i(2)|=|(x_1^ix_2)^2|=2+2i.
        \]
        The longest subword of $w(a, k, s, 2)$ with only one of the $\{r_{a+2j}\}_{j=0}^{s-1}$ as a subword has length less than 
        \begin{align*}
            \displaystyle\max_{a\leq i\leq a+2(s-3)} \{|r_ir_{i+2}^kr_{i+4}|, |r_{a+2(s-2)}r_{a+2(s-1)}^kr_a|, |r_{a+2(s-1)}r_a^kr_{a+2}|\}
        \end{align*}
        As before, $r_ir_{i+2}^kr_{i+4}$ is shorter than $r_jr_{j+2}r_{j+4}$ when $i<j$. Thus the longest boundary arc of $B$ has length less than
        \[
            \displaystyle\max \{|r_{a+2(s-3)}r_{a+2(s-2)}^kr_{a+2(s-1)}|, |r_{a+2(s-2)}r_{a+2(s-1)}^kr_a|, |r_{a+2(s-1)}r_a^kr_{a+2}|\}
        \] A calculation shows that, for for $s>2, k>s-2$, $|r_{a+2(s-2)}r_{a+2(s-1)}^kr_a|$ is largest. We conclude
        \[  
            4a+4s-4+k(2a+4s-2)
        \] is larger than the length of the longest boundary arc of $B$.  Note that \[
        4a+4s-4+k(2a+4s-2)<(k+2)(2a+4s)=:M_2(a, k, s).
        \]
\end{proof}

We now prove \Cref{clm:upper-bound-on-boundary-D} in the cases $n=3$ and $n=2.$
\upperBoundD*

    \begin{proof}\noindent\textbf{Case 2: $\mathbf{n=3}$.}
    Recall that
    \begin{equation*}
        r_i(3)=\big((x_1^2x_2^ix_3)^2x_2^2\big)^2
    \end{equation*} and 
    \begin{equation*}
        w(a, k, s, 3)= \displaystyle\prod_{i=a}^{a+s-1} r_i(3)^k.
    \end{equation*}
    For the first part, note that the longest subwords of $w(a, k, s, 3)$ with no $r_i$ as a subword are of length less than
    \begin{equation*}
        m_3(a, s):=|\big((x_1^2x_2^{a+s-1}x_3)^2x_2^2\big)^3| = 18+6a+6s.
    \end{equation*}

    For the second part, observe that the only instance of $r_i(3)$ as a cyclic subword of $w(a, k, s, 3)$ is in the term $r_i(3)^k$. Note that $|r_i(3)^k|<M_3(a, k, s)$. By \Cref{lemma: only one ri}, each region of $D$ contains at most one $r_i$ as a subword of its boundary relation. \Cref{clm:longest-bdry-arc} implies that every boundary arc of $D$ containing an $r_i$ as a subword has length less than $M_3(a, k, s)$. Hence if $R$ is a region of $D$ and $i$ is an integer, the total length of the boundary arcs of $R$ which contain a subword $r_i$ is less than $M_3(a, k, s)$. By assumption, $D$ has at most $s-1$ regions, so the total length of the boundary arcs of $D$ containing any $r_i$ as a subword is at most $(s-1)M_3(a, k, s).$  The remaining boundary arcs of $D$ do not contain a subword $r_i$ for any $i$, and each of these arcs has length less than $m_3(a, s)$. Observe that $m_3(a, s)<M_3(a, k, s)$, so we get a higher upper bound for the sum of the lengths of the boundary arcs of $D$ by assuming each region of $D$ has a boundary arc with some $r_i$ subword. Since there are at most $2(s-2)$ boundary arcs of $D$, the sum of the lengths of the boundary arcs of $D$ is at most \[(s-1)M_3(a, k, s)+(s-3)m_3(a, s).\] 
    
    \noindent\textbf{Case 3: $\textbf{n=2}$.} Recall that
        \[
            r_i(2)=(x_1^ix_2)^2
        \] and
        \[
            w(a, k, s, 2)=\prod_{j=0}^{s-1}r_{a+2j}(2)^k.
        \] One can check that the longest subword of $w(a, k, s, 2)$ with none of the $r_i(2)$ as a subword is of length less than \[|(x_1^{a+2(s-1)}x_2)^3|=3a+6s-3<3a+6s=:m_2(a, s).\]
        
        For every $i\neq a$, the only instance of $r_i(2)$ as a cyclic subword of $w(a, k, s, 2)$ is in the term $r_i(2)^k$, and the only instance of $r_a(2)$ as a cyclic subword of $w(a, k, s, 2)$ is within $x_1^{a+2(s-1)}x_2r_a(2)^k$. Observe that $|r_i(2)^k|<M_2(a, k, s)$ and $|x_1^{a+2(s-1)}x_2r_a(2)^k|<M_2(a, k, s)$. Thus if a region $R$ of $D$ has a boundary arc containing $r_i$ as a subword for $i\neq a$, $R$ intersects this boundary within the section labeled $r_i^k$ (within $x_1^{a+2(s-1)}x_2r_a^k$ for $i=a$). By \Cref{lemma: only one ri}, each region of $D$ contains at most one $r_i$ as a subword of its boundary relation. \Cref{clm:longest-bdry-arc} implies that every boundary arc of $D$ containing an $r_i$ as a subword has length less than $M_2(a, k, s)$. Hence if $R$ is a region of $D$ and $i$ is an integer, the total length of the boundary arcs of $R$ which contain a subword $r_i$ is less than $M_2(a, k, s)$. By assumption, $D$ has at most $s-1$ regions, so the total length of the boundary arcs of $D$ containing an $r_i$ as a subword is at most $(s-1)M_2(a, k, s).$  The remaining boundary arcs of $D$ do not contain a subword $r_i$ for any $i$. Since $D$ has at most $2(s-2)$ boundary arcs, assuming each of the $s-1$ regions of $D$ contains a boundary arc with an $r_i$ subword, there are at most $(s-3)$ boundary arcs of $D$ with no $r_i$ subword. Again, it gives us a higher upper bound to assume that each region has a boundary arc with an $r_i$ subword as $m_2(a, s)<M_2(a, k, s)$. Thus, the sum of the lengths of boundary arcs of $D$ is less than $(s-1)M_2(a, k, s)+(s-3)m_2(a, s)$. This completes the proof of \Cref{clm:upper-bound-on-boundary-D}.
    \end{proof}  
This concludes the proof of \Cref{prop: seplength-w} for $n=3$ and $n=2.$  We now prove \Cref{prop: dsepab} for $n=3$ and $n=2$.

\dsepab*

\begin{proof}
    \noindent\textbf{Case 2: $\mathbf{n=3}$.} The structure of the proof is the same as it was for $n>3$. Let $b=a+s$. We assume toward a contradiction that $A$ is an annular diagram of area at most $2s-1$ with boundaries $w(a, k, s, 3)$ and $w(b, k, s, 3)$. Let $O_3(a, s)$ be the length of the longest subword shared by $w(a, k, s, 3)$ and $w(b, k, s, 3)$. Recall from Case 1 that $|w(a, k, s, 3)|+|w(b, k, s, 3)|-(2s-1)O_3(a, s)$ is a lower bound for the length of the boundary of $A$, and an upper bound for the length of the boundary of $A$ is given by $(2s-1)M_3(b, k, s)+(2s-1)m_3(b, s)+(2s-1)O_3(a, s)$. The contradiction we find is that there are $a$ and $k$ large enough so that 
    \begin{inequality}\label{inequality:contradiction-annular-n-3}
        |w(a, k, s, 3)|+|w(b, k, s, 3)|-2(2s-1)O_3(a, s)>(2s-1)M_3(b, k, s) +(2s-1)m_3(b, s).
    \end{inequality} A computation shows that
    \begin{equation*}
        \begin{split}
            |w(a, k, s, 3)|+|w(b, k, s, 3)|&=k(14s+4as+2s^2)+k(14s+4s(a+s)+2s^2)\\
            &=k(28s+8as+8s^2).
        \end{split}
    \end{equation*}
    We want to compute the length of the longest subword shared by 
    \begin{equation*}
        w(a, k, s, 3)=\Big((x_1^2x_2^ax_3)^2x_2^2\Big)^k\Big((x_1^2x_2^{a+1}x_3)^2x_2^2\Big)^k\ldots\Big((x_1^2x_2^{a+s-1}x_3)^2x_2^2\Big)^k
    \end{equation*} and 
    \begin{equation*}
        w(b, k, s, 3)=\Big((x_1^2x_2^{a+s}x_3)^2x_2^2\Big)^{2k}\Big((x_1^2x_2^{a+s+1}x_3)^2x_2^2\Big)^{2k}\ldots\Big((x_1^2x_2^{a+2s-1}x_3)^2x_2^2\Big)^{2k}.
    \end{equation*}
    By inspection, the longest subword of both $w(a, k, s, 3)$ and $w(b, k, s, 3)$ is $x_2^{a+s-1}x_3x_2^2x_1^2x_2^{a+s-1},$ which has length
    \begin{equation*}
        |x_2^{a+s-1}x_3x_2^2x_1^2x_2^{a+s-1}|=2a+2s+3.
    \end{equation*} Thus $O_3(a, s)=2a+2s+3$. Thus a lower bound for the length of the boundary of $A$ is 
    \begin{equation*}
        |w(a, k, s, n)|+|w(b, k, s, n)|-(2s-1)O_3(a, s)= k(28s+8as+8s^2)-(2s-1)(2a+2s+3).
    \end{equation*}
We now compute an upper bound for the length of the boundary of $A$. By the argument in Case 1, an upper bound for the length of the boundary of $A$ is given by \[(2s-1)M_3(b, k, s)+(2s-1)m_3(b, s)+(2s-1)O_3(a, s).\]  In \Cref{clm:longest-bdry-arc} and \Cref{clm:upper-bound-on-boundary-D}, respectively, we computed 
\begin{equation*}
    M_3(b, k, s)=(k+2)(12+4b+4s)
\end{equation*} and 
\begin{equation*}
    m_3(b, s)=18+6b+6s.
\end{equation*} Thus an upper bound for the length of the boundary of $A$ is 
\begin{equation*}
    (2s-1)(k+2)(12+4a+8s)+(2s-1)(18+6a+12s)+(2s-1)(2a+2s+3).
\end{equation*} To obtain a contradiction, we show that \Cref{inequality:contradiction-annular-n-3} holds for sufficiently large $a$ and $k$. \Cref{inequality:contradiction-annular-n-3} is equivalent to 
\begin{inequality}\label{inequality:constants-contradiction-annular-n-3}
    \begin{split}
        k(28s+8as+8s^2)-&2(2s-1)(2a+2s+3)\\&>(2s-1)(k+2)(12+4a+8s)+(2s-1)(18+6a+12s),
    \end{split}
\end{inequality} which simplifies to
\begin{inequality}\label{inequality:constants-contradiction-annular-n-3-last-step}
    k(12s-8s^2+12+4a)>(2s-1)(48+18a+32s).
\end{inequality} When $a>2s^2$, $12s-8s^2+12+4a$ is positive.  So, letting $a>2s^2$ and 
\begin{equation*}
    k>\frac{(2s-1)(48+18a+32s)}{12s-8s^2+12+4a}
\end{equation*} gives us our contradiction. We conclude $\dsep{[w(a, k, s, 3)]}{[w(b, k, s, 3)]}=2s$ for sufficiently high $a$ and $k.$

\medskip

\noindent\textbf{Case 3: $\mathbf{n=2}$.} The proof is the same as before. Let $b=a+2s$ and let $A$ be an annular diagram of area at most $2s-1$ with boundaries labeled by $w(a, k, s, 2)$ and $w(b, k, s, 2).$ It suffices to show 
\begin{inequality}\label{inequality:contradiction-annular-n-2}
    |w(a, k, s, 2)|+|w(b, k, s, 2)|-2(2s-1)O_2(a, s)>(2s-1)M_2(b, k, s)+(2s-1)m_2(b, s).
\end{inequality} Recall from \Cref{clm:longest-bdry-arc} and \Cref{clm:upper-bound-on-boundary-D} that $M_2(b, k, s)=(k+2)(2b+4s)$ and $m_2(b, s)=3b+6s.$ \Cref{prop: seplength-w} shows that $|w(a, k, s, 2)|=k(2as+2s^2)$, so we compute
\begin{equation*}
    |w(a, k, s, 2)|+|w(b, k, s, 2)|=k(4as+8s^2).
\end{equation*}
We want to compute $O_2(a, s)$. Recall that
\begin{equation*}
    w(a, k, s, 2)=(x_1^ax_2)^{2k}(x_1^{a+2}x_2)^{2k}(x_1^{a+4}x_2)^{2k}\ldots (x_1^{a+2s-2}x_2)^{2k}
\end{equation*} and 
\begin{equation*}
    w(b, k, s, 2)=(x_1^{a+2s}x_2)^{2k}(x_1^{a+2s+2}x_2)^{2k}(x_1^{a+2s+4}x_2)^{2k}\ldots (x_1^{a+4s-2}x_2)^{2k}.
\end{equation*} By inspection, the longest shared subword of $w(a, k, s, 2)$ and $w(b, k, s, 2)$ is $x_1^{a+2s-2}x_2x_1^{a+2s-2},$ which has length $O_2(a, s)=2a+4s-3.$ \Cref{inequality:contradiction-annular-n-2} is equivalent to 
\begin{inequality}\label{inequality:constants-contradiction-annular-n-2}
    k(4as+8s^2)-2(2s-1)(2a+4s-3)>(2s-1)(k+2)(2a+8s)+(2s-1)(3a+12s).
\end{inequality} After some simplification of \Cref{inequality:constants-contradiction-annular-n-2}, we obtain
\begin{inequality}\label{inequality:contradiction-annular-last-step-n-2}
    k(2a+8s-8s^2)>(2s-1)(11a+36s-6).
\end{inequality} For $a>4s^2$, the left side of \Cref{inequality:contradiction-annular-last-step-n-2} is positive. Thus, taking $a>4s^2$ and 
\begin{inequality}
    k>\displaystyle\frac{(2s-1)(11a+36s-6)}{2a+8s-8s^2}
\end{inequality} gives us the desired contradiction. We conclude there are $a$ and $k$ large enough so that $\dsep{[w(a, k, s, 2)]}{[w(b, k, s, 2)]}=2s.$ 
\end{proof}

\bibliographystyle{alpha}
\bibliography{main}

@article{boggi2023generating,
  title={Generating the homology of covers of surfaces},
  author={Boggi, Marco and Putman, Andrew and Salter, Nick},
  journal={arXiv preprint arXiv:2305.13109},
  year={2023}
}

@article{malestein2019simple,
  author = {Justin Malestein and Andrew Putman},
    title = {{Simple closed curves, finite covers of surfaces, and power subgroups of $\operatorname{Out}(F_{n})$}},
    volume = {168},
    journal = {Duke Mathematical Journal},
    number = {14},
    publisher = {Duke University Press},
    pages = {2701--2726},
    year = {2019},
}

@article{farb2016finite,
title = {Finite covers of graphs, their primitive homology, and representation theory},
author = {Farb, Benson and Hensel, Sebastian},
journal = {New York Journal of Math},
year = {2016},
pages = {1365--1391},
volume = {22}
}

@book{lyndon1977combinatorial,
  title={Combinatorial group theory},
  author={Lyndon, Roger C and Schupp, Paul E},
  volume={188},
  year={1977},
  publisher={Springer}
}

@article{stallings1983topology,
  title={Topology of finite graphs},
  author={Stallings, John R},
  journal={Inventiones mathematicae},
  volume={71},
  number={3},
  pages={551--565},
  year={1983},
  publisher={Springer-Verlag Berlin/Heidelberg}
}

@article{stallings1999whitehead,
  title={Whitehead graphs on handlebodies},
  author={Stallings, John R},
  journal={Geometric group theory down under (Canberra, 1996)},
  pages={317--330},
  year={1999},
  publisher={Citeseer}
}

@misc{Kent2012,
    TITLE = {Homology generated by lifts of simple curves},
    AUTHOR = {Kent, Autumn Exum},
    HOWPUBLISHED = {MathOverflow},
    NOTE = {URL:https://mathoverflow.net/q/86938 (version: 2012-01-29)},
    EPRINT = {https://mathoverflow.net/q/86938},
    YEAR = {2012},
    URL = {https://mathoverflow.net/q/86938}
}

@article{whitehead1936certain,
  title={On certain sets of elements in a free group},
  author={Whitehead, John Henry Constantine},
  journal={Proceedings of the London Mathematical Society},
  volume={2},
  number={1},
  pages={48--56},
  year={1936},
  publisher={Wiley Online Library}
}

@article{Koberda_2016,
    doi = {10.1007/s00222-016-0652-x},
    url = {https://doi.org/10.1007%2Fs00222-016-0652-x}, 
    year = 2016,
    month = {feb},
    publisher = {Springer Science and Business Media {LLC}},
    volume = {206},
    number = {2},
    pages = {269--292},
    author = {Thomas Koberda and Ramanujan Santharoubane},
    title = {Quotients of surface groups and homology of finite covers via quantum representations},
    journal = {Inventiones mathematicae}
}

@article{lee2020graph,
author = {Lee, Destine and Rosenblum-Sellers, Iris and Safin, Jakwanul and Tenie, Anda},
year = {2021},
month = {06},
pages = {1-8},
title = {Graph Coverings and (Im)primitive Homology: {E}xamples of Low Degree},
volume = {32},
journal = {Experimental Mathematics},
doi = {10.1080/10586458.2021.1926013}
}

@article{kent2009trees,
    title = {Trees and mapping class groups},
    author = {Autumn E. Kent and Christopher J. Leininger and Saul Schleimer},
    pages = {1--21},
    volume = {2009},
    number = {637},
    journal = {Journal für die reine und angewandte Mathematik},
    year = {2009}
}

@misc{Kent2023,
  author = "Kent, Autumn",
  date = "2023",
  howpublished = "personal communication"
}

@article{boggi2006profinite,
  title={Profinite Teichm{\"u}ller theory},
  author={Boggi, Marco},
  journal={Mathematische Nachrichten},
  volume={279},
  number={9-10},
  pages={953--987},
  year={2006},
  publisher={Wiley Online Library}
}

@article{handel2013free,
  title={The free splitting complex of a free group, {I} {H}yperbolicity},
  author={Handel, Michael and Mosher, Lee},
  journal={Geometry \& Topology},
  volume={17},
  number={3},
  pages={1581--1670},
  year={2013},
  publisher={Mathematical Sciences Publishers}
}

@article{hatcher1998complex,
  title={The complex of free factors of a free group},
  author={Hatcher, Allen and Vogtmann, Karen},
  journal={Quarterly Journal of Mathematics},
  volume={49},
  number={196},
  pages={459--468},
  year={1998},
  publisher={Oxford: Clarendon Press, 1930-}
}

@article{bestvina2014hyperbolicity,
  title={Hyperbolicity of the complex of free factors},
  author={Bestvina, Mladen and Feighn, Mark},
  journal={Advances in Mathematics},
  volume={256},
  pages={104--155},
  year={2014},
  publisher={Elsevier}
}

@article{kapovich2014hyperbolicity,
  title={On hyperbolicity of free splitting and free factor complexes},
  author={Kapovich, Ilya and Rafi, Kasra},
  journal={Groups, geometry, and dynamics},
  volume={8},
  number={2},
  pages={391--414},
  year={2014}
}

@article{hilion2017hyperbolicity,
  title={The hyperbolicity of the sphere complex via surgery paths},
  author={Hilion, Arnaud and Horbez, Camille},
  journal={Journal f{\"u}r die reine und angewandte Mathematik},
  volume={2017},
  number={730},
  pages={135--161},
  year={2017},
  publisher={De Gruyter}
}

@article{Kent2016,
url = {https://doi.org/10.1515/crelle-2014-0023},
title = {Congruence kernels around affine curves},
author = {Kent, Autumn Exum},
pages = {1--20},
volume = {2016},
number = {713},
journal = {Journal für die reine und angewandte Mathematik (Crelles Journal)},
doi = {doi:10.1515/crelle-2014-0023},
year = {2016},
}

@unpublished{Boggi-announced,
    author = "Boggi, Marco",
    title = "The congruence subgroup property for the hyperelliptic modular group: the closed surface case",
    note = "Announced"
}

@article{klukowski2023simple,
  title={Simple closed curves, non-kernel homology and Magnus embedding},
  author={Klukowski, Adam},
  journal={arXiv preprint arXiv:2304.13196},
  year={2023}
}

@book{magnus2004combinatorial,
  title={Combinatorial group theory: {P}resentations of groups in terms of generators and relations},
  author={Magnus, Wilhelm and Karrass, Abraham and Solitar, Donald},
  year={2004},
  publisher={Courier Corporation}
}

@book{farb2006problems,
  title={Problems on mapping class groups and related topics},
  author={Farb, Benson},
  volume={74},
  year={2006},
  publisher={American Mathematical Soc.}
}

@article{culler1986moduli,
  title={Moduli of graphs and automorphisms of free groups},
  author={Culler, Marc and Vogtmann, Karen},
  journal={Inventiones mathematicae},
  volume={84},
  number={1},
  pages={91--119},
  year={1986},
  publisher={Springer-Verlag Berlin/Heidelberg}
}

@article{sabalka2014geometry,
  title={On the geometry of the edge splitting complex},
  author={Sabalka, Lucas and Savchuk, Dmytro},
  journal={Groups, Geometry, and Dynamics},
  volume={8},
  number={2},
  pages={565--598},
  year={2014}
}

@article{bardakov2003palindromic,
title = {On the palindromic and primitive widths of a free group},
journal = {Journal of Algebra},
volume = {285},
number = {2},
pages = {574-585},
year = {2005},
issn = {0021-8693},
doi = {https://doi.org/10.1016/j.jalgebra.2004.11.003},
url = {https://www.sciencedirect.com/science/article/pii/S002186930400612X},
author = {Valery Bardakov and Vladimir Shpilrain and Vladimir Tolstykh}
}

@book{epstein1992word,
  title={Word processing in groups},
  author={Epstein, David BA},
  year={1992},
  publisher={AK Peters/CRC Press}
}

@article{Chu25infinite,
    author = {Chu, Sabine and Domat, George and Gao, Christine and Prasanna, Ananya and Wright, Alex},
    title = {The infinite dimensional geometry of conjugation invariant generating sets},
    journal = {Preprint},
    year = {2025}
}

@article{brandenbursky2016cancellation,
  title={The cancellation norm and the geometry of bi-invariant word metrics},
  author={Brandenbursky, Michael and Gal, {\'S}wiatos{\l}aw R and K{\k{e}}dra, Jarek and Marcinkowski, Micha{\l}},
  journal={Glasgow Mathematical Journal},
  volume={58},
  number={1},
  pages={153--176},
  year={2016},
  publisher={Cambridge University Press}
}

@article{brandenbursky2019aut,
  title={Aut-invariant norms and Aut-invariant quasimorphisms on free and surface groups.},
  author={Brandenbursky, Michael and Marcinkowski, Micha{\l}},
  journal={Commentarii Mathematici Helvetici},
  volume={94},
  number={4},
  year={2019}
}

@misc{Bering2026,
  author = "Bering, Edgar",
  date = "2026",
  howpublished = "personal communication"
}

\end{document}